\documentclass[
final
, nomarks
]{dmtcs-episciences}

\usepackage[utf8]{inputenc}
\usepackage{subfigure}

\usepackage{tikz}
\usetikzlibrary{arrows,calc}

\tikzstyle{arc}=[->,shorten <=3pt, shorten >=3pt,
                 >=stealth, line width=1.1pt]
\tikzstyle{edge}=[shorten <=1pt, shorten >=1pt,
                  >=stealth, line width=0.8pt]
\tikzstyle{curveEdge}=[shorten <=1pt, shorten >=1pt, >=stealth, bend right=30,
                       bend left=30, line width=1pt]
\tikzstyle{vertex}=[circle, fill=white, draw,
                    minimum size=5pt,
                    inner sep=0pt, outer sep=0pt]
\tikzstyle{node}=[circle, fill=white, draw,
                    minimum size=5pt,
                    inner sep=1pt, outer sep=0pt]

\tikzstyle{blackV}=[circle, fill=black,
                    minimum size=5pt,
                    inner sep=0pt, outer sep=0pt]

\usepackage{amsmath}
\usepackage{amssymb}

\usepackage{amsthm}

\usepackage[capitalize]{cleveref}

\usepackage{multicol}
\usepackage{scalerel}
  \def\Spider{\scalerel*{\includegraphics{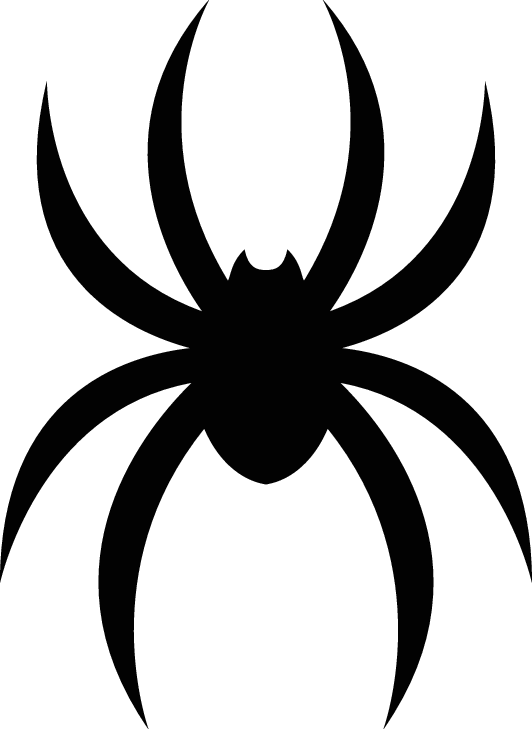}}{X}}
\newcommand{\spider}{\mathbin{\Spider{}}}

\newtheorem{theorem}{Theorem}
\newtheorem{lemma}[theorem]{Lemma}
\newtheorem{corollary}[theorem]{Corollary}
\newtheorem{proposition}[theorem]{Proposition}
\newtheorem{remark}[theorem]{Remark}
\newtheorem{problem}{Problem}

\author[Fernando Esteban Contreras-Mendoza and C\'esar
  Hern\'andez-Cruz]{Fernando Esteban Contreras-Mendoza \and C\'esar
  Hern\'andez-Cruz} \title[Polarity variants on cograph
  superclasses]{$2$-polarity and algorithmic aspects of polarity variants on
  cograph superclasses\thanks{Supported by grants UNAM-PAPIIT IA101423,
  SEP-CONACYT A1-S-8397, and CONACYT FORDECYT-PRONACES/39570/2020}}
\affiliation{
  Facultad de Ciencias, Universidad Nacional Aut\'onoma de M\'exico, M\'exico}

\begin{document}
\publicationdata
{vol. 26:3}
{2024}
{11}
{10.46298/dmtcs.11479}
{2023-06-19; 2023-06-19; 2024-04-30}
{2024-09-09}
\maketitle
\begin{abstract}
    A graph $G$ is said to be an $(s, k)$-polar graph if its vertex set admits a
    partition $(A, B)$ such that $A$ and $B$ induce, respectively, a complete
    $s$-partite graph and the disjoint union of at most $k$ complete graphs.
    Polar graphs and monopolar graphs are defined as $(\infty, \infty)$- and
    $(1, \infty)$-polar graphs, respectively, and unipolar graphs are those
    graphs with a polar partition $(A, B)$ such that $A$ is a clique.
  
    The problems of deciding whether an arbitrary graph is a polar graph or a
    monopolar graph are known to be NP-complete. In contrast, deciding whether a
    graph is a unipolar graph can be done in polynomial time. In this work we
    prove that the three previous problems can be solved in linear time on the
    classes of $P_4$-sparse and $P_4$-extendible graphs, generalizing analogous
    results previously known for cographs.
  
    Additionally, we provide finite forbidden subgraph characterizations for
    $(2,2)$-polar graphs on $P_4$-sparse and $P_4$-extendible graphs, also
    generalizing analogous results recently obtained for the class of cographs.
\end{abstract}

\section{Introduction}

All graphs in this paper are finite and simple; for basic terminology not
defined here we refer the reader to the beautiful book of Bondy and Murty
\cite{bondy2008}. For graphs $G$ and $H$, we denote that $H$ is an induced
subgraph of $G$ by $H \le G$. Given a family of graphs $\mathcal{H}$, we say
that $G$ is \textit{$\mathcal{H}$-free} if $G$ does not have induced subgraphs
isomorphic to any graph $H \in \mathcal{H}$; accordingly, we say that $G$ is an
\textit{$H$-free graph} if it is $\{H\}$-free. A property of graphs is
\textit{hereditary} if it is closed under taking induced subgraphs. Given a
hereditary property $\mathcal{P}$ of graphs, a \textit{minimal
$\mathcal{P}$-obstruction} is a graph $G$ that does not have the property
$\mathcal{P}$ but such that any vertex-deleted subgraph of $G$ does.

A \textit{$k$-cluster} is the disjoint union of at most $k$ complete graphs; a
\textit{cluster} is a $k$-cluster for some positive integer $k$. It is easy to
verify that $k$-clusters coincide with $\{\overline{K_{k+1}}, P_3\}$-free
graphs, while clusters are precisely $P_3$-free graphs. A \textit{complete
$k$-partite graph} is the complement of a $k$-cluster, or equivalently, a
$\{K_{k+1}, \overline{P_3}\}$-free graph; a \textit{complete multipartite graph}
is the complement of a cluster, i.e., a $\overline{P_3}$-free graph. An
\textit{$(s,k)$-polar partition} of a graph $G$ is a partition of $V_G$ in two
possible empty sets $A$ and $B$ such that $G[A]$ is a complete $s$-partite
graph, and $G[B]$ is a $k$-cluster. If $G$ admits an $(s,k)$-polar partition we
say that it is an \textit{$(s,k)$-polar graph}. A $(k,k)$-polar partition is
simply referred as a \textit{$k$-polar partition}, and a graph which admits such
partition is a \textit{$k$-polar graph}. A $1$-polar graph is commonly called a
\textit{split graph}; in the classic article of Foldes and Hammer
\cite{foldesGTC1977}, split graphs were characterized to be
$\{2K_2,C_4,C_5\}$-free graphs. If we replace $s$ or $k$ by $\infty$, it means
that the number of components of $\overline{G[A]}$ or $G[B]$, respectively, is
unbounded. An $(\infty, \infty)$-polar partition of a graph is simply called a
\textit{polar partition}, and a graph with such partition is a \textit{polar
graph}. A graph with polar partition $(A, B)$ such that $A$ is an independent
set (respectively, a clique) is called a \textit{monopolar graph} (resp. a
\textit{unipolar graph}). Naturally, the polar partitions associated to
monopolar and unipolar graphs are referred as monopolar and unipolar partitions,
respectively.

Graphs without induced paths on four vertices are known as \textit{cographs}. A
graph such that any set of five vertices induces at most one $P_4$ is called a
\textit{$P_4$-sparse graph}, and a graph such that, for any vertex subset $W$
inducing a $P_4$ there exists at most one vertex $v \notin W$ belonging to a
$P_4$ which shares vertices with $W$, is a \textit{$P_4$-extendible graph}.

In \cite{contrerasCogGen1} it was proved that any hereditary property of graphs
restricted to $P_4$-sparse graphs and $P_4$-extendible graphs can be
characterized by a finite set of forbidden induced subgraphs. In the same paper,
such characterizations for the properties of having a polar partition, a
monopolar partition, a unipolar partition, and an $(s,1)$-polar partition for
any fixed positive integer $s$ were given. In this paper, we continue with the
work started in \cite{contrerasCogGen1}, establishing linear-time algorithms to
find maximum subgraphs associated with properties related to polarity in
$P_4$-sparse and $P_4$-extendible graphs, and giving forbidden subgraph
characterizations for $P_4$-sparse and $P_4$-extendible graph which admit a
$2$-polar partition. For the sake of length we invite the reader to read
\cite{contrerasCogGen1} where a discussion on the relevance of the topic of this
paper can be found.

The rest of the paper is organized as follows. \Cref{sec:cogGen} is devoted to a
brief introduction of $P_4$-sparse and $P_4$-extendible graphs. In
\Cref{sec:2obsmin} we give complete lists of minimal $P_4$-sparse and
$P_4$-extendible $2$-polar obstructions, while in \Cref{sec:maxSubgraphs} we
provide algorithms for finding maximum polar, unipolar, and monopolar subgraphs
in both $P_4$-sparse and $P_4$-extendible graphs. Conclusions and some open
problems are given in \Cref{sec:conc}.


\section{Cograph generalizations}
\label{sec:cogGen}

We use $G + H$ to denote the disjoint union of the graphs $G$ and $H$;
accordingly, we denote by $nG$ the disjoint union of $n$ copies of the graph
$G$. The join of $G$ and $H$, defined as the graph $\overline{\overline{G} +
\overline{H}}$, will be denoted by $G \oplus H$. We say that two vertex subsets
are \textit{completely adjacent} if every vertex of one of them is adjacent to
any vertex of the other. Similarly, if no vertex of one of them is adjacent to a
vertex of the other, we say that those vertex subsets are \textit{completely
nonadjacent}. The following proposition include some characterizations for
cographs which are particularly relevant for this work.

\begin{theorem}\cite{corneilDAM3}
\label{theo:cogChar}
    Let $G$ be a graph. The following statements are equivalent.
    \begin{enumerate}
    \item $G$ is a $P_4$-free graph (i.e. a cograph).

    \item $G$ can be constructed from trivial graphs by means of join and
        disjoint union operations.

    \item For any nontrivial induced subgraph $H$ of $G$, either $H$ or
        $\overline{H}$ is disconnected.
    \end{enumerate}
\end{theorem}

It follows from item 3 of the previous theorem that cographs can be uniquely
represented by a rooted labeled tree, its cotree, introduced by Corneil, Lerchs
and Stewart Burlingham in \cite{corneilDAM3}. In \cite{bretscherSIAMJDM22},
Bretscher, Corneil, Habib and Paul, showed that cographs can be recognized, and
its associated cotree can be constructed, in linear time by an algorithm based
on LexBFS. From here, using bottom-up algorithms on their cotrees, many
algorithmic problems which are difficult in general graphs can be efficiently
solved on cographs.

Much of the relevance of cographs comes from real-life applications involving
graph models with just a few induced paths of length three, as discussed by
Corneil, Perl and Stewart Burlingham in \cite{corneilSIAMJC14}. Evidently,
$P_4$-free graphs (cographs) are the most restrictive graph class in this way,
so it becomes important to ask whether a cograph superclass with less
restrictions on the amount of allowed induced $P_4$'s has a behavior similar to
cographs, particularly, whether it allows us to develop efficient algorithms for
solving problems by using a unique tree representation. Next, we briefly
introduce two graph classes which are unlikely to have many induced paths on
four vertices. Such families are known to have unique tree representations
analogous to the cotree, which can be computed in linear time and can be used to
solve some problems in linear time.


\subsection{\texorpdfstring{$P_4$}{P4}-sparse}
\label{sec:P4sparse}

The \textit{$P_4$-sparse graphs} are defined as the graphs such that the
subgraphs induced by any five vertices have at most one induced copy of $P_4$.
Clearly, $P_4$-sparse graphs are precisely the $\{C_5, P_5, \overline{P_5}, P,
\overline{P}, F, \overline{F}\}$-free graphs (see \Cref{fig:extSets}).
Additionally, Jamison and Olariu \cite{jamisonDAM35} provided a connectedness
characterization of $P_4$-sparse graphs based on some special graphs called
spiders, which we now introduce.

A graph $G$ is said to be an \textit{spider} if its vertex set admits a
partition $(S, K, R)$ such that $S$ is an independent set with at least two
vertices, $K$ is a clique, $R$ is completely adjacent to $K$ but completely
nonadjacent to $S$, and there is a bijection $f \colon S \to K$ such that either
$N(s) = \{f(s)\}$ for each $s \in S$ or $N(s) = K -\{f(s)\}$ for each $s \in S$.
For a spider $G = (S, K, R)$ we will say that $S$ is its \textit{legs set}, $K$
is its \textit{body}, and $R$ is its \textit{head}. A \textit{headless spider}
is a spider with empty head. An spider will be called \textit{thin}
(respectively \textit{thick}) if $d(s)=1$ (respectively $d(s) = |K| -1$) for any
$s \in S$. Observe that the complement of a thin spider is a thick spider and
vice versa.

\begin{theorem}\cite{jamisonDAM35}
\label{thm:connCharSparse}
    A graph $G$ is a $P_4$-sparse graph if and only if for every nontrivial
    induced subgraph $H$ of $G$, exactly one of the following statements is
    satisfied
    \begin{enumerate}
        \item $H$ is disconnected.

        \item $\overline{H}$ is disconnected.

        \item $H$ is an spider.
    \end{enumerate}
\end{theorem}

The next observation about spiders will be important in \Cref{sec:2obsmin}. It
follows from the fact that a graph is $(0,\infty)$-polar (i.e. a cluster) if and
only if it is a $P_3$-free graph and, complementarily, that a graph is $(\infty,
0)$-polar if and only if it is a $\overline{P_3}$-free graph.

\begin{remark}
\label{rem:spiderSplit}
    Let $G$ be a spider. If $G$ is a headless spider or the head of $G$ induces
    a split graph, then $G$ is a split graph that has both, $P_3$ and its
    complement, as proper induced subgraphs. Hence, $G$ is not a minimal
    $(s,k)$-polar obstruction for any election of $s$ and $k$.
\end{remark}


\subsection{\texorpdfstring{$P_4$}{P4}-extendible graphs}

Given a graph $G$ and a vertex subset $W$, we denote by $S(W)$ the set of
vertices $x \in V_G - W$ such that $x$ belongs to a $P_4$ sharing vertices with
$W$. If a vertex subset $W$ inducing $P_4$ is such that $S(W)$ has at most one
vertex, we say that $W \cup S(W)$ is an \textit{extension set}. In
\cite{jamisonDAM34}, \textit{$P_4$-extendible graphs} were introduced by Jamison
and Olariu as the graphs $G$ such that, for every set $W$ inducing a $P_4$, $W
\cup S(W)$ is an extension set.

An extension set $D$ is \textit{separable} if no vertex of $D$ is both an
endpoint of some $P_4$ and a midpoint of some $P_4$ in $G[D]$. Notice that any
extension set must induce one of the eight graphs depicted in
\Cref{fig:extSets}; we call these graphs \textit{extension graphs}. In addition,
separable extension sets must induce one of $P_4, P, F$ or their complements;
these graphs are called \textit{separable extension graphs}.

For a separable extension graph $X$ with midpoints set $K$ and endpoints set
$S$, a graph $H$ is said to be an \textit{$X$-spider} if $H$ is an induced
supergraph of $X$ such that $R := V_H \setminus V_X$ is completely adjacent to
$K$ but completely nonadjacent to $S$. If $H$ is an $X$-spider, we say that $(S,
K, R)$ is an \textit{$X$-spider partition} of $H$, and we refer to $S, K$ and
$R$ as the \textit{legs set}, the \textit{body}, and the \textit{head} of $H$,
respectively.   From now on, every time we use the term $X$-spider, we are
assuming that $X$ is a separable extension graph.

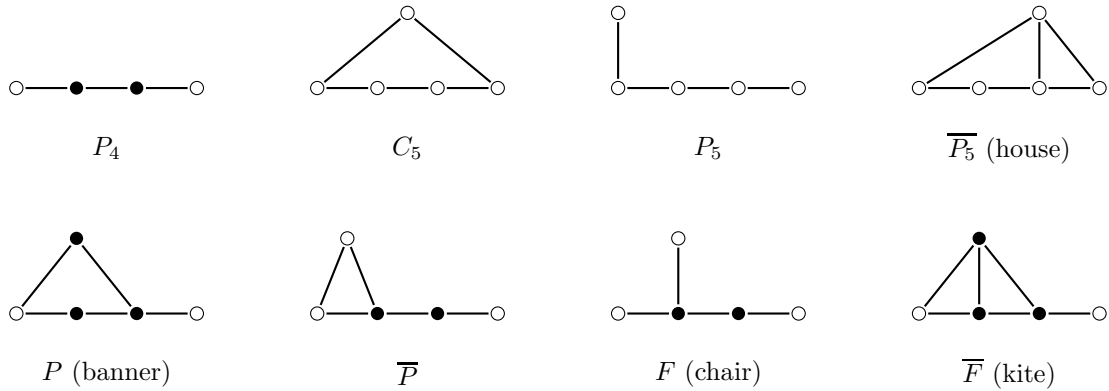
\begin{figure}[ht!]
\centering
\begin{tikzpicture}

\begin{scope}[xscale=0.8] 
\foreach \i in {0,3}
	\node [vertex] (\i) at (\i,0)[]{};
\foreach \i in {1,2}
	\node [blackV] (\i) at (\i,0)[]{};

\foreach \i in {0,1,2}
	\draw let \n1={int(\i+1)} in [edge]
		(\i) to node [above] {} (\n1);
\node [rectangle] (n) at (1.5,-0.8){$P_4$};
\end{scope}

\begin{scope}[xshift=4cm, xscale=0.8] 
\foreach \i in {0,1,2,3}
	\node [vertex] (\i) at (\i,0)[]{};
\node [vertex] (4) at (1.5,1)[]{};

\foreach \i in {0,1,2}
	\draw let \n1={int(\i+1)} in [edge]
		(\i) to node [above] {} (\n1);

\foreach \i/\j in {4/0,4/3}
	\draw [edge] (\i) to node [above] {} (\j);
\node [rectangle] (n) at (1.5,-0.8){$C_5$};
\end{scope}

\begin{scope}[xshift=8cm, xscale=0.8] 
\foreach \i in {0,1,2,3}
	\node [vertex] (\i) at (\i,0)[]{};
\node [vertex] (4) at (0,1)[]{};

\foreach \i in {0,1,2}
	\draw let \n1={int(\i+1)} in [edge]
		(\i) to node [above] {} (\n1);

\foreach \i/\j in {4/0}
	\draw [edge] (\i) to node [above] {} (\j);
\node [rectangle] (n) at (1.5,-0.8){$P_5$};
\end{scope}

\begin{scope}[xshift=12cm, xscale=0.8] 
\foreach \i in {0,1,2,3}
	\node [vertex] (\i) at (\i,0)[]{};
\node [vertex] (4) at (2,1)[]{};

\foreach \i in {0,1,2}
	\draw let \n1={int(\i+1)} in [edge]
		(\i) to node [above] {} (\n1);

\foreach \i/\j in {4/0,4/2,4/3}
	\draw [edge] (\i) to node [above] {} (\j);
\node [rectangle] (n) at (1.5,-0.8){$\overline{P_5}$ (house)};
\end{scope}

\begin{scope}[yshift=-3cm, xscale=0.8] 
\foreach \i in {0,3}
	\node [vertex] (\i) at (\i,0)[]{};
\foreach \i in {1,2}
	\node [blackV] (\i) at (\i,0)[]{};
\node [blackV] (4) at (1,1)[]{};

\foreach \i in {0,1,2}
	\draw let \n1={int(\i+1)} in [edge]
		(\i) to node [above] {} (\n1);

\foreach \i/\j in {4/0,4/2}
	\draw [edge] (\i) to node [above] {} (\j);
\node [rectangle] (n) at (1.5,-0.8){$P$ (banner)};
\end{scope}

\begin{scope}[yshift=-3cm, xshift=4cm, xscale=0.8] 
\foreach \i in {0,3}
	\node [vertex] (\i) at (\i,0)[]{};
\foreach \i in {1,2}
	\node [blackV] (\i) at (\i,0)[]{};
\node [vertex] (4) at (0.5,1)[]{};

\foreach \i in {0,1,2}
	\draw let \n1={int(\i+1)} in [edge]
		(\i) to node [above] {} (\n1);

\foreach \i/\j in {4/0,4/1}
	\draw [edge] (\i) to node [above] {} (\j);

\node [rectangle] (n) at (1.5,-0.8){$\overline{P}$};
\end{scope}

\begin{scope}[yshift=-3cm, xshift=8cm, xscale=0.8] 
\foreach \i in {0,3}
	\node [vertex] (\i) at (\i,0)[]{};
\foreach \i in {1,2}
	\node [blackV] (\i) at (\i,0)[]{};
\node [vertex] (4) at (1,1)[]{};

\foreach \i in {0,1,2}
	\draw let \n1={int(\i+1)} in [edge]
		(\i) to node [above] {} (\n1);

\foreach \i/\j in {4/1}
	\draw [edge] (\i) to node [above] {} (\j);
\node [rectangle] (n) at (1.5,-0.8){$F$ (chair)};
\end{scope}

\begin{scope}[yshift=-3cm, xshift=12cm, xscale=0.8] 
\foreach \i in {0,3}
	\node [vertex] (\i) at (\i,0)[]{};
\foreach \i in {1,2}
	\node [blackV] (\i) at (\i,0)[]{};
\node [blackV] (4) at (1,1)[]{};

\foreach \i in {0,1,2}
	\draw let \n1={int(\i+1)} in [edge]
		(\i) to node [above] {} (\n1);

\foreach \i/\j in {4/0,4/1,4/2}
	\draw [edge] (\i) to node [above] {} (\j);
\node [rectangle] (n) at (1.5,-0.8){$\overline{F}$ (kite)};
\end{scope}

\end{tikzpicture}

\caption{The eight extension graphs. Black vertices are the midpoints of
  separable extension graphs.}
\label{fig:extSets}
\end{figure}

Jamison and Olariu also gave in \cite{jamisonDAM34} the following connectedness
characterization for the class of $P_4$-extendible graphs.

\begin{theorem}\cite{jamisonDAM34}
\label{thm:connChar}
    A graph $G$ is a $P_4$-extendible graph if and only if, for every nontrivial
    induced subgraph $H$ of $G$, precisely one of the following conditions holds
    \begin{enumerate}
        \item $H$ is disconnected.

        \item $\overline{H}$ is disconnected.

        \item $H$ is an extension graph.

        \item There is a unique separable extension graph $X$ such that $H$ is
        an $X$-spider with nonempty head.
    \end{enumerate}
\end{theorem}

Notice that any extension graph, but $P_4$, is a $P_4$-extendible graph which is
not a $P_4$-sparse graph. In addition, any headless spider of order at least six
is a $P_4$-sparse graph which is not a $P_4$-extendible graph. Thus,
$P_4$-sparse and $P_4$-extendible graphs are two cograph superclasses which are
incomparable to each other.


\section{Minimal 2-polar obstructions}
\label{sec:2obsmin}

Throughout this section we give complete lists of minimal $P_4$-sparse and
minimal $P_4$-extendible $2$-polar obstructions, obtaining in this way
characterizations for the $P_4$-sparse and $P_4$-extendible graphs which admit a
$2$-polar partition. These characterizations generalize analogous results given
for cographs by Hell, Hern\'andez-Cruz and Linhares-Sales in \cite{hellDAM261}.
In fact, we base our characterizations in the following propositions, most of
them taken from the mentioned paper.

We start with two lemmas which provide some useful general structural properties
about minimal $k$-polar obstructions.

\begin{lemma}\cite{hellDAM261}
    \label{lem:hellDAM261lem1}
    Let $H$ be a minimal $k$-polar obstruction. The following statements are
    true
    \begin{enumerate}
        \item $H$ has at most $k+2$ components.

        \item $H$ has at least one nontrivial component.

        \item $H$ has at most $k+1$ trivial components.

        \item If $H$ has at least one trivial component, $H$ has at most one
        noncomplete component.

        \item If $H \not \cong (k+1)K_{k+1}$, every complete component of $H$ is
            isomorphic to $K_1$ or $K_2$.
    \end{enumerate}
\end{lemma}

\begin{lemma}\cite{hellDAM261}
\label{lem:hellDAM261seven}
    Let $H$ be a minimal $2$-polar obstruction.
    \begin{enumerate}
        \item $H$ has at least seven vertices.
        
        \item If $H$ has seven vertices and three connected components, then at
            least one of them is an isolated vertex.
    \end{enumerate}
\end{lemma}

Next, we give a slight correction to Lemma 2 in \cite{hellDAM261}, which
characterize the minimal $k$-polar obstructions with the maximum possible number
of components; it is worth noticing that it does not affect the main results in
such paper.

\begin{lemma}
\label{lem:charExtrKobsmin}
    Let $k$ be an integer, $k\ge 2$, and let $G$ be graph. Then, $G$ is a
    minimal $k$-polar obstruction with exactly $k+2$ connected components if and
    only if $G\cong \ell K_1 + (k-\ell+1)K_2 + G'$, where $\ell$ is an integer
    in the set $\{1,\dots,k+1\}$ and $G'$ is a connected complete $k$-partite
    graph which is a minimal $(1,\ell-1)$-polar obstruction and such that, if
    $\ell \le k$, $G'$ is a $(1,\ell)$-polar graph.
\end{lemma}

\begin{proof}
	Suppose $G\cong \ell K_1 + (k-\ell+1)K_2 + G'$, where $\ell$ is an integer
	in the set $\{1,\dots,k+1\}$ and $G'$ is a connected complete $k$-partite
	graph which is a minimal $(1,\ell-1)$-polar obstruction such that, if $\ell
	\le k$, it is a $(1,\ell)$-polar graph. If $G$ is a $(1,k)$-polar graph,
	then $G'$ is a $(1, \ell-1)$-polar graph, but it is not. Thus, since $G$ is
	not $(1,k)$-polar, if its admits a $k$-polar partition $(A,B)$, the subgraph
	$G[A]$ is a connected graph and hence it is completely contained in some
	component of $G$.   But then, $G$ would have at most $k+1$ connected
	components, which is not the case. Hence, $G$ is not a $k$-polar graph.

	Let $v$ be an isolated vertex of $G$. Then $G-v$ is the disjoint union of a
	$k$-cluster with $G'$, and since $G'$ is a complete $k$-partite graph, then
	$G-v$ is a $k$-polar graph. Now, since $G'$ is a minimal $(1,\ell-1)$-polar
	obstruction, for any vertex $w$ of $G'$, $G'-w$ can be partitioned into an
	stable set and an $(\ell-1)$-cluster, so $G-w$ is a $(1,k)$-polar graph, and
	then a $k$-polar graph. Finally, if at least one component of $G$ is a copy
	of $K_2$, then $\ell \le k$ and we have that $G'$ is a $(1,\ell)$-polar
	graph. Thus, for any vertex $u$ in a $K_2$-component of $G$, $G-u$ is a
	$(1,k)$-polar graph. Hence, $G$ is a minimal $k$-polar obstruction which
	evidently has exactly $k+2$ connected components.

	For the converse implication, assume that $G$ is a minimal $k$-polar
	obstruction with precisely $k+2$ components and $\ell$ isolated vertices. If
	$\ell =0$ then $G$ properly contains $K_1 + (k+1)K_2$ as an induced
	subgraph, but that is impossible because both, $G$ and $K_1 + (k+1)K_2$ are
	minimal $k$-polar obstructions. Then, $G$ has at least one isolated vertex,
	and evidently $G$ is not an empty graph, so $\ell \le k+1$. We know by
	\Cref{lem:hellDAM261lem1} that $G$ has at most one noncomplete connected
	component and that any complete component of $G$ has at most two vertices,
	so $G \cong \ell K_1 + (k-\ell+1)K_2 + G'$ where $\ell\in\{1,\dots, k+1\}$
	and $G'$ is a connected graph.

	Notice that $G'$ is not a $(1,\ell-1)$-polar graph, otherwise $G$ would be a
	$(1,k)$-polar graph, and hence a $k$-polar graph. Let $u$ be a vertex of
	$G'$. By the minimality of $G$, we have that $G-u$ is a $k$-polar graph.
	Moreover, since $G-u$ has at least $k+2$ connected components, any $k$-polar
	partition of $G-u$ is necessarily a $(1,k)$-polar partition, which implies
	that $G'-u$ is a $(1,\ell-1)$-polar graph. Then $G'$ is a minimal
	$(1,\ell-1)$-polar obstruction. Now, let $v$ be an isolated vertex of $G$.
	By the minimality of $G$, $G-u$ has a $k$-polar partition $(A,B)$, but it
	cannot be a $(1,k)$-polar partition or $G$ would be a $(1,k)$-polar graph.
	Thus, either $G' \cong K_2$ and $\ell=1$, or $A=V(G')$ and hence $G'$ is a
	complete $k$-partite graph. Finally, if $l\le k$, $G$ has at least one
	$K_2$-component. Let $w$ be a vertex in one of such components. Then $G-w$
	is a $k$-polar graph with $k+2$ connected components, which implies that in
	fact $G-w$ is a $(1,k)$-polar graph, and hence $G'$ is a $(1,\ell)$-polar
	graph.
\end{proof}

A \textit{partial complement} of a graph $H$ is either the usual complement of
$H$, or a graph $\overline{H_1}+\overline{H_2}$, where $H_1$ and $H_2$ are
subgraphs of $H$ obtained by splitting the components of $H$ into two parts,
$H_1$ and $H_2$. The next result shows how the partial complement operation
preserves $2$-polarity, which will be useful for giving compact lists of minimal
$2$-polar obstructions on $P_4$-sparse and $P_4$-extendible graphs. Remarkably,
this lemma was originally proven for the special class of cographs, but the same
proof works for any hereditary class of graphs closed under complement and
disjoint union operations, particularly, it works for the classes of
$P_4$-sparse and $P_4$-extendible graphs.

\begin{lemma}\cite{hellDAM261}
\label{lem:hellDAM261partialComp}
    Let $\mathcal{G}$ be a hereditary class of graphs closed under complement
    and disjoint union operations, and let $G \in \mathcal{G}$ be a $2$-polar
    graph. Then, any partial complement of $G$ is a $2$-polar graph belonging to
    $\mathcal{G}$.
\end{lemma}

Based on the previous propositions, we can easily check that the graphs in
\Crefrange{fig:2obsmin9v}{fig:2obsP4ext4}, as well as their complements, are all
of them minimal $2$-polar obstructions: it is enough to verify that, for each of
the mentioned figures, one of its graphs, let say $F$, is a $2$-polar
obstruction, and that the closure of $\{F\}$ under partial complements is
precisely the set of all graphs in the same figure and their complements. In the
following sections we prove that any minimal $2$-polar obstruction that is a
$P_4$-sparse or a $P_4$-extendible graph is either a graph depicted in Figures
\Crefrange{fig:2obsmin9v}{fig:2obsP4ext4} of the complement of one of them.


\subsection{\texorpdfstring{$P_4$}{P4}-sparse minimal 2-polar obstructions}

Throughout this section we characterize $P_4$-sparse graphs admitting a
$2$-polar partition by means of its family of minimal obstructions. At the end
of the section we conclude that any $P_4$-sparse minimal $2$-polar obstruction
is in fact a cograph, which is interesting since also any known $P_4$-sparse
minimal $(s,k)$-polar obstruction is a cograph. We start by proving that the
complement of any connected $P_4$-sparse minimal $2$-polar obstruction is a
disconnected graph.

\begin{proposition}
\label{pro:p4sp2polSIIsplit}
    If $G$ is a spider, then $G$ is $2$-polar if and only if $G$ is a split
    graph.
\end{proposition}

\begin{proof}
	Let $(S, K, R)$ be the spider partition of $G$. We only need to prove that
	any $2$-polar spider is, in fact, a split graph. Since $k$-polar graphs are
	closed under complements, and headless spiders trivially are split graphs,
	we can assume that $G$ is a thin spider with nonempty head. Let $(V_1, V_2,
	V_3, V_4)$ be a $2$-polar partition of $G$, and for any $i\in\{1,2,3,4\}$,
	let $R_i = V_i\cap R$. Notice that, since $K$ is completely adjacent to $R$,
	$R_i = \varnothing$ for some $i \in \{1, \dots, 4\}$.

	First, suppose that $(R_1,R_3,R_4)$ is a $(1,2)$-polar partition of $G[R]$.
	Again, some of $R_1,R_3$ and $R_4$ must be empty because $K$ and $R$ are
	completely adjacent and $K$ has at least two vertices. Thus, either $(R_1,
	R_3)$ is a split partition of $G[R]$, or $(R_3,R_4)$ is a $(0,2)$-polar
	partition of $G[R]$. But the second case is not possible since then, $S\cup
	K\subseteq V_1\cup V_2$, which is impossible since $G[S \cup K]$ is not a
	complete multipartite graph. Hence, $G[R]$ is a split graph and, by
	\Cref{rem:spiderSplit}, also is $G$. The case in which $(R_1,R_2,R_3)$ is a
	$(2,1)$-polar partition of $G[R]$ can be treated in a similar way.
\end{proof}

\begin{corollary}
    If $G$ is a spider, then $G$ is not a minimal $2$-polar obstruction. In
    consequence, for any $P_4$-sparse minimal $2$-polar obstruction $H$, either
    $H$ or its complement is disconnected.
\end{corollary}

\begin{proof}
  Let $(S, K, R)$ be the spider partition of $G$. As in the lemma above, we can
  suppose that $G$ is a thin spider. Assume for a contradiction that $G$ is a
  minimal $2$-polar obstruction so, by the previous lemma and
  \Cref{rem:spiderSplit}, we have that $G[R]$ is not a split graph. Then, for
  any $r\in R$, $G-r$ is a spider which is $2$-polar, so $G[R]-r$ is a split
  graph. Thus $G[R]$ is a $P_4$-sparse minimal split obstruction, that is to
  say, $G[R]$ is isomorphic to either $2K_2$ or $C_4$. From here is easy to
  prove that deleting either one leg or one vertex of the body of $G$ the
  resulting graph is not a $2$-polar graph, contradicting the minimality of $G$.
  Hence, a $P_4$-sparse minimal $2$-polar obstruction is not a spider, and the
  result directly follows from \Cref{thm:connCharSparse}.
\end{proof}

By \Cref{lem:hellDAM261lem1}, any $P_4$-sparse minimal $2$-polar obstruction has
at most four connected components. With the purpose of giving the complete list
of such obstructions, we mention first some useful propositions on minimal
$(s,1)$-polar obstructions.

\begin{theorem}\cite{contrerasCogGen1}
\label{thm:allDisc(s1)obsmin}
    Let $s$ be an integer, $s\ge 2$. If $G$ is a disconnected minimal
    $(s,1)$-polar obstruction, then $G$ satisfies one of the following
    assertions:
    \begin{enumerate}
        \item $G$ is isomorphic to one of the graphs depicted in
        \Cref{fig:essentials}.

        \item $G \cong 2K_{s+1}$.

        \item $G \cong K_2 + (2K_1\oplus K_s)$.

        \item $G \cong K_1 + (C_4\oplus K_{s-1})$.
    \end{enumerate}
\end{theorem}

\begin{figure}[ht!]
\centering
\begin{tikzpicture}

\begin{scope}[scale=0.8]
\begin{scope}[scale=1]
\node [vertex] (1) at (0,0)[]{};
\node [vertex] (2) at (1,0)[]{};
\node [vertex] (3) at (0,0.8)[]{};
\node [vertex] (4) at (1,0.8)[]{};
\node [vertex] (5) at (0.5,1.5)[]{};

\foreach \from/\to in {1/2,3/4}
  \draw [edge] (\from) to (\to);

\node (1) at (0.5,-0.75){$E_1 = K_1+2K_2$};
\end{scope}

\begin{scope}[xshift=3.5cm, scale=1]
\node [vertex] (1) at (0,0)[]{};
\node [vertex] (2) at (0,0.9)[]{};
\node [vertex] (3) at (0,1.8)[]{};
\node [vertex] (4) at (1,0)[]{};
\node [vertex] (5) at (1,0.9)[]{};
\node [vertex] (6) at (1,1.8)[]{};

\foreach \from/\to in {1/2,2/3,4/5,5/6}
  \draw [edge] (\from) to (\to);

\node (1) at (0.5,-0.75){$E_2 = 2P_3$};
\end{scope}

\begin{scope}[xshift=7cm, scale=1]
\node [vertex] (1) at (0,0)[]{};
\node [vertex] (2) at (1,0)[]{};
\node [vertex] (3) at (0,1)[]{};
\node [vertex] (4) at (1,1)[]{};
\node [vertex] (5) at (0,1.8)[]{};
\node [vertex] (6) at (1,1.8)[]{};

\foreach \from/\to in {1/2,1/3,3/4,2/4}
  \draw [edge] (\from) to (\to);

\node (1) at (0.5,-0.75){$E_3 = C_4+2K_1$};
\end{scope}

\begin{scope}[xshift=11cm, scale=1]
\node [vertex] (1) at (0.5,0)[]{};
\node [vertex] (2) at (-0.1,0.6)[]{};
\node [vertex] (3) at (0.5,1.2)[]{};
\node [vertex] (4) at (1.1,0.6)[]{};
\node [vertex] (5) at (0.5,0.6)[]{};
\node [vertex] (6) at (0.5,1.8)[]{};

\foreach \from/\to in {1/2,2/3,3/4,4/1,1/5,2/5,4/5}
  \draw [edge] (\from) to (\to);

\node (1) at (0.5,-0.7){$E_7 = K_1+\overline{P_3+K_2}$};
\end{scope}

\begin{scope}[xshift=1.75cm, yshift=-3cm, scale=0.85]
\node [vertex] (1) at (126:0.8)[]{};
\node [vertex] (2) at (54:0.8)[]{};
\node [vertex] (3) at (342:0.8)[]{};
\node [vertex] (4) at (270:0.8)[]{};
\node [vertex] (5) at (198:0.8)[]{};
\node [vertex] (6) at (90:1.33)[]{};

\foreach \from/\to in {1/2,2/3,3/4,4/5,5/1}
  \draw [edge] (\from) to (\to);

\node (1) at (0,-1.72){$E_{10} = K_1+C_5$};
\end{scope}

\begin{scope}[xshift=6.25cm, yshift=-3cm, scale=0.8]
\node [vertex] (1) at (-2,0)[]{};
\node [vertex] (2) at (-1.2,0.8)[]{};
\node [vertex] (3) at (-0.4,0)[]{};
\node [vertex] (4) at (-1.2,-0.8)[]{};
\node [vertex] (5) at (0.8,0)[]{};
\node [vertex] (6) at (0,1.45)[]{};

\foreach \from/\to in {1/2,2/3,3/4,4/1,3/5}
  \draw [edge] (\from) to (\to);

\node (1) at (-0.2,-1.8){$E_{11} = K_1+P$};
\end{scope}

\begin{scope}[xshift=9cm, yshift=-2.7cm, scale=1]
\node [vertex] (1) at (0,-1)[]{};
\node [vertex] (2) at (0,-0.2)[]{};
\node [vertex] (3) at (1,-0.2)[]{};
\node [vertex] (4) at (1,-1)[]{};
\node [vertex] (5) at (0.5,0.3)[]{};
\node [vertex] (6) at (1,0.85)[]{};

\foreach \from/\to in {1/2,2/3,3/4,4/1,2/5,5/3}
  \draw [edge] (\from) to (\to);

\node (1) at (0.5,-1.7){$E_{12} = K_1+\overline{P_5}$};
\end{scope}

\end{scope}

\end{tikzpicture}

\caption{Some minimal $(\infty,1)$-polar obstructions.}
\label{fig:essentials}
\end{figure}
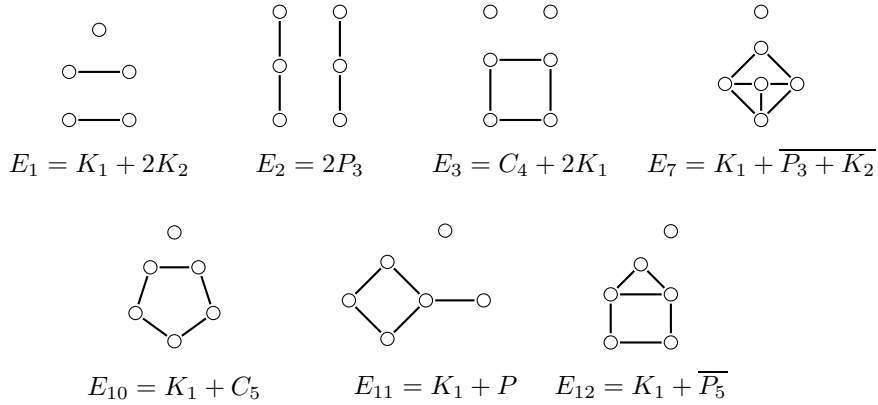

\begin{proposition}\cite{contrerasCogGen1}
\label{cor:spiderNOobsmin}
Let $s$ be a positive integer. Any $P_4$-sparse minimal $(s,1)$-polar
obstruction $G$ is a cograph. In consequence, either $G$ or its complement is
disconnected.
\end{proposition}

\begin{proposition}\cite{contrerasCogGen1}
\label{rem:psE}
There are exactly nine $P_4$-sparse minimal $(2,1)$-polar obstructions; they are
the graphs $E_1, \dots, E_9$ depicted in
\Cref{fig:essentials,fig:P4ext(21)obsmin}.
\end{proposition}

\begin{figure}[ht!]
\centering
\begin{tikzpicture}

\begin{scope}[scale=0.8]
\begin{scope}[yscale=0.75]
\foreach \i in {0,1}
	\foreach \j in {0,1,2}
		\node [vertex] (\i\j) at (\i-0.5,\j-1)[]{};

\foreach \i/\j in {00/11,00/12,01/10,01/12,02/10,02/11,00/01,01/02,10/11,11/12}
  \draw [edge] (\i) to (\j);

\foreach \i/\j in {00/02,12/10}
  \draw [curveEdge] (\i) to (\j);

\node (n) at (0,-1.75){$E_4 = \overline{3K_2}$};
\end{scope}

\begin{scope}[xshift=3.5cm, yscale=0.75]
\foreach \i in {0,1}
	\foreach \j in {0,1,2}
		\node [vertex] (\i\j) at (\i-0.5,\j-1)[]{};

\foreach \i/\j in {00/01,01/02,10/11,11/12,00/10,00/11,01/10,01/12,02/11,02/12}
  \draw [edge] (\i) to (\j);

\node (n) at (0,-1.8){$E_5 = \overline{K_2+C_4}$};
\end{scope}

\begin{scope}[xshift=7cm, scale=0.75]
\foreach \i in {0,...,3}
	\node [vertex] (\i) at ({(\i*90)}:0.9){};
\node [vertex] (4) at (0,0){};
\node [vertex] (5) at (1,1){};

\foreach \i in {0,...,3}
{
	\draw let \n1={int(mod(\i+1,4))} in [edge] (\i) to (\n1);
	\draw [edge] (\i) to (4);
}

\node (n) at (0,-1.8){$E_6 = K_1 + W_4$};
\end{scope}

\begin{scope}[xshift=0cm, yshift=-3.32cm, scale=0.75, yscale=1]
\foreach \i in {0,...,3}
	\node [vertex] (\i) at ($ (0,-0.5) + ({(\i*90)}:0.9) $){};
\node [vertex] (4) at (-0.6,1.2){};
\node [vertex] (5) at (0.6,1.2){};

\foreach \i in {0,...,3}
	\draw let \n1={int(mod(\i+1,4))} in [edge] (\i) to (\n1);

\foreach \i/\j in {1/3,4/5}
	\draw [edge] (\i) to (\j);

\node (n) at (0,-2.4){$E_8 = K_2+(K_2\oplus \overline{2K_1})$};
\end{scope}

\begin{scope}[xshift=3.5cm, yshift=-3.5cm, scale=0.75]
\foreach \i in {0,...,2}
	\node [vertex] (\i) at ($ (-0.3,0.6) + ({(\i*120)}:0.9) $){};
\foreach \i in {3,...,5}
	\node [vertex] (\i) at ($ (0.3,-0.6) + ({180+(\i*120)}:0.9) $){};

\foreach \i in {0,...,2}
	\draw let \n1={int(mod(\i+1,3))} in [edge] (\i) to (\n1);
\foreach \i in {3,...,5}
	\draw let \n1={int(3+mod(\i+1,3))} in [edge] (\i) to (\n1);

\node (n) at (0,-2.2){$E_9 = 2K_3$};
\end{scope}

\begin{scope}[xshift=7cm, yshift=-3.5cm, scale=0.75]
\foreach \i in {0,...,4}
	\node [vertex] (\i) at ({90+(\i*72)}:1.3){};
\node [vertex] (5) at (-0.4,-0.3){};
\node [vertex] (6) at (0.4,-0.3){};

\foreach \i in {0,...,4}
	\draw let \n1={int(mod(\i+1,5))} in [edge] (\i) to (\n1);
\foreach \i in {0,...,4}
	\foreach \j in {5,6}
		\draw [edge] (\i) to (\j);

\node (n) at (0,-2.2){$E_{13} = \overline{K_2+C_5}$};
\end{scope}

\end{scope}

\end{tikzpicture}

\caption{Some minimal $(2,1)$-polar obstructions.}
\label{fig:P4ext(21)obsmin}
\end{figure}
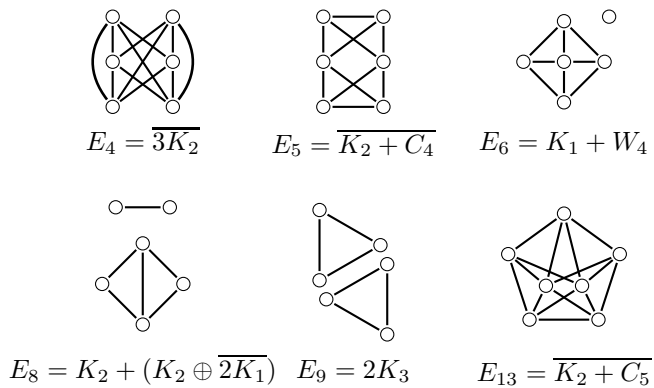

Now we have the necessary tools to prove that there are exactly three
$P_4$-sparse minimal $2$-polar obstructions with four connected components.

\begin{proposition}
\label{pro:techLem}
    Let $\ell$ be a positive integer. If $G$ is a connected $P_4$-sparse minimal
    $(1,\ell-1)$-polar obstruction which is a complete multipartite graph, then
    $G$ is isomorphic to either $K_{\ell,\ell}$ or $K_1\oplus C_4$.
\end{proposition}

\begin{proof}
	Clearly, if $\ell = 1$, $G \cong K_2$, while if $\ell = 2$, $G \cong C_4$.
	For $\ell \ge 3$, we have from \Cref{cor:spiderNOobsmin} that $\overline{G}$
	is a disconnected graph, and it follows from \Cref{thm:allDisc(s1)obsmin}
	that $G$ is isomorphic to either $K_{\ell, \ell}$ or $K_1 \oplus C_4$.
\end{proof}

\begin{corollary}
\label{cor:Kobsmink+2}
    If $G$ is a $P_4$-sparse graph, then $G$ is a minimal $2$-polar obstruction
    with exactly 4 connected components if and only if $G \cong \ell K_1 +
    (3-\ell) K_2 + K_{\ell,\ell}$ for some integer $\ell \in \{1, 2, 3\}$.
\end{corollary}

\begin{proof}
	Let $G$ be a $P_4$-sparse graph. By \Cref{lem:charExtrKobsmin} we have that
	$G$ is a minimal $2$-polar obstruction with precisely four connected
	components if and only if $G \cong \ell K_1 + (3 - \ell) K_2 + G'$, where
	$\ell \in \{1, 2, 3\}$, and $G'$ is a connected complete bipartite graph
	which is a minimal $(1, \ell-1)$-polar obstruction such that, if $\ell \ne
	3$, $G'$ is a $(1, \ell)$-polar graph. In addition, we have from
	\Cref{pro:techLem} that the only connected $P_4$-sparse minimal $(1,
	\ell-1)$-polar obstruction which is a complete bipartite graph is $K_{\ell,
	\ell}$. The result follows since $K_{\ell, \ell}$ trivially is a $(1,
	\ell)$-polar graph.
\end{proof}

Hannebauer \cite{hannebauerTH'10} proved that, for any nonnegative integers $s$
and $k$, any $P_4$-sparse minimal $(s,k)$-polar obstruction has at most
$(s+1)(k+1)$ vertices. Thus, we have by \Cref{lem:hellDAM261seven} that any
$P_4$-sparse minimal $2$-polar obstruction has at least seven and at most nine
vertices. The following three lemmas completely characterize such minimal
obstructions depending on their order; the proofs are simple generalizations of
the analogous proofs given in \cite{hellDAM261} for cographs.

\begin{lemma}
    The disconnected $P_4$-sparse minimal $2$-polar obstructions on 7 vertices
    are exactly the graphs $F_1, \dots, F_5$ depicted in \Cref{fig:2obsmin7v}.
\end{lemma}

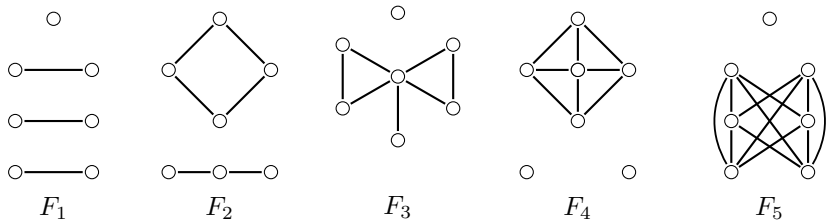
\begin{figure}[ht!]
\centering
\begin{tikzpicture}

\begin{scope}[scale=0.8]

\begin{scope}[scale=0.85]
\foreach \i in {0,1}
	\foreach \j in {0,1,2}
		\node [vertex] (\i\j) at (\i*1.5,\j)[]{};
\node [vertex] (7) at (0.75,3)[]{};

\foreach \j in {0,1,2}
	\draw [edge]  (0\j) to (1\j);
\node [rectangle] (n) at (0.75,-0.7){$F_1$};
\end{scope}

\begin{scope}[scale=0.85, xshift=4cm, yshift=2cm]
\foreach \i in {0,...,2}
	\node [vertex] (\i) at (\i-1,-2)[]{};
\foreach \i in {3,...,6}
	\node [vertex] (\i) at ({(\i*90)}:1)[]{};

\foreach \i in {0,1}
	\draw let \n1={int(\i+1)} in [edge] (\i) to (\n1);

\foreach \i/\j in {3/4,4/5,5/6,6/3}
	\draw [edge] (\i) to (\j);

\node [rectangle] (n) at (0,-2.7){$F_2$};
\end{scope}

\begin{scope}[scale=1.06, xshift=6cm, yshift=1.5cm]
\foreach \i in {0,...,5}
	\node [vertex] (\i) at ({150+(\i*60)}:1)[]{};
\node [vertex] (6) at (0,0)[]{};

\foreach \i in {0,...,4}
	\draw [edge] (\i) to (6);

\foreach \i/\j in {0/1,3/4}
	\draw [edge] (\i) to (\j);

\node [rectangle] (n) at (0,-2.05){$F_3$};
\end{scope}

\begin{scope}[scale=0.85, xshift=11cm, yshift=2cm]
\node [vertex] (0) at (-1,-2)[]{};
\node [vertex] (1) at (1,-2)[]{};
\node [vertex] (2) at (0,0)[]{};
\foreach \i in {3,...,6}
	\node [vertex] (\i) at ({(\i*90)}:1)[]{};
\node [vertex] (2) at (0,0)[]{};

\foreach \i/\j in {3/4,4/5,5/6,6/3,2/3,2/4,2/5,2/6}
	\draw [edge] (\i) to (\j);

\node [rectangle] (n) at (0,-2.7){$F_4$};
\end{scope}

\begin{scope}[scale=0.85, xshift=14cm, yshift=0cm]
\foreach \i in {0,1}
	\foreach \j in {0,1,2}
		\node [vertex] (\i\j) at (\i*1.5,\j)[]{};
\node [vertex] (7) at (0.75,3)[]{};

\foreach \i\j in {00/11,00/12,01/10,01/12,
02/10,02/11,00/01,01/02,10/11,11/12}
	\draw [edge]  (\i) to (\j);

\foreach \i\j in {00/02,12/10}
	\draw [-, shorten <=1pt, shorten >=1pt, >=stealth, line width=.7pt, bend
    left=30]  (\i) to (\j);

\node [rectangle] (n) at (0.75,-0.7){$F_5$};
\end{scope}

\end{scope}

\end{tikzpicture}

\caption{$P_4$-sparse minimal $2$-polar obstructions on 7 vertices.}
\label{fig:2obsmin7v}
\end{figure}

\begin{proof}
    Let $H$ be a disconnected $P_4$-sparse minimal $2$-polar obstruction on
    seven vertices. By the observation after \Cref{lem:hellDAM261partialComp} it
    is enough to prove that $H \cong F_i$ for some $i \in \{1, \dots, 5\}$. If
    $H$ has four connected components or it can be transformed by a sequence of
    partial complementations into a graph with four components, it follows from
    \Cref{cor:Kobsmink+2,lem:hellDAM261partialComp} that $H$ is isomorphic to
    $F_i$ for some $i \in \{ 1, \dots, 5 \}$. Thus, we can assume that any graph
    obtained from $H$ by partial complementations has at most three components;
    from here we can replicate the argument in Lemma 7 of \cite{hellDAM261} to
    assume that $H$ is a graph with precisely two connected components, one of
    them being a trivial graph.

    Since $H$ is not a $2$-polar graph, its nontrivial component must contain a
    minimal $(2,1)$-polar obstruction $H'$ as an induced subgraph. Moreover,
    $H'$ cannot be a disconnected graph on six vertices, so we have from
    \Cref{rem:psE} that $H \in \{ K_1 + 2K_2, \overline{3K_2}, 2K_2 \oplus 2K_1
    \}$. If $H' \cong \overline{3K_2}$, $H$ is the graph $F_5$ in
    \Cref{fig:2obsmin7v}. If $H' \cong 2K_2\oplus 2K_1$, is straightforward to
    verify that $H$ is a $(1,2)$-polar graph, which cannot occur. Otherwise, if
    $H' \cong K_1 + 2K_2$, we have that $H \cong F_3$, because $P_4$-sparse
    graphs are $\{\overline P, P_5\}$-free and $H'$ is contained in a connected
    component of $H$ on six vertices.
\end{proof}

\begin{lemma}
\label{lem:p4sp2obsminOrder9}
The disconnected $P_4$-sparse minimal $2$-polar obstructions on 9 vertices are
the graphs $F_{21}, \dots, F_{24}$ depicted in \Cref{fig:2obsmin9v}.
\end{lemma}

\begin{figure}[ht!]
\centering
\begin{tikzpicture}

\begin{scope}[scale=0.8]

\begin{scope}[scale=0.85]
\foreach \i in {0,1}
	\foreach \j in {1,2,3}
		\node [vertex] (\i\j) at (\i*2,\j)[]{};
\foreach \i in {0,1,2}
	\node [vertex] (\i) at (\i,0)[]{};

\foreach \i in {01,02,03}
	\foreach \j in {11,12,13}
	\draw [edge]  (\i) to (\j);

\node [rectangle] (n) at (1,-0.85){$F_{21}$};
\end{scope}

\begin{scope}[scale=0.63, xshift=6.4cm, yshift=2cm]
\foreach \i in {0,...,2}
	\node [vertex] (\i) at ($ (-0.5,1.2) + (\i*120:1) $)[]{};

\foreach \i in {3,...,5}
	\node [vertex] (\i) at ($ (0.5,0) + ({180+(\i*120)}:1) $)[]{};

\foreach \i in {6,...,8}
	\node [vertex] (\i) at ($ (-0.5,-1.2) + (\i*120:1) $)[]{};

\foreach \i/\j in {0/1,1/2,2/0,3/4,4/5,5/3,6/7,7/8,8/6}
	\draw [edge] (\i) to (\j);

\node [rectangle] (n) at (0,-3.15){$F_{22}$};
\end{scope}

\begin{scope}[scale=0.765, xshift=7.6cm, yshift=0cm]
\foreach \j in {0,1,3}
	\foreach \i in {0,2}
		\node [vertex] (\i\j) at (\i,\j)[]{};
\foreach \j in {1,2,3}
	\node [vertex] (1\j) at (1,\j)[]{};

\foreach \i in {03,13,23,01,11,21}
	\draw [edge] (\i) to (12);

\foreach \i/\j in {03/13,13/23,01/11,11/21,00/20}
	\draw [edge] (\i) to (\j);

\foreach \i/\j in {03/23,21/01}
	\draw [-, shorten <=1pt, shorten >=1pt, >=stealth, line width=.7pt, bend
    left=40] (\i) to (\j);

\node [rectangle] (n) at (1,-0.95){$F_{23}$};
\end{scope}

\begin{scope}[scale=0.765, xshift=11cm, yshift=0cm]
\foreach \i/\j in {0/3,1/3,2/3,0/1,1/1,2/1,1/0}
	\node [vertex] (\i\j) at (\i,\j)[]{};
\node [vertex] (x) at (0.4,2)[]{};
\node [vertex] (y) at (1.6,2)[]{};

\foreach \i in {x,y}
	\foreach \j in {03,13,23,01,11,21}
		\draw [edge] (\i) to (\j);

\foreach \i/\j in {x/y,03/13,13/23,01/11,11/21}
	\draw [edge] (\i) to (\j);

\foreach \i/\j in {03/23,21/01}
    \draw [-, shorten <=1pt, shorten >=1pt, >=stealth, line width=.7pt, bend
    left=40] (\i) to (\j);

\node [rectangle] (n) at (1,-0.95){$F_{24}$};
\end{scope}

\end{scope}

\end{tikzpicture}

\caption{$P_4$-sparse minimal $2$-polar obstructions on 9 vertices.}
\label{fig:2obsmin9v}
\end{figure}
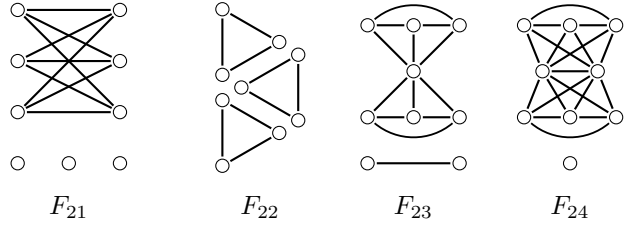

\begin{proof}
	Almost all the arguments used in the proof of Lemma 8 in \cite{hellDAM261}
	are still valid for $P_4$-sparse graphs. We only have to care about the case
	when $H$ is a $P_4$-sparse minimal $2$-polar obstruction on 9 vertices with
	three connected components and precisely two isolated vertices. In such a
	case the nontrivial connected component of $H$, $B_3$, is either a spider or
	the join of two smaller $P_4$-sparse graphs $T_1$ and $T_2$. In the former
	case, since the head of $B_3$ has at most three vertices, $B_3$ is a split
	graph, so $H$ is too. The latter case follows as in the original proof.
\end{proof}

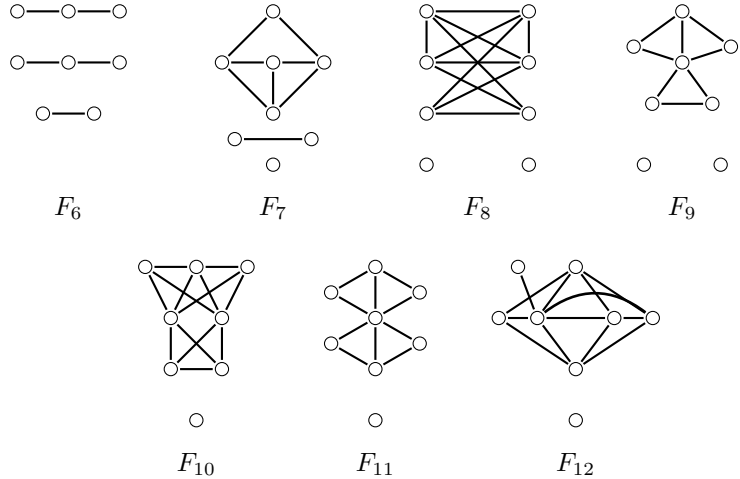
\begin{figure}[ht!]
\centering
\begin{tikzpicture}

\begin{scope}[scale=0.8]

\begin{scope}[scale=0.85, xshift=0cm, yshift=-0cm]
\foreach \i in {0,1,2}
	\foreach \j in {2,3}
		\node [vertex] (\i\j) at (\i,\j)[]{};
\foreach \i in {0,1}
	\node [vertex] (\i) at (\i+0.5,1)[]{};

\foreach \i/\j in {0/1,03/13,13/23,02/12,12/22}
	\draw [edge]  (\i) to (\j);

\node [rectangle] (n) at (1,-0.85){$F_{6}$};
\end{scope}

\begin{scope}[scale=0.85, xshift=5cm, yshift=0cm]
\foreach \j in {0,...,3}
	\node [vertex] (\j) at ($ (0,2) + ({(\j*90)}:1) $)[]{};
\node [vertex] (c) at (0,2)[]{};
\node [vertex] (a) at (-0.75,0.5)[]{};
\node [vertex] (b) at (0.75,0.5)[]{};
\node [vertex] (x) at (0,0)[]{};

\foreach \i in {0,2,3}
	\draw [edge]  (\i) to (c);
\foreach \i/\j in {0/1,1/2,2/3,3/0,a/b}
	\draw [edge]  (\i) to (\j);

\node [rectangle] (n) at (0,-0.85){$F_{7}$};
\end{scope}

\begin{scope}[scale=0.85, xshift=8cm, yshift=0cm]
\foreach \i in {0,2}
	\foreach \j in {0,...,3}
		\node [vertex] (\i\j) at (\i,\j)[]{};

\foreach \i in {23,22,21}
	\foreach \j in {03,02,01}
		\draw [edge]  (\i) to (\j);

\foreach \i/\j in {02/03,22/23}
	\draw [edge]  (\i) to (\j);

\node [rectangle] (n) at (1,-0.85){$F_{8}$};
\end{scope}

\begin{scope}[scale=0.85, xshift=13cm, yshift=0cm]
\foreach \j in {0,...,4}
	\node [vertex] (\j) at ($ (0,2) + ({90+(\j*72)}:1) $)[]{};
\node [vertex] (c) at (0,2)[]{};
\node [vertex] (a) at (-0.75,0)[]{};
\node [vertex] (b) at (0.75,0)[]{};

\foreach \i in {0,...,4}
	\draw [edge]  (\i) to (c);
\foreach \i/\j in {4/0,0/1,2/3}
	\draw [edge]  (\i) to (\j);

\node [rectangle] (n) at (0,-0.85){$F_{9}$};
\end{scope}

\begin{scope}[scale=0.85, xshift=2.5cm, yshift=-5cm]
\foreach \i in {0,1,2}
	\node [vertex] (\i3) at (\i,3)[]{};
\foreach \i in {0,1}
	\foreach \j in {1,2}
		\node [vertex] (\i\j) at (\i+0.5,\j)[]{};
\node [vertex] (x) at (1,0)[]{};

\foreach \i in {02,12}
	\foreach \j in {03,13,23,01,11}
		\draw [edge]  (\i) to (\j);

\foreach \i/\j in {03/13,13/23,01/11}
	\draw [edge]  (\i) to (\j);

\node [rectangle] (n) at (1,-0.85){$F_{10}$};
\end{scope}

\begin{scope}[scale=0.85, xshift=7cm, yshift=-5cm]
\foreach \j in {0,...,5}
	\node [vertex] (\j) at ($ (0,2) + ({90+(\j*60)}:1) $)[]{};
\node [vertex] (c) at (0,2)[]{};
\node [vertex] (x) at (0,0)[]{};

\foreach \i in {0,...,5}
	\draw [edge]  (\i) to (c);
\foreach \i/\j in {5/0,0/1,2/3,3/4}
	\draw [edge]  (\i) to (\j);

\node [rectangle] (n) at (0,-0.85){$F_{11}$};
\end{scope}

\begin{scope}[xscale=0.64, yscale=0.85, xshift=14.5cm, yshift=-5cm]
\foreach \j in {0,...,3}
	\node [vertex] (\j) at ($ (0,2) + ({90+(\j*90)}:1) $)[]{};
\node [vertex] (a) at (-2,2) []{};
\node [vertex] (b) at (2,2) []{};
\node [vertex] (c) at (-1.5,3)[]{};
\node [vertex] (x) at (0,0)[]{};

\foreach \i in {0,2}
	\foreach \j in {a,1,3,b}
		\draw [edge]  (\i) to (\j);
\foreach \i/\j in {1/c,a/1,1/3,3/b}
	\draw [edge]  (\i) to (\j);
\draw [curveEdge]  (1) to (b);

\node [rectangle] (n) at (0,-0.85){$F_{12}$};
\end{scope}

\end{scope}

\end{tikzpicture}

\caption{Family $A$ of $P_4$-sparse minimal $2$-polar obstructions on 8
vertices.}
\label{fig:2obsmin8va}
\end{figure}

\begin{lemma}
    The disconnected $P_4$-sparse minimal $2$-polar obstructions on 8 vertices
    are the graphs $F_{6}, \dots, F_{20}$ and $F_{25}$, depicted in
    \Cref{fig:2obsmin8va,fig:2obsmin8vb}.
\end{lemma}

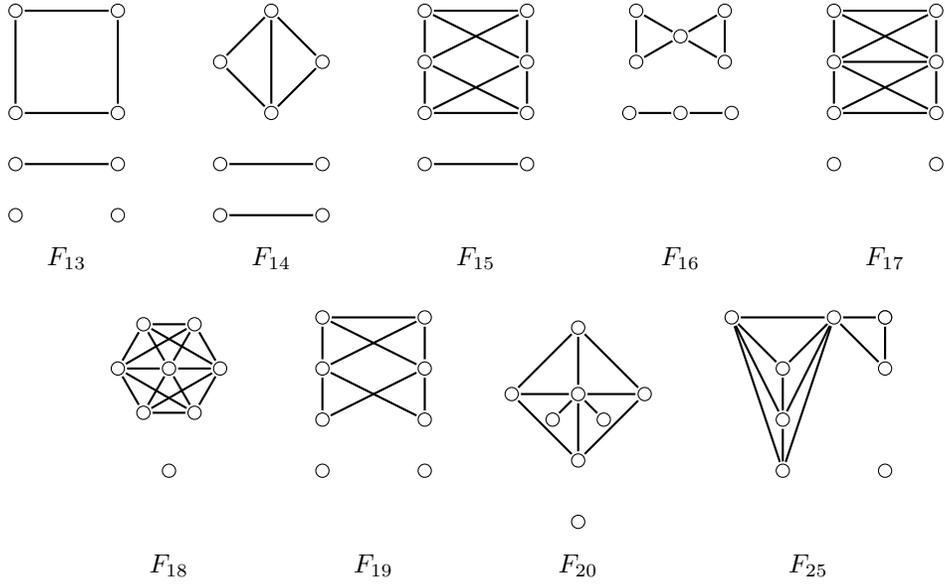
\begin{figure}[ht!]
\centering
\begin{tikzpicture}

\begin{scope}[scale=0.8]

\begin{scope}[scale=0.85, xshift=0cm, yshift=0cm]
\foreach \i in {0,2}
	\foreach \j in {0,1,2,4}
		\node [vertex] (\i\j) at (\i,\j)[]{};

\foreach \i/\j in {04/24,24/22,22/02,02/04,01/21}
	\draw [edge]  (\i) to (\j);

\node [rectangle] (n) at (1,-0.85){$F_{13}$};
\end{scope}

\begin{scope}[scale=0.85, xshift=4cm, yshift=0cm]
\foreach \i in {0,2}
	\foreach \j in {0,1,3}
		\node [vertex] (\i\j) at (\i,\j)[]{};
\foreach \j in {2,4}
	\node [vertex] (1\j) at (1,\j)[]{};

\foreach \i/\j in {03/14,14/23,23/12,12/03,12/14,01/21,00/20}
	\draw [edge]  (\i) to (\j);

\node [rectangle] (n) at (1,-0.85){$F_{14}$};
\end{scope}

\begin{scope}[scale=0.85, xshift=8cm, yshift=0cm]
\foreach \i in {0,2}
	\foreach \j in {1,...,4}
		\node [vertex] (\i\j) at (\i,\j)[]{};

\foreach \i in {04,24,22,02}
	\foreach \j in {03,23}
	\draw [edge]  (\i) to (\j);

\foreach \i/\j in {04/24,02/22,01/21}
	\draw [edge]  (\i) to (\j);

\node [rectangle] (n) at (1,-0.85){$F_{15}$};
\end{scope}

\begin{scope}[scale=0.85, xshift=13cm, yshift=0cm]
\foreach \j in {0,1,3,4}
	\node [vertex] (\j) at ($ (0,3.5) + ({(\j*60)-30}:1) $)[]{};
\node [vertex] (x) at (0,3.5)[]{};

\foreach \i in {5,6,7}
	\node [vertex] (\i0) at (\i-6,2)[]{};

\foreach \i/\j in {3/4,1/0,x/3,x/4,x/0,x/1,50/60,60/70}
	\draw [edge]  (\i) to (\j);

\node [rectangle] (n) at (0,-0.85){$F_{16}$};
\end{scope}

\begin{scope}[scale=0.85, xshift=16cm, yshift=0cm]
\foreach \i in {0,2}
	\foreach \j in {1,...,4}
		\node [vertex] (\i\j) at (\i,\j)[]{};

\foreach \i in {04,24,22,02}
	\foreach \j in {03,23}
	\draw [edge]  (\i) to (\j);

\foreach \i/\j in {04/24,02/22,03/23}
	\draw [edge]  (\i) to (\j);

\node [rectangle] (n) at (1,-0.85){$F_{17}$};
\end{scope}

\begin{scope}[scale=0.85, xshift=3cm, yshift=-6cm]
\foreach \j in {0,...,5}
	\node [vertex] (\j) at ($ (0,3) + ({(\j*60)}:1) $)[]{};
\node [vertex] (x) at (0,3)[]{};
\node [vertex] (y) at (0,1)[]{};

\foreach \i in {2,1,5,4}
	\draw [edge]  (\i) to (3);
\foreach \i in {0,...,5}
	\draw [edge]  (\i) to (x);
\foreach \i in {1,2,4,5}
	\draw [edge]  (\i) to (0);
\foreach \i/\j in {2/1,4/5}
	\draw [edge]  (\i) to (\j);

\node [rectangle] (n) at (0,-0.85){$F_{18}$};
\end{scope}

\begin{scope}[scale=0.85, xshift=6cm, yshift=-6cm]
\foreach \i in {0,2}
	\foreach \j in {1,...,4}
		\node [vertex] (\i\j) at (\i,\j)[]{};

\foreach \i in {04,24,22,02}
	\foreach \j in {03,23}
	\draw [edge]  (\i) to (\j);

\foreach \i/\j in {04/24}
	\draw [edge]  (\i) to (\j);

\node [rectangle] (n) at (1,-0.85){$F_{19}$};
\end{scope}

\begin{scope}[scale=0.85, xshift=11cm, yshift=-6cm]
\foreach \j in {0,...,3}
	\node [vertex] (\j) at ($ (0,2.5) + ({(\j*90)}:1.3) $)[]{};
\node [vertex] (a) at (-0.5,2)[]{};
\node [vertex] (b) at (0.5,2)[]{};
\node [vertex] (c) at (0,2.5)[]{};
\node [vertex] (x) at (0,0)[]{};

\foreach \i in {0,1,2,3,a,b}
	\draw [edge]  (\i) to (c);
\foreach \i/\j in {0/1,1/2,2/3,3/0}
	\draw [edge]  (\i) to (\j);

\node [rectangle] (n) at (0,-0.85){$F_{20}$};
\end{scope}

\begin{scope}[scale=0.85, xshift=14cm, yshift=-5cm]
\foreach \i in {0,2,3}
	\node [vertex] (\i3) at (\i,3)[]{};
\foreach \j in {0,1,2}
	\node [vertex] (1\j) at (1,\j)[]{};
\foreach \j in {0,2,3}
	\node [vertex] (3\j) at (3,\j)[]{};

\foreach \i in {03,23}
	\foreach \j in {12,11,10}
		\draw [edge]  (\i) to (\j);
\foreach \i/\j in {03/23,23/33,23/32,33/32,12/11,11/10}
	\draw [edge]  (\i) to (\j);

\node [rectangle] (n) at (1.5,-1.85){$F_{25}$};
\end{scope}

\end{scope}

\end{tikzpicture}

\caption{Family $B$ of $P_4$-sparse minimal $2$-polar obstructions on 8
vertices.}
\label{fig:2obsmin8vb}
\end{figure}

\begin{proof}
	The proof of Lemma 9 in \cite{hellDAM261} is still valid for $P_4$-sparse
	graphs with the only addition of the graph $F_{25}$ as a partial complement
	of the graph $F_{19}$, which was omitted by mistake in \cite{hellDAM261}.
	The main arguments are similar to those used in the proof of
	\Cref{lem:p4sp2obsminOrder9}.
\end{proof}

We summarize the results of this section in the following theorem.

\begin{theorem}
    There are exactly 50 $P_4$-sparse minimal $2$-polar obstructions, and each
    of them is a cograph. The disconnected $P_4$-sparse minimal $2$-polar
    obstructions are the graphs $F_1, \dots, F_{25}$ depicted in
    \Crefrange{fig:2obsmin7v}{fig:2obsmin8vb}.
\end{theorem}


\subsection{\texorpdfstring{$P_4$}{P4}-extendible minimal 2-polar obstructions}

In \cite{contrerasCogGen1} it was observed that the set of cograph minimal
$(s,k)$-polar obstructions is a proper subset of the set of $P_4$-extendible
minimal $(s,k)$-polar obstructions for the cases $\min \{s,k\} = 1$ and $s = k =
\infty$. In the present section we give the complete family of $P_4$-extendible
minimal $2$-polar obstructions, and show that also in the case $s = k = 2$ there
are $P_4$-extendible minimal $(s,k)$-polar obstructions which are not cographs.
Indeed, each graph depicted in \Crefrange{fig:2obsP4ext}{fig:2obsP4ext4} is a
$P_4$-extendible minimal $2$-polar obstruction which is not a cograph.

We start by proving that there exists only one $P_4$-extendible connected
minimal $2$-polar obstruction whose complement is also a connected graph.

\begin{lemma}
\label{lem:coPspider2obsmin}
    If $G = (S, K, R)$ is a $\overline{P}$-spider and $H = G[R]$, then $G$ is a
    minimal $2$-polar obstruction if and only if $H \cong P_3$, that is, if $G$
    is isomorphic to the graph $F_{26}$ in \Cref{fig:2obsP4ext}.
\end{lemma}

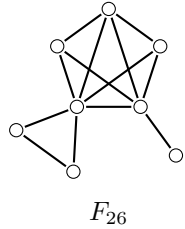
\begin{figure}[ht!]
\centering
\begin{tikzpicture}

\begin{scope}[scale=0.8]
\foreach \i in {0,...,4}
	\node [vertex] (\i) at ($ (0,1) + ({90+(\i*72)}:0.9) $)[]{};
\foreach \i in {7,9,12}
	\node [vertex] (\i) at ($ (0,1) + ({90+(\i*18)}:1.9) $)[]{};

\foreach \i in {2,3}
	\foreach \j in {0,1,4}
		\draw [edge]  (\i) to (\j);

\foreach \i/\j in {2/7,2/9,7/9,3/12,2/3,4/0,0/1}
	\draw [edge]  (\i) to (\j);
\node [rectangle] (n) at (0,-1.5){$F_{26}$};
\end{scope}

\end{tikzpicture}

\caption{A connected $P_4$-extendible minimal
$2$-polar obstruction with connected complement.}
\label{fig:2obsP4ext}
\end{figure}

\begin{proof}
	If $H \cong P_3$, then $G \cong F_{26}$, so $G$ is a minimal $2$-polar
	obstruction. Suppose for a contradiction that $G$ is another $\overline
	P$-spider minimal $2$-polar obstruction. Being $P_3$-free, $H$ is a cluster.
	Moreover, if $H$ is not a complete multipartite graph, then $G$ properly
	contains $F_3$ as an induced subgraph, which is impossible. Then $H$ is a
	cluster which is a complete multipartite graph, so it is either a complete
	or an empty graph. However, it is easy to check that in both cases $G$ is a
	$2$-polar graph, contradicting our original assumption.
\end{proof}

The proofs of the next proposition and its corollaries are very similar to the
proofs of \Cref{pro:p4sp2polSIIsplit} and its corollaries, so we only sketch
them without going into details.

\begin{proposition}
\label{pro:Char2polP4ext}
    Let $X \in \{P_4, F\}$.  If $G$ is an $X$-spider, then $G$ is a $2$-polar
    graph if and only if $R$ induces a split graph.
\end{proposition}

\begin{proof}
	Let $(S, K, R)$ be the spider partition of $G$. First, assume that $(A,B)$
	is a split partition of $G[R]$. Then, $(A \cup S, B \cup K)$ is a split
	partition of $G$, so $G$ is a split graph, and hence a $2$-polar graph. Now,
	suppose that $G$ has a $2$-polar partition $(V_1, V_2, V_3, V_4)$, and let
	$R_i = V_i \cap R$ for each $i\in\{1,\dots,4\}$. Notice that, if $R_1$ and
	$R_2$ are both nonempty, then $S\cup K \subseteq V_3\cup V_4$, which is
	impossible since $X$ is not a cluster. Analogously, since $X$ is not a
	complete multipartite graph, $R_3$ and $R_4$ cannot be both nonempty.
	Therefore $G[R]$ is a split graph.
\end{proof}

\begin{corollary}
    Let $X \in \{P_4, F\}$.   If $G$ is an $X$-spider, then it is not a minimal
    $2$-polar obstruction.
\end{corollary}

\begin{proof}
	Let $(S, K, R)$ be the spider partition of $G$. In order to reach a
	contradiction, suppose that $G$ is a minimal $2$-polar obstruction. By
	\Cref{pro:Char2polP4ext}, $G[R]$ is not a split graph, but for any vertex
	$v\in R$, $G[R]-v$ is. Hence, $G[R]$ is a minimal split obstruction, i.e.,
	$G[R]$ is isomorphic to some of $2K_2, C_4$ or $C_5$. But then, $G$ contains
	$F_3, \overline{F_3}$ or $F_{27}$, respectively, as a proper induced
  subgraph, contradicting the minimality of $G$.
\end{proof}

\begin{corollary}
    If $G$ is a $P_4$-extendible minimal $2$-polar obstruction different from
    $F_{26}$ and its complement, then $G$ or its complement is disconnected.
\end{corollary}

\begin{proof}
	It is a simple exercise to verify that any extension graph is a $2$-polar
	graph. In addition, by \Cref{lem:coPspider2obsmin,pro:Char2polP4ext}, the
	only $X$-spiders that are minimal $2$-polar obstructions are $F_{26}$ and
	its complement. Therefore, by \Cref{thm:connChar}, any other
	$P_4$-extendible minimal $2$-polar obstruction is disconnected or has a
	disconnected complement.
\end{proof}

As we did in the case of $P_4$-sparse graphs, now we characterize the
$P_4$-extendible minimal $2$-polar obstructions with the maximum possible number
of connected components. We start by quoting two useful results of
$P_4$-extendible minimal $(s,1)$-polar obstructions.

\begin{theorem}\cite{contrerasCogGen1}
\label{theo:charExt(s1)obsmin}
    Let $s$ be an integer, $s \ge 2$.   If $G$ is a $P_4$-extendible graph, then
    $G$ is a minimal $(s,1)$-polar obstruction if and only if $G$ satisfies
    exactly one of the following assertions:
    \begin{enumerate}
        \item $G$ is isomorphic to one of the seven graphs depicted in
        \Cref{fig:essentials}.

        \item $G$ is isomorphic to some of $2K_{s+1}, K_2 + (K_s \oplus 2K_1)$
            or $K_1 + (K_{s-1}\oplus C_4)$.

        \item For some nonnegative integers $s_1, s_2, \dots, s_t$ such that $s
            = t-1+ \sum_{i=1}^t s_i$, the complement of $G$ is a disconnected
            graph with components $G_1, \dots, G_t$, where each $G_i$ is a
            minimal $(1, s_i)$-polar obstruction whose complement is different
            from the graphs in \Cref{fig:essentials}.
    \end{enumerate}
\end{theorem}

\begin{corollary}
\label{cor:extE}
    There are exactly 13 $P_4$-extendible minimal $(2,1)$-polar obstructions;
    they are the graphs $E_1, \dots, E_{13}$ depicted in
    \Cref{fig:essentials,fig:P4ext(21)obsmin}.
\end{corollary}

As the reader can check, the proofs of the next proposition and its corollary
are analogous to those of \Cref{pro:techLem,cor:Kobsmink+2}.

\begin{proposition}
\label{pro:techLemExt}
    Let $\ell$ be a positive integer. If $G$ is a connected $P_4$-extendible
    minimal $(1,\ell-1)$-polar obstruction which is a complete multipartite
    graph, then $G$ is isomorphic to either $K_{\ell,\ell}$ or $K_1 \oplus C_4$.
\end{proposition}

\begin{proof}
    Clearly, if $\ell=1$, then $G \cong K_2$, while if $\ell =2$, we have $G
    \cong C_4$. By \Cref{theo:charExt(s1)obsmin}, if $\ell \ge 3$, $G$ is
    isomorphic to either $K_{\ell,\ell}$ or $K_1 \oplus C_4$.
\end{proof}

\begin{corollary}
\label{cor:Kobsmink+2P4Ext}
If $G$ is a $P_4$-extendible graph, then $G$ is a minimal $2$-polar obstruction
with exactly 4 connected components if and only if $G \cong \ell K_1 + (3 -
\ell) K_2 + K_{\ell, \ell}$ for some integer $\ell \in \{1, 2, 3\}$.
\end{corollary}

\begin{proof}
	This result represents to $P_4$-extendible graphs the same as
	\Cref{cor:Kobsmink+2} is to $P_4$-sparse graphs. In fact, the proof of this
	result is basically the same as that of \Cref{cor:Kobsmink+2}, but using
	instead \Cref{pro:techLemExt}, which is to $P_4$-extendible graphs as
	\Cref{pro:techLem} is to $P_4$-sparse graphs.
\end{proof}

By \Cref{lem:hellDAM261seven}, we have that no $P_4$-extendible minimal
$2$-polar obstruction has less than seven vertices. In the rest of the section
we give the complete list of such obstructions, obtaining as a consequence that
they have at most 9 vertices, as in the case of $P_4$-sparse graphs. We remark
that these proofs are very similar in flavor to the analogous proofs for
$P_4$-sparse graphs.

\begin{lemma}
The disconnected $P_4$-extendible minimal $2$-polar obstructions on 7 vertices
are exactly the graphs $F_1, \dots, F_5$ depicted in \Cref{fig:2obsmin7v}.
\end{lemma}

\begin{proof}
	Let $H$ be a disconnected $P_4$-extendible minimal $2$-polar obstruction on
	7 vertices. By the observation after \Cref{lem:hellDAM261partialComp}, it is
	enough to prove that $H \cong F_i$ for some $i \in \{1, \dots, 5\}$.  It
	follows from \Cref{cor:Kobsmink+2P4Ext} that, if $H$ has four components, or
	it can be transformed into a graph with four components through a sequence
	of partial complementations, then it is one of $F_1,\dots,F_5$.

	So, assume that none of the graphs that can be obtained from $H$ by means of
	partial complements has more than three connected components. Notice that
	any $P_4$-extendible graph $H$ on seven vertices with exactly two
	components, can be transformed by partial complementation into a graph with
	at least three components, one of which is an isolated vertex, except in the
	case that $H$ is the disjoint union of $K_1$ with an $X$-spider on 6
	vertices, in which case it can be checked that $H$ is a $(1,2)$-polar graph.
	Taking a partial complementation separating one isolated vertex of $H$ from
	the rest of the graph, we obtain a graph with two components, one of them
	being an isolated vertex. Let us suppose without loss of generality that $H$
	has this form.

	Since $H$ is not $2$-polar, its nontrivial component must contain a
	$P_4$-extendible minimal $(2,1)$-polar obstruction $H'$ as an induced
	subgraph. Moreover, either $H'$ has fewer than six vertices, or it has
	exactly six vertices and is connected, so it follows from \Cref{cor:extE}
	that $H' \in \{K_1 + 2K_2, \overline{3K_2}, \overline{K_2 + C_4}\}$. If $H'
	\cong \overline{3K_2}$, then $H$ is the graph $F_5$ in \Cref{fig:2obsmin7v}.
	If $H' \cong \overline{K_2 + C_4}$, it is straightforward to verify that $H$
	is a $(1,2)$-polar graph. Otherwise, $H' \cong K_1 + 2K_2$. But $H'$ is
	contained in a connected component of $H$ on six vertices, which must be
	isomorphic to $K_1 \oplus (K_1 + 2K_2)$ because $H$ is a $P_4$-extendible
	graph. Then, $H$ is isomorphic to $F_3$.
\end{proof}

The next technical lemma will be needed to give the complete list of
$P_4$-extendible minimal $2$-polar obstructions with at least eight vertices.

\begin{lemma}
\label{lem:oneNoCog}
    Let $H$ be a disconnected minimal $2$-polar obstruction. If $H$ has a
    component $H'$ which is not a cograph, then $H-H'$ is a split graph. In
    consequence, at most one component of $H$ is not a cograph.
\end{lemma}

\begin{proof}
	If $H-H'$ is not a split graph it contains $2K_2,C_4$ or $C_5$ as an induced
	subgraph, and $H$ would contain $F_1,F_2$ or $F_{29}$ as a proper induced
	subgraph, respectively (see \Cref{fig:2obsmin7v,fig:2obsP4ext2}). Now,
	assume for a contradiction that $H$ has at least two components, $H_1$ and
	$H_2$, which are not cographs. By the first part of this lemma, $H-H_1$ and
	$H-H_2$ (and hence $H_1$) are split graphs, so $H$ is the disjoint union of
	two split graphs, which implies that it is a $(1,2)$-polar graph,
	contradicting that $H$ is a $2$-polar obstruction.
\end{proof}

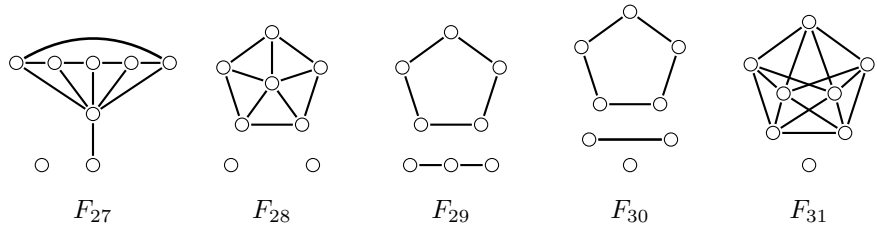
\begin{figure}[ht!]
\centering
\begin{tikzpicture}

\begin{scope}[scale=0.8]

\begin{scope}[scale=0.85]
\foreach \i in {0,...,4}
	\node [vertex] (\i) at ({(\i-2)*0.75},2)[]{};
\node [vertex] (5) at (0,1)[]{};
\node [vertex] (6) at (0,0)[]{};
\node [vertex] (7) at (-1,0)[]{};

\foreach \i in {0,...,3}
	\draw let \n1={int(\i+1)} in [edge] (\i) to node [above] {} (\n1);

\draw [curveEdge]  (0) to (4);

\foreach \i in {0,...,4}
	\draw [edge]  (\i) to (5);

\draw [edge]  (6) to (5);

\node [rectangle] (n) at (0,-0.9){$F_{27}$};
\end{scope}

\begin{scope}[scale=0.85, xshift=3.5cm]
\foreach \i in {0,...,4}
	\node [vertex] (\i) at ($ (0,1.6) + ({90+(\i*72)}:1) $)[]{};
\node [vertex] (c) at (0,1.6)[]{};
\node [vertex] (a) at (-0.8,0)[]{};
\node [vertex] (b) at (0.8,0)[]{};

\foreach \i in {0,...,4}
	\draw let \n1={int(mod(\i+1,5))} in [edge] (\i) to node [above] {} (\n1);

\foreach \i in {0,...,4}
	\draw [edge]  (\i) to (c);

\node [rectangle] (n) at (0,-0.9){$F_{28}$};
\end{scope}

\begin{scope}[scale=0.85, xshift=7cm]
\foreach \i in {0,...,4}
	\node [vertex] (\i) at ($ (0,1.6) + ({90+(\i*72)}:1) $)[]{};
\node [vertex] (x) at (-0.8,0)[]{};
\node [vertex] (y) at (0,0)[]{};
\node [vertex] (z) at (0.8,0)[]{};

\foreach \i in {0,...,4}
	\draw let \n1={int(mod(\i+1,5))} in [edge] (\i) to node [above] {} (\n1);

\foreach \i/\j in {x/y,y/z}
	\draw [edge]  (\i) to (\j);

\node [rectangle] (n) at (0,-0.9){$F_{29}$};
\end{scope}

\begin{scope}[scale=0.85, xshift=10.5cm]
\foreach \i in {0,...,4}
	\node [vertex] (\i) at ($ (0,2) + ({90+(\i*72)}:1) $)[]{};
\node [vertex] (c) at (0,0)[]{};
\node [vertex] (a) at (-0.8,0.5)[]{};
\node [vertex] (b) at (0.8,0.5)[]{};

\foreach \i in {0,...,4}
	\draw let \n1={int(mod(\i+1,5))} in [edge] (\i) to node [above] {} (\n1);

\draw [edge]  (a) to (b);

\node [rectangle] (n) at (0,-0.9){$F_{30}$};
\end{scope}

\begin{scope}[scale=0.85, xshift=14cm]
\foreach \i in {0,...,4}
	\node [vertex] (\i) at ($ (0,1.6) + ({90+(\i*72)}:1.2) $)[]{};
\foreach \i in {5,6}
	\node [vertex] (\i) at ($ (0,1.4) + ({(\i*180)}:0.5) $)[]{};
\node [vertex] (c) at (0,0)[]{};

\foreach \i in {0,...,4}
	\draw let \n1={int(mod(\i+1,5))} in [edge] (\i) to node [above] {} (\n1);

\foreach \i in {0,...,4}
	\foreach \j in {5,6}
		\draw [edge] (\i) to node [above] {} (\j);

\draw [edge]  (a) to (b);

\node [rectangle] (n) at (0,-0.9){$F_{31}$};
\end{scope}

\end{scope}

\end{tikzpicture}

\caption{Family $C$ of $P_4$-extendible minimal $2$-polar obstructions on $8$
  vertices.}
\label{fig:2obsP4ext2}
\end{figure}

In the proof of our following lemma, we will implicitly use
\cref{lem:hellDAM261partialComp}, and the observation right after it, when
analyzing which minimal obstructions appear in the different cases.   Hence, it
is natural that at most one graph from each figure appears in the proof.   For
example, since $F_{13}$ appears as an induced subgraph in one of the cases, then
none of the graphs $F_{14}, \dots, F_{25}$ (see \cref{fig:2obsmin8vb}) will be
explicitly mentioned in the proof.

\begin{lemma}
\label{lem:noBig2obminP4ext}
    The only disconnected $P_4$-extendible minimal $2$-polar obstructions with
    at least 8 vertices are the graphs $F_6, F_7, \dots, F_{25}, F_{27}, F_{28},
    \dots, F_{41}$ depicted in
    \Cref{fig:2obsmin9v,fig:2obsmin8va,fig:2obsmin8vb,%
    fig:2obsP4ext2,fig:2obsP4ext3,fig:2obsP4ext4}.
\end{lemma}

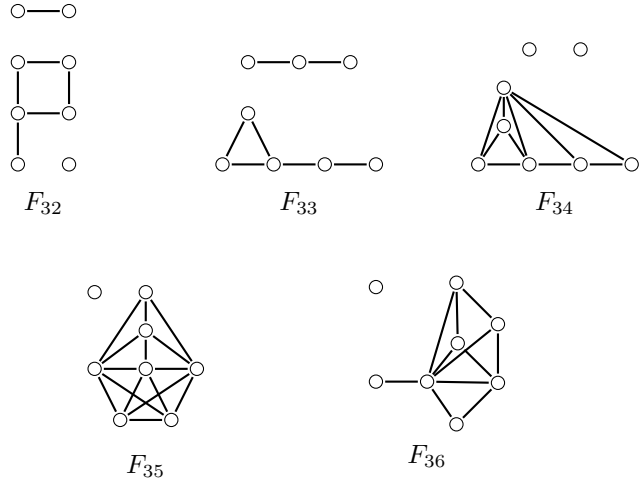
\begin{figure}[ht!]
\centering
\begin{tikzpicture}

\begin{scope}[scale=0.8]

\begin{scope}[scale=0.85, xshift=-0.5cm]
\foreach \i in {0,1}
	\foreach \j in {0,...,3}
		\node [vertex] (\i\j) at (\i,\j)[]{};

\foreach \i/\j in {03/13,02/12,01/11,00/01,01/02,11/12}
	\draw [edge] (\i) to node [above] {} (\j);

\node [rectangle] (n) at (0.5,-0.75){$F_{32}$};
\end{scope}

\begin{scope}[scale=0.85, xshift=4cm]
\foreach \i in {0,...,2}
	\node [vertex] (\i2) at (\i,2)[]{};
\foreach \i in {0,...,3}
	\node [vertex] (\i0) at (\i-0.5,0)[]{};
\node [vertex] (01) at (0,1)[]{};

\foreach \i/\j in {02/12,12/22,00/10,10/20,20/30,01/00,01/10}
	\draw [edge] (\i) to node [above] {} (\j);

\node [rectangle] (n) at (1,-0.75){$F_{33}$};
\end{scope}

\begin{scope}[scale=0.85, xshift=9cm]
\foreach \i in {0,1}
	\node [vertex] (\i3) at (\i+0.5,2.25)[]{};
\foreach \i in {0,...,3}
	\node [vertex] (\i0) at (\i-0.5,0)[]{};
\node [vertex] (01) at (0,0.75)[]{};
\node [vertex] (02) at (0,1.5)[]{};

\foreach \i/\j in {00/10,10/20,20/30,01/00,01/10}
	\draw [edge] (\i) to node [above] {} (\j);
\foreach \i in {00,10,20,30,01}
	\draw [edge] (\i) to node [above] {} (02);

\node [rectangle] (n) at (1,-0.75){$F_{34}$};
\end{scope}

\begin{scope}[scale=0.85, xshift=0.5cm, yshift=-5cm]
\node [vertex] (10) at (1,0)[]{};
\node [vertex] (20) at (2,0)[]{};
\node [vertex] (01) at (0.5,1)[]{};
\node [vertex] (11) at (1.5,1)[]{};
\node [vertex] (21) at (2.5,1)[]{};
\node [vertex] (12) at (1.5,1.75)[]{};
\node [vertex] (13) at (1.5,2.5)[]{};
\node [vertex] (x) at (0.5,2.5)[]{};

\foreach \i in {01,21}
	\foreach \j in {13,12,11,10,20}
		\draw [edge] (\i) to node [above] {} (\j);
\foreach \i/\j in {10/20,10/11,20/11,11/12,12/13}
	\draw [edge] (\i) to node [above] {} (\j);

\node [rectangle] (n) at (1.5,-0.9){$F_{35}$};
\end{scope}

\begin{scope}[scale=0.85, xshift=7.5cm, yshift = -4.7cm]
\foreach \i in {0,1,7,6}
	\node [vertex] (\i) at ($ (0,1) + ({22.5+(\i*45)}:1.5) $)[]{};
\node [vertex] (00) at (0.6,1.2)[]{};
\node [vertex] (u) at (0,0.45)[]{};
\node [vertex] (h) at (-1,0.45)[]{};
\node [vertex] (a) at (-1,2.3)[]{};

\foreach \j in {h,00,1,0,7,6}
	\draw [edge] (u) to node [above] {} (\j);
\foreach \i/\j in {7/00,00/1,1/0,0/7,7/6}
	\draw [edge] (\i) to node [above] {} (\j);

\node [rectangle] (n) at (0,-1){$F_{36}$};
\end{scope}

\end{scope}

\end{tikzpicture}

\caption{Family $D$ of $P_4$-extendible minimal $2$-polar obstructions on $8$
  vertices.}
\label{fig:2obsP4ext3}
\end{figure}

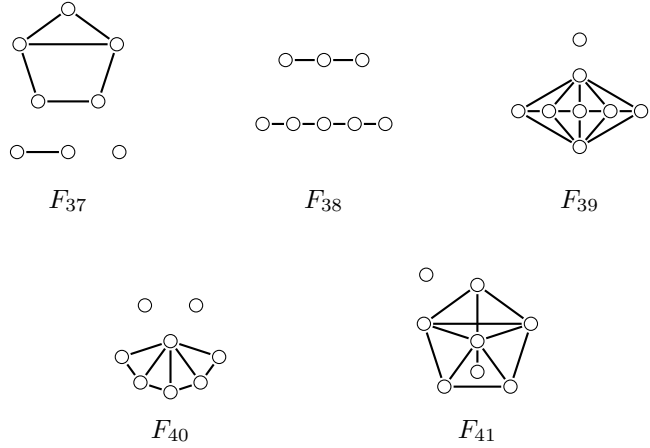
\begin{figure}[ht!]
\centering
\begin{tikzpicture}

\begin{scope}[scale=0.8]

\begin{scope}[scale=0.85]
\foreach \i in {0,...,4}
	\node [vertex] (\i) at ($ (0,2) + ({90+(\i*72)}:1) $)[]{};
\foreach \i in {0,...,2}
	\node [vertex] (\i0) at (\i-1,0.2)[]{};

\foreach \i in {0,...,4}
	\draw let \n1={int(mod(\i+1,5))} in [edge] (\i) to node [above] {} (\n1);
\foreach \i/\j in {00/10,1/4}
	\draw [edge] (\i) to node [above] {} (\j);

\node [rectangle] (n) at (0,-0.75){$F_{37}$};
\end{scope}

\begin{scope}[scale=0.85, xshift=5cm]
\foreach \i in {0,...,2}
	\node [vertex] (\i) at ({(\i*0.75)-0.75},2)[]{};
\foreach \i in {3,...,7}
	\node [vertex] (\i) at ({(\i*0.6)-3},0.75)[]{};

\foreach \i in {0,1,3,4,5,6}
	\draw let \n1={int(\i+1)} in [edge] (\i) to node [above] {} (\n1);

\node [rectangle] (n) at (0,-0.75){$F_{38}$};
\end{scope}

\begin{scope}[scale=0.85, xshift=10cm]
\node [vertex] (i) at (0,2.4)[]{};
\node [vertex] (t) at (0,1.7)[]{};
\node [vertex] (b) at (0,0.3)[]{};
\foreach \i in {0,...,4}
	\node [vertex] (\i) at ({(\i*0.6)-1.2},1)[]{};

\foreach \i in {0,...,3}
	\draw let \n1={int(\i+1)} in [edge] (\i) to node [above] {} (\n1);
\foreach \i in {t,b}
	\foreach \j in {0,...,4}
		\draw [edge] (\i) to node [above] {} (\j);

\node [rectangle] (n) at (0,-0.75){$F_{39}$};
\end{scope}

\begin{scope}[scale=0.85, xshift=2cm, yshift = -4.5cm]
\foreach \i in {5,...,9}
	\node [vertex] (\i) at ($ (0,1) + ({18+(\i*36)}:1) $)[]{};
\node [vertex] (a) at (-0.5,1.7)[]{};
\node [vertex] (b) at (0.5,1.7)[]{};
\node [vertex] (o) at (0,1)[]{};

\foreach \i in {5,...,8}
	\draw let \n1={int(mod(\i+1,10))} in [edge] (\i) to node [above] {} (\n1);
\foreach \i in {5,...,9}
	\draw [edge] (\i) to node [above] {} (o);

\node [rectangle] (n) at (0,-0.75){$F_{40}$};
\end{scope}

\begin{scope}[scale=0.85, xshift=8cm, yshift = -4.5cm]
\foreach \i in {0,...,4}
	\node [vertex] (\i) at ($ (0,1) + ({90+(\i*72)}:1.1) $)[]{};
\node [vertex] (a) at (-1,2.3)[]{};
\node [vertex] (c) at (0,1)[]{};
\node [vertex] (b) at (0,0.4)[]{};

\foreach \i in {0,...,4}
	\draw let \n1={int(mod(\i+1,5))} in [edge] (\i) to node [above] {} (\n1);
\foreach \i/\j in {1/4}
	\draw [edge] (\i) to node [above] {} (\j);
\foreach \i in {0,1,2,3,4,b}
	\draw [edge] (\i) to node [above] {} (c);

\node [rectangle] (n) at (0,-0.75){$F_{41}$};
\end{scope}

\end{scope}

\end{tikzpicture}

\caption{Family $E$ of $P_4$-extendible minimal $2$-polar obstructions on $8$
  vertices.}
\label{fig:2obsP4ext4}
\end{figure}

\begin{proof}
	Let $H$ be a $P_4$-extendible disconnected minimal $2$-polar obstruction
	with at least eight vertices. By the observation after
	\Cref{lem:hellDAM261partialComp}, it is enough to prove that $H \cong F_i$
	for some $i \in \{6, \dots, 41\}$, $i \ne 26$. If $H$ can be transformed by
	means of partial complementations into a graph with four connected
	components, we have by \Cref{cor:Kobsmink+2P4Ext} that $H$ is one of
	$F_{13},\dots,F_{25}$.

	Now, assume that $H$ can be transformed by partial complementations into a
	graph $H'$ with three components, but it cannot be transformed into a graph
	with four connected components. Notice that at least one component of $H'$
	is a cograph, otherwise $3P_4$ is an induced subgraph of $H'$, but $F_1$ is
	a proper induced subgraph of $3P_4$, contradicting that $H$ is a minimal
	$2$-polar obstruction. Having a cograph component, $H'$ can be transformed
	by a finite sequence of partial complementations into a graph $H''$ with
	three connected components where at least one of them, $B_3$, is a trivial
	component. Moreover, since $H''$ is also a minimal $2$-polar obstruction,
	$H''-B_3$ is $2$-polar but it is neither a $(2,1)$- nor a $(1,2)$-polar
	graph. Therefore, a component $B_2$ of $H''-B_3$, is a complete graph while
	its other component, $B_1$, is a $(2,1)$-polar graph that is neither a split
	nor a complete bipartite graph. Without loss of generality we can assume
	that $B_1,B_2$ and $B_3$ are the components of $H$ itself. Denote by $m$ the
	order of $B_2$.

	Suppose first that $m \ge 2$. Since $B_1$ is not a split graph, then it
	contains some of $2K_2, C_5$ or $C_4$ as an induced subgraph. If $2K_2 \le
	B_1$, then $H$ properly contains a copy of $F_1$, while if $C_5 \le B_1$,
	then $H$ must be isomorphic to $F_{30}$. Otherwise, $B_1$ contains a copy
	$C$ of $C_4$. Observe that if $B_1$ contains $K_1+C_4$ as an induced
	subgraph, then $H$ properly contains a copy of $F_{13}$, which is
	impossible. Hence, any vertex in $B_1$ not in $C$ is adjacent to some vertex
	of $C$. Let $u$ be a vertex in $B_1$ not in $C$. If $u$ is adjacent to
	exactly one vertex of $C$, then $H \cong F_{32}$; if $u$ is adjacent to two
	adjacent vertices of $C$, then $H \cong F_{37}$; if $u$ is adjacent to
	exactly three vertices of $C$, $H \cong F_7$; and if $u$ is adjacent to all
	vertices of $C$, then $H$ properly contains a copy of $F_4$. Thus, if $H$ is
	none of the graphs mentioned before, any vertex $u$ in $B_1$ not in $C$ is
	adjacent to two antipodal vertices in $C$. In addition, two vertices
	adjacent to the same pair of antipodal vertices cannot be adjacent to each
	other, otherwise $H$ contains $F_7$ as a proper induced subgraph.
	Furthermore, any two vertices adjacent to distinct pairs of antipodal
	vertices in $C$ must be adjacent to each other, or $H$ would contain
	$F_{32}$ as a proper induced subgraph. It is easy to observe that under such
	restrictions $B_1$ is a complete bipartite graph, which is impossible.

	Now let us consider the case $m=1$. We have that $B_1$ is a connected
	$P_4$-extendible graph with at least six vertices, so $B_1$ is either an
	$X$-spider or the join of two smaller $P_4$-extendible graphs. Suppose first
	that $B_1$ is an $X$-spider and let $R$ be its head. If $R$ contains
	$2K_2,C_4$ or $C_5$ as an induced subgraph, then $H$ properly contains $F_3,
	F_4$ or $F_{28}$, respectively, but this is impossible. Then, $R$ is a split
	graph, which implies that $X \notin \{P_4, F, \overline{F}\}$, or $H$ would
	be a split graph. We can assume that $X = \overline P$. If $R$ contains an
	induced $P_3$, then $H$ properly contains an induced copy of $F_{26}$, so
	$R$ must be a cluster. Hence, $R$ is a split graph which is a cluster, so $R
	= K_a + bK_1$ for some nonnegative integers $a$ and $b$. Observe that $a \ge
	2$ and $b \ge 1$, otherwise $H$ is a $2$-polar graph or it contains $F_9$ as
	a proper induced subgraph. Then, $R$ contains an induced copy of
	$\overline{P_3}$, but this implies that $H$ has a proper induced copy of
	$F_3$. Hence, $B_1$ is not an $X$-spider, so $B_1$ is the join of two
	smaller $P_4$-extendible graphs, $T_1$ and $T_2$, and hence $H = T_1 \oplus
	T_2 + B_2 + B_3$. If the complement of $T_i$ is disconnected for some $i \in
	\{1, 2\}$, then $\overline{B_1} + \overline{B_2 + B_3}$ has four connected
	components, a contradiction. Then each $T_i$ has a connected complement, so
	it is isomorphic to $K_1$ or it contains $P_4$ as an induced subgraph.
	Evidently, at least one of $T_1$ and $T_2$ is a nontrivial graph. First
	assume, without loss of generality, that $T_1$ is an isolated vertex, then
	$\overline{B_1} + \overline{B_2+B_3}$ has three connected components, one of
	them isomorphic to $K_2$, and other isomorphic to $K_1$, so we are in the
	case $m=2$. Otherwise, each of $T_1$ and $T_2$ contain an induced copy of
	$P_4$, so $\overline{B_1} + \overline{B_2+B_3}$ contains $F_1$ as a proper
	induced subgraph, which is impossible.

	Finally, assume that $H$ cannot be transformed by partial complementations
	into a graph with at least three connected components. Thus, $H$ has exactly
	two connected components, and the complement of any of them is also a
	connected graph. Then, by \Cref{lem:oneNoCog}, $H$ is the disjoint union of
	$K_1$ and an $X$-spider, but exactly as in the case $m=1$, it can be proved
	that this is impossible for a $P_4$-extendible minimal $2$-polar
	obstruction.
\end{proof}

We summarize the results of this section in the following theorem.

\begin{theorem}
  There are exactly 82 $P_4$-extendible minimal $2$-polar obstructions, they are
  the graphs $F_1,\dots,F_{41}$ and their complements.
\end{theorem}

\section{Largest polar subgraphs}
\label{sec:maxSubgraphs}

In this section, we give algorithms to find maximum order induced subgraphs with
some given properties (related to polarity) in $P_4$-sparse and $P_4$-extendible
graphs using their tree representations. Ekim, Mahadev and de Werra
\cite{ekimDAM156} previously obtained similar results for cographs using the
cotree. Given a graph $G$, we denote by $\mathsf{MC}(G), \mathsf{MI}(G)$, and
$\mathsf{MS}(G)$ a maximum subset of $V_G$ inducing a complete graph, an empty
graph, and a split graph, respectively. We use $\mathsf{MB}(G)$ and
$\mathsf{McB}(G)$ to denote a maximum subset of $V_G$ inducing a bipartite and a
co-bipartite graph, respectively. We also use $\mathsf{MUC}(G)$ and
$\mathsf{MJI}(G)$ to denote maximum subsets of $V_G$ inducing a cluster and a
complete multipartite graph, respectively; $\mathsf{MM}(G), \mathsf{McM}(G)$,
and $\mathsf{MP}(G)$ stand for maximum subsets of $V_G$ inducing a monopolar, a
co-monopolar and a polar subgraph of $G$, while $\mathsf{MU}(G)$ and
$\mathsf{McU}(G)$ are used for denoting maximum subsets of $V_G$ inducing a
unipolar or a co-unipolar graph, respectively. To simplify the notation, when we
are working with preset subgraphs $G_i$ of $G$, we write $\mathsf{MC_i}$ instead
of $\mathsf{MC}(G_i)$ and, if there is no possibility of confusion, we write
$\mathsf{MC}$ instead of $\mathsf{MC}(G)$; we use an analogous notation for all
other maximal subgraphs. Given a family $\mathcal{F}$ of subsets of $V_G$, a
\textit{witness of} $M = \max_{F \in \mathcal{F}} \{ \lvert F \rvert \}$
\textit{in $\mathcal{F}$} is an element $F'$ of $\mathcal{F}$ such that $\lvert
F' \rvert = M$.

The following proposition provides recursive characterizations for the
aforementioned maximum subgraphs in a disconnected graph.

\begin{proposition}
\label{prop:maxUnion}
Let $G = G_0 + G_1$ be a graph, and let $W$ be a subset of $V_G$. The following
statements hold true.
\begin{enumerate}
  \item $W$ is a maximum clique of $G$ if and only if $W$ is a witness of
  $\max \{ \lvert \mathsf{MC_0} \rvert, \lvert \mathsf{MC_1} \rvert \}$.

  \item $W$ is a maximum independent set of $G$ if and only if $W$ is a witness
  of $\max \{ \lvert \mathsf{MI_0} \cup \mathsf{MI_1} \rvert \}$.

  \item $W$ induces a maximum bipartite subgraph of $G$ if and only if $W$ is a
  witness of $\max \{ \lvert \mathsf{MB_0} \cup \mathsf{MB_1} \rvert \}$.

  \item $W$ induces a maximum co-bipartite subgraph of $G$ if and only if $W$ is
  a witness of
  \[
    \max \{ \lvert \mathsf{McB_0} \rvert,
    \lvert \mathsf{McB_1} \rvert,
    \lvert \mathsf{MC_0} \cup \mathsf{MC_1} \rvert \}.
  \]

  \item $W$ induces a maximum split subgraph of $G$ if and only if $W$ is a
  witness of $\max_{i\in\{0,1\}} \{ \lvert \mathsf{MI_i} \cup \mathsf{MS_{1-i}}
  \rvert \}$.

  \item $W$ induces a maximum cluster in $G$ if and only if $W$ is a witness of
  $\max \{ \lvert \mathsf{MUC_0} \cup \mathsf{MUC_1} \rvert \}$.

  \item $W$ induces a maximum complete multipartite subgraph of $G$ if and only
  if $W$ is a witness of
  \[
    \max \{ \lvert \mathsf{MI} \rvert,
    \lvert \mathsf{MJI_0} \rvert,
    \lvert \mathsf{MJI_1} \rvert \}.
  \]

  \item $W$ induces a maximum monopolar subgraph of $G$ if and only if $W$ is a
  witness of $\max \{ \lvert \mathsf{MM_0} \cup \mathsf{MM_1} \rvert \}$.

  \item $W$ induces a maximum co-monopolar subgraph of $G$ if and only if $W$ is
  a witness of
  \[
    \max_{i\in\{0,1\}} \{ \lvert \mathsf{MS_i} \cup \mathsf{MI_{1-i}} \rvert,
    \lvert \mathsf{McM_i} \rvert,
    \lvert \mathsf{MC_i} \cup \mathsf{MJI_{1-i}} \rvert \}.
  \]

  \item $W$ induces a maximum polar subgraph of $G$ if and only if $W$ is a
  witness of
  \[
    \max \{ \lvert \mathsf{MM} \rvert,
    \lvert \mathsf{MP_0} \cup \mathsf{MUC_{1}} \rvert,
    \lvert \mathsf{MP_1} \cup \mathsf{MUC_{0}} \rvert \}.
  \]

  \item $W$ induces a maximum unipolar subgraph of $G$ if and only if $W$ is a
    witness of
    \[
      \max_{i\in\{0,1\}} \{ \lvert \mathsf{MU_i} \cup \mathsf{MUC_{1-i}} \rvert,
      \lvert \mathsf{MU_{1-i}} \cup \mathsf{MUC_i} \rvert \}.
    \]

  \item $W$ induces a maximum co-unipolar subgraph of $G$ if and only if $W$ is
    a witness of
    \[
      \max \{ \lvert \mathsf{MB} \rvert,
      \lvert \mathsf{MI_0} \cup \mathsf{McU_1} \rvert,
      \lvert \mathsf{MI_1} \cup \mathsf{McU_0} \rvert \}.
    \]
\end{enumerate}
\end{proposition}

\begin{proof}
  In this paper, we present several propositions whose proofs are based on
  similar ideas to those used in the current one. Nonetheless, for the sake of
  completeness, we include the entire proof for all such propositions. In this
  particular proof, we consider that property 9 is a good example of the general
  arguments used for proving the entire statement.

  \begin{enumerate}
    \item Let $W$ be a maximum clique of $G$. Evidently, for some $i\in \{0,
        1\}$, $W \cap V_{G_i} = \varnothing$ and $W \cap V_{G_{1-i}}$ is a
        clique of $G_{1-i}$. It follows that $W$ is a maximum clique for either
        $G_0$ or $G_1$ such that $\lvert W \rvert = \max \{ \lvert \mathsf{MC}_0
        \rvert, \lvert \mathsf{MC}_1 \rvert \}$.

    \item Let $W$ be a maximum independent set of $G$. Clearly, $W \cap V_{G_i}$
        is an independent set of $G_i$ for each $i \in \{0, 1\}$. It follows
        that $W$ is the union of a maximum independent set of $G_0$ with a
        maximum independent set of $G_1$.

    \item Let $W$ be a set inducing a maximum bipartite subgraph of $G$. For
        each $i \in \{0,1\}$, $G[W \cap V_{G_i}]$ is a bipartite graph, and the
        disjoint union of two bipartite graphs clearly is a bipartite graph, so
        the result follows.

    \item Let $W$ be a set inducing a maximum co-bipartite subgraph of $G$, and
        let $(A, B)$ be a partition of $W$ into two cliques. Clearly, each of
        $A$ and $B$ is completely contained in one of $V_{G_1}$ or $V_{G_2}$. If
        both $A$ and $B$ are contained in $V_{G_i}$ for some $i \in \{0,1\}$,
        then $W$ induces a maximum co-bipartite subgraph of $G_i$. Otherwise, $A
        \subseteq V_{G_i}$ and $B \subseteq V_{G_{1-i}}$ for some $i \in
        \{0,1\}$, so $G[A]$ is a maximum clique in $G_i$ and $G[B]$ is a maximum
        clique in $G_{1-i}$. The result easily follows from here.

    \item Let $W$ be a set inducing a maximum split subgraph of $G$, and let
        $(A, B)$ be a split partition of $G[W]$. Since $B$ is a clique, $B$ is
        contained in either $V_{G_0}$ or $V_{G_1}$. Hence, for some $i \in \{0,
        1\}$, $W \cap V_{G_i}$ induces a split graph while $W \cap V_{G_{1-i}}$
        is an independent set. It follows that $W = V_i \cup V_{1-i}$, where
        $V_i$ is a subset of $V_{G_i}$ inducing a maximum split graph, $V_{1-i}$
        is a maximum independent subset of $V_{G_{1-i}}$, and $\lvert W \rvert =
        \max_{i\in\{0,1\}} \{ \lvert \mathsf{MI}_i \cup \mathsf{MS}_{1-i} \rvert
        \}$.

    \item Let $W$ be a set inducing a maximum cluster of $G$. Clearly, for each
        $i \in \{0, 1\}$, $W \cap V_{G_i}$ induces a cluster. It follows that
        $W$ is the union of a set inducing a maximum cluster of $G_0$ with a set
        inducing a maximum cluster of $G_1$.

    \item Let $W$ be a set inducing a maximum complete multipartite subgraph of
        $G$. If $W$ is an independent set, it is evidently a maximum independent
        set of $G$. Otherwise, $G[W]$ is a connected graph, so $W$ is completely
        contained in $V_{G_i}$ for some $i \in \{0, 1\}$, and therefore, $W$
        induces a maximum complete multipartite subgraph of $G_i$. In any case
        we have that $\lvert W \rvert = \max \{ \lvert \mathsf{MI} \rvert,
        \lvert \mathsf{MJI}_0 \rvert, \lvert \mathsf{MJI}_1 \rvert \}$.

    \item Let $W$ be a set inducing a maximum monopolar subgraph of $G$.
        Evidently, for any $i \in \{0, 1\}$, $W \cap V_{G_i}$ induces a
        monopolar graph, so we have that $W$ is the union of a set inducing a
        maximum monopolar subgraph of $G_0$ with a set inducing a maximum
        monopolar subgraph of $G_1$.

    \item Let $W$ be a set inducing a maximum co-monopolar subgraph	of $G$, and
        let $(A, B)$ be a partition of $W$ such that $A$ induces a complete
        multipartite graph and $B$ is a clique. Since $B$ is a clique, it is
        completely contained in either $V_{G_0}$ or $V_{G_1}$. Now, if $A$ is an
        independent set, then $W = V_i \cup V_{1-i}$ for some $i \in \{0, 1\}$,
        where $V_i$ induces a maximum split subgraph of $G_i$ and $V_{1-i}$
        induces a maximum independent set of $G_{1-i}$. Otherwise, if $A$ is not
        an independent set, it induces a connected graph and is contained in
        either $V_{G_0}$ or $V_{G_1}$; hence, either $W$ induces a maximum
        co-monopolar subgraph of $G_i$ for some $i \in \{0, 1\}$, or there
        exists $i \in \{0,1\}$ such that $W$ is the union of a maximum clique in
        $G_i$ and a set inducing a maximum complete multipartite subgraph of
        $G_{1-i}$.

    \item Let $W$ be a set inducing a maximum polar subgraph of $G$, and let
        $(A, B)$ be a polar partition of $G[W]$. If $A$ is an independent set,
        then $W \cap V_{G_i}$ induces a monopolar subgraph of $G_{i}$ for each
        $i \in \{0, 1\}$, so $W$ induces a maximum monopolar subgraph of $G$.
        Otherwise, if $A$ is not an independent set, $G[A]$ is connected and $A$
        is completely contained in $V_{G_i}$ for some $i \in \{0, 1\}$; hence,
        $W$ is the union of a set inducing a maximum polar subgraph of $G_i$
        with a set inducing a maximum cluster of $G_{1-i}$.

    \item Let $W$ be a set inducing a maximum unipolar subgraph of $G$, and let
        $(A, B)$ be a unipolar partition of $G[W]$. Since $A$ is a clique, it is
        completely contained in $V_{G_i}$ for some $i\in\{0,1\}$. Thus, $W \cap
        V_{G_{1-i}}$ induces a cluster and $W\cap V_{G_i}$ induces a unipolar
        graph, so $W$ is the union of a set inducing a maximum unipolar subgraph
        of $G_i$ with a set inducing a maximum cluster in $G_{1-i}$.

    \item Let $W$ be a set inducing a maximum co-unipolar subgraph of $G$, and
        let $(A,B)$ be a unipolar partition of $\overline{G[W]}$. Since $G[B]$
        is a complete multipartite graph, if $B \cap V_{G_1} \ne \varnothing$
        and $B \cap V_{G_2} \ne \varnothing$, $B$ is an independent set, so $W$
        induces a bipartite graph. Otherwise, $B \cap V_{G_i} = \varnothing$ for
        some $i\in\{0,1\}$, and we have that $W \cap V_{G_i}$ is an independent
        set and $W \cap V_{G_{1-i}}$ induces a co-unipolar graph. The result
        follows easily from here.
  \end{enumerate}
\end{proof}

Since $G \oplus H = \overline{\overline{G} + \overline{H}}$ for any pair of
graphs $G$ and $H$, the following statement is an immediate consequence of the
previous proposition, so we omit the proof. Notice that, by \Cref{theo:cogChar},
\Cref{prop:maxUnion,prop:maxJoin} can be used together in a mutual recursive
algorithm to determine the maximum subgraphs listed in them for any cograph.

\begin{proposition}
\label{prop:maxJoin}
Let $G = G_0 \oplus G_1$ be a graph, and let $W$ be a subset of $V_G$. The
following statements hold true.
\begin{enumerate}
  \item $W$ is a maximum clique of $G$ if and only if $W$ is a witness of
    $\max \{ \lvert \mathsf{MC}_0 \cup \mathsf{MC}_1 \rvert \}$.

  \item $W$ is a maximum independent set of $G$ if and only if $W$ is a witness
    of $\max \{ \lvert \mathsf{MI_0} \rvert, \lvert \mathsf{MI_1} \rvert \}$.

  \item $W$ induces a maximum bipartite subgraph of $G$ if and only if $W$ is a
    witness of \( \max \{ \lvert \mathsf{MB_0} \rvert, \lvert \mathsf{MB_1}
    \rvert,$ $\lvert \mathsf{MI_0} \cup \mathsf{MI_1} \rvert \}. \)

  \item $W$ induces a maximum co-bipartite subgraph of $G$ if and only if $W$ is
    a witness of $\max \{ \lvert \mathsf{McB_0} \cup \mathsf{McB_1} \rvert \}$.

  \item $W$ induces a maximum split subgraph of $G$ if and only if, $W$ is a
    witness of $\max_{ i \in \{0,1\}} \{ \lvert \mathsf{MC_i} \cup
    \mathsf{MS_{1-i}} \rvert \}$.

  \item $W$ induces a maximum cluster in $G$ if and only if $W$ is a witness of
    $\max \{ \lvert \mathsf{MC} \rvert, \lvert \mathsf{MUC_0} \rvert, \lvert
    \mathsf{MUC_1} \rvert \}$.

  \item $W$ induces a maximum complete multipartite graph of $G$ if and only if
    $W$ is a witness of
    \[
        \max \{ \lvert \mathsf{MUI_0} \cup \mathsf{MUI_1} \rvert \}.
    \]

  \item $W$ induces a maximum monopolar subgraph of $G$ if and only if $W$ is a
    witness of
    \[
      \max_{i\in\{0, 1\}} \{ \lvert \mathsf{MS_i} \cup \mathsf{MC_{1-i}} \rvert,
      \lvert \mathsf{MM_i} \rvert,
      \lvert \mathsf{MI_i} \cup \mathsf{MUC_{1-i}} \rvert \}.
    \]

  \item $W$ induces a maximum co-monopolar subgraph of $G$ if and only if $W$ is
    a witness of $\max \{ \lvert \mathsf{McM_0} \cup \mathsf{McM_1} \rvert \}$.

  \item $W$ induces a maximum polar subgraph of $G$ if and only if $W$ is a
    witness of
    \[
      \max \{ \lvert \mathsf{McM} \rvert,
      \lvert \mathsf{MP_0} \cup \mathsf{MJI_{1}} \rvert,
      \lvert \mathsf{MP_1} \cup \mathsf{MJI_{0}} \rvert \}.
    \]

  \item $W$ induces a maximum unipolar subgraph of $G$ if and only if $W$ is a
    witness of
    \[
      \max \{ \lvert \mathsf{McB} \rvert,
      \lvert \mathsf{MU_1} \cup \mathsf{MC_0} \rvert,
      \lvert \mathsf{MU_0} \cup \mathsf{MC_1} \rvert \}.
    \]

  \item $W$ induces a maximum co-unipolar subgraph of $G$ if and only if $W$ is
    a witness of
    \[
      \max \{ \lvert \mathsf{McU_0} \cup \mathsf{MJI_1} \rvert,
      \lvert \mathsf{McU_1} \cup \mathsf{MJI_0} \rvert \}.
    \]

\end{enumerate}
\end{proposition}

In the next sections, we characterize maximum subgraphs related to polarity
properties in both $P_4$-sparse and $P_4$-extendible graphs, and we use such
characterizations to give linear time algorithms to find the largest subgraphs
with such properties in a given graph of the mentioned graph families.


\subsection{Largest polar subgraph in \texorpdfstring{$P_4$}{P4}-sparse graphs}

We start by introducing a tree representation for $P_4$-sparse graphs which is
the base for our algorithms. Let $G_1 = (V_1, \varnothing)$ and $G_2 = (V_2,
E_2)$ be disjoint graphs such that $V_2 = K \cup R \cup \{s_0\}$, where $K$ is a
clique completely adjacent to $R$, $\lvert K \rvert = \lvert V_1 \rvert +1 \ge
2$ and either $N_{G_2}(s_0) = \{k_0\}$ or $N_{G_2}(s_0) = K \setminus \{k_0\}$
for some vertex $k_0$ in $K$. Let $f$ be a bijection from $V_1$ to $K \setminus
\{k_0\}$. We define $G_1 \spider G_2$ as the graph $G$ with vertex set $V_1 \cup
V_2$ such that $G[V_1] \cong G_1$, $G[V_2] \cong G_2$ and, for each $s\in V_1$,
either $N_G(s) = \{f(s)\}$, provided $N_{G_2}(s_0) = \{k_0\}$, or $N_G(s) = K
\setminus \{f(s)\}$ otherwise.

\begin{proposition}
\cite{jamisonDAM35}
    If $G$ is a graph, then $G$ is a spider if and only if there exist graphs
    $G_1$ and $G_2$ such that $G = G_1 \spider G_2$.
\end{proposition}

By \Cref{thm:connCharSparse}, for any nontrivial $P_4$-sparse graph $G$, either
$G$ is disconnected, or $\overline{G}$ is disconnected, or $G$ is an spider.
Hence, for each $P_4$-sparse graph $G$, a labeled tree $T$ with $G$ as its root
and some subgraphs of $G$ as each node can be constructed in the following way.
Let $H$ be a node of $T$. If $H$ is a trivial graph, it is an unlabeled node in
$T$ with no children. If $H$ is a disconnected graph, it is labeled as a
\textit{$0$-node} and its children are its connected components. If
$\overline{H}$ is disconnected, $H$ is labeled as a \textit{$1$-node} and its
children are the complements of the connected components of $\overline{H}$.
Finally, if $H$ is a spider, let say $H = H_1 \spider H_2$, $H$ is labeled as a
\textit{$2$-node} and its children are $H_1$ and $H_2$. The labeled tree
constructed in this way is called the \textit{ps-tree} of $G$. The ps-tree of a
$P_4$-sparse graph was introduced by Jamison and Olariu in
\cite{jamisonSIAMJC21}, where they proved that such representation can be
computed in linear time.   It follows from the results in \cite{jamisonSIAMJC21}
that the ps-tree of any $P_4$-sparse graph or order $n$ has $O(n)$ nodes.
Particularly, it implies that we can compute the lists of children for each node
of a ps-tree $T$ in linear time and provide each node with such list preserving
the linear space representation for $T$. Additionally, having the lists of
children dor each node of a ps-tree, we can compute in $O(n)$ time the number of
unlabeled children that each node has.  This will be helpful later.   In what
follow, we assume that if $T$ is the ps-tree of $G$, and $x$ is a node of $T$,
then $c_1x, c_2x, \dots$ denote the children of $x$. We will use $G_x$ to
represent the subgraph of $G$ induced by the leaf descendants of $x$ in $T$.

The following proposition implies that, given a ps-tree, we can decide in linear
time whether the graphs associated to its nodes labeled 2 are thin spiders or
thick spiders.

\begin{proposition}
    Let $G = G_1 \spider G_2$ be a spider, and let $T$ be its ps-tree.  Let $w$
    be the only child of $G$ with label $1$ in $T$. If $w$ has two or more
    unlabeled children, then $G$ is a thick spider. Otherwise, $G$ is a thin
    spider.
\end{proposition}

\begin{proof}
  Let $v$ be the only leg of $G$ in $G_2$. Observe that a vertex of $G_2$ is a
  universal vertex if and only if it is adjacent to $v$. Additionally, a vertex
  of $G_2$ is universal if and only if it is an unlabeled child of $w$. Hence,
  if $w$ has two or more unlabeled children, the degree of $v$ in $G$ is at
  least two, so $G$ is a thick spider. Otherwise, if $w$ has precisely one
  unlabeled child, $d_G(v) = 1$ so that $G$ is a thin spider.
\end{proof}

Some of the algorithms we give in this section require us to be able to
recognize the spider partition of any spider from its associated ps-tree.
Nevertheless, this is not always possible, for instance, if we consider any thin
spider whose head complement is disconnected, there will be vertices for which
it is impossible to decide from the associated ps-tree if they belong to the
body or the head of the spider (see \Cref{fig:psTreeP4}).

\begin{figure}[ht!]
\centering
\begin{tikzpicture}

  \begin{scope}[scale=1, yscale=0.5]

  \node [node] (6) at (0,0)[]{$2$};
  \node [vertex] (7) at (-1,-1)[label=270:$s_1$]{};
  \node [node] (8) at (1,-1)[]{$1$};
  \node [vertex] (9) at (0,-2)[label=270:$k_0$]{};
  \node [node] (1) at (2,-2)[]{$0$};
  \node [vertex] (2) at (1,-3)[label=270:$s_0$]{};
  \node [node] (3) at (3,-3)[]{$1$};
  \node [blackV] (4) at (2,-4)[label=270:$k_1$]{};
  \node [blackV] (5) at (3,-4.3)[label=270:$r_1$]{};
  \node [node] (10) at (4,-4)[]{$0$};
  \node [vertex] (11) at (4,-5.3)[label=270:$r_0$]{};
  \node [vertex] (12) at (5,-5)[label=270:$r_2$]{};

  \foreach \from/\to in {6/7,6/8,8/9,8/1,1/2,1/3,3/4,3/5,3/10,10/11,10/12}
    \draw [edge] (\from) to (\to);
  \end{scope}

\end{tikzpicture}

\caption{The ps-tree associated to the thin spider with 2 legs whose head is
isomorphic to $P_3$. The solid vertices are indistinguishable, but one of them
belong to the body of the spider, and the other one belongs to its head.}
\label{fig:psTreeP4}
\end{figure}

However, it is clear that, given a ps-tree $T$, there is a unique $P_4$-sparse
graph (up to isomorphism) associated with $T$, and it results that if we fix a
spider partition for any node labeled 2 in $T$, the graph is completely
determined. Next, we explain how to fix the spider partition for such nodes, and
how to save this data maintaining the linear space needed for storing $T$.

Let $G = G_1 \spider G_2$ be a thin spider, and let $T$ be its associated
ps-tree. Let $V_1, V_2, K, R$, and $s_0$ be like in the definition of $G_1
\spider G_2$, and assume that $N_{G_2}(s_0) = \{k_0\}$. Clearly, the root $r$ of
$T$ is labeled 2, and it  has precisely two children in $T$, namely a child $v$
labeled 1 such that $G_v \cong G_2$, and a child $u$, which is unlabeled if $
\lvert V_1 \rvert =1$, or it is labeled 0 otherwise; we call $v$ the 1-child of
$r$. In addition, since $G$ is a thin spider, $G_2$ can be obtained from $G[R
\cup (K \setminus \{k_0\})]$ by adding first an isolated vertex $s_0$ and then a
universal vertex $k_0$. Thus, $v$ has precisely two children, namely an
unlabeled child ($k_0$) and a 0-labeled child $w$, which we will call the
0-child of $v$. Finally, if $\lvert K \rvert > 2$ or $R \ne \varnothing$, $w$
has exactly two children, one unlabeled ($s_0$) and one child $x$ labeled 1
($G[R \cup (K \setminus \{k_0\})]$) called the 1-child of $w$. Otherwise, if $
\lvert K \rvert = 2$ and $R = \varnothing$ (in which case $G \cong P_4$), $w$
has exactly two unlabeled children, namely $s_0$ and the only vertex $x$ in the
singleton $K \setminus \{k_0\}$ (see \Cref{fig:psTreeThin}).

  \begin{figure}[ht!]
  \centering
  \begin{tikzpicture}

    \begin{scope}[scale=1]

    \node [node] (0) at (0,0)[label=0:$r$]{$2$};

    \node [vertex] (1) at (-1,-1)[label=270:{\footnotesize $G_1$}]{};
    \node [vertex] (u) at (-1,-1)[label=180:{$u$}]{};
    \node [node] (2) at (1,-1)[label=0:$v$]{$1$};

    \node [vertex] (6) at (0,-2)[label=270:$k_0$]{};
    \node [node] (7) at (2,-2)[label=0:$w$]{$0$};

    \node [vertex] (8) at (1,-3)[label=270:$s_0$]{};
    \node [vertex] (9) at (3,-3)[label=270:{\footnotesize $G_2[R\cup (K
    \setminus\{k_0\})]$}]{};
    \node [vertex] (x) at (3,-3)[label=0:$x$]{};

    \foreach \from/\to in {0/1,0/2,2/6,2/7,7/8,7/9}
      \draw [edge] (\from) to (\to);
    \end{scope}

  \end{tikzpicture}
  
  \caption{General structure of the ps-tree of a thin spider.}
  \label{fig:psTreeThin}
  \end{figure}

As we mentioned before, if $\lvert K \rvert =2$ and $R = \varnothing$, then $w$
has precisely two children, $s_0$ and $x$, both of them unlabeled. Notice that
in $G$, precisely one child of $w$ is adjacent to the 0-child of $r$, but we are
not able to distinguish from the ps-tree which child of $w$ is such vertex, so
we must choose arbitrarily some of them to fix a spider partition (which will
completely determine a graph $G'$ isomorphic to $G$, but possibly different from
it, whose ps-tree is $T$ and has the fixed spider partition).  Now, if $R$
induces either a disconnected graph or a spider, then $x$ has precisely $\lvert
K \rvert -1$ unlabeled children, all of them elements of $K$. Nevertheless, if
the complement of $R$ is disconnected, then there are potentially more than $
\lvert K \rvert -1$ unlabeled children of $x$, and they will be
indistinguishable, so we must choose arbitrarily $\lvert K \rvert -1$ of them
to fix a spider partition.

Now, let $G = G_1 \spider G_2$ be a thick spider which is not a thin spider, and
let $T$ be its associated ps-tree. Let $V_1, V_2, K, R$, and $s_0$ be like in
the definition of $G_1 \spider G_2$, and assume that $N_{G_2}(s_0) = K \setminus
\{k_0\}$. As before, the root $r$ of $T$ is labeled 2, and it has a child $v$
labeled 1, and a child $u$ labeled 0. Since $G$ is a thick spider, $G_v$ is the
join of $G[K]-k_0$ with the disjoint union of the graph obtained from $G[R \cup
\{k_0\}]$ by adding an isolated vertex $s_0$.   Thus, $v$ has precisely $\lvert
K \rvert$ children, $\lvert K \rvert -1$ unlabeled children and a 0-labeled
child $w$. Finally, since $\lvert K \rvert \ge 3$ (because $G$ is not a thin
spider), $w$ has exactly two children, one unlabeled ($s_0$) and one child $x$
labeled 1 ($G[R \cup \{k_0\}]$). Similarly to the case of thin spiders, if $R$
induces either a disconnected graph or a spider, then $x$ has precisely one
unlabeled child, $k_0$. Nevertheless, if the complement of $R$ is disconnected,
then there are potentially more than one unlabeled children of $x$, and they
will be indistinguishable, so we must chose arbitrarily one unlabeled child to
fix a spider partition.

As we have seen, to fix the spider partition of a node labeled 2 it is enough to
select some unlabeled descendants of such node which will completely determine
the body of the associated spider, as well as the entire spider partition.
Moreover, we can simply mark the selected vertices for the body of any node
labeled 2 and, since these marked vertices are considered only for the spider
partition of their great great grandfather (or great grandfather) in the
ps-tree, we can save and process the vertices of the bodies of each node labeled
2 in $O(n)$ space and time, in such a way that any time we need a spider
partition of such nodes we use the same fixed partition. It is worth noticing
that we could simultaneously mark the vertices of the spider bodies while
constructing the ps-tree of a $P_4$-sparse graph, avoiding the extra processing
time and ensuring that we can recover with precision the original graph from the
ps-tree.

The following proposition is to thin spiders as \Cref{prop:maxUnion} is to
disconnected graphs. In it, we characterize maximum subgraphs of thin spiders
with some properties related to polarity.

\begin{proposition}
\label{prop:maxThinSp}
Let $G = (S, K, R)$ be a thin spider and let $f : S \to K$ be the bijection
such that $N(s) = \{f(s)\}$ for each $s \in S$. Let $H$ be the subgraph of
$G$ induced by $R$. The following statements hold for any subset $W$ of
$V_G$.
\begin{enumerate}
	\item $W$ is a maximum clique of $G$ if and only if $W$ is a witness of
  \( \max_{s\in S} \{
    \lvert \{s, f(s)\} \rvert,
    \lvert K \cup \mathsf{MC}(H) \rvert \}.
  \)

	\item $W$ is a maximum independent set in $G$ if and only if $W$ is a
  witness of $\max_{s\in S} \{ \lvert \{f(s)\} \cup (S \setminus \{s\}) \rvert,
  \lvert S \cup \mathsf{MI}(H) \rvert \}$.

  \item $W$ induces a maximum bipartite subgraph of $G$ if and only if $W$ is a
  witness of
  \[ \max_{k_1,k_2 \in K} \{
    \lvert S \cup \{k_1, k_2\} \rvert,
    \lvert \mathsf{MI}(H) \cup S \cup \{k_1\} \rvert,
    \lvert \mathsf{MB}(H) \cup S \rvert\}.
  \]

  \item $W$ induces a maximum co-bipartite subgraph of $G$ if and only if $W$ is
  a witness of
  \[ \max_{s_1,s_2\in S} \{
    \lvert \{s_1, s_2, f(s_1), f(s_2)\} \rvert,
    \lvert \mathsf{MC}(H) \cup K \cup \{s_1\} \rvert,
    \lvert \mathsf{McB}(H) \cup K \rvert \}.
  \]

  \item $W$ induces a maximum split subgraph of $G$ if and only if $W$ is a
	witness of \( \max \{ \lvert S \cup K \cup \mathsf{MS}(H) \rvert \}. \)

	\item $W$ induces a maximum cluster in $G$ if and only if $W$ is a witness
	of
    \[
        \max_{\substack{k \in K \\ W' \in \mathcal{X}}} \{ \lvert S \cup \{k\}
        \rvert,
        \lvert S \cup \mathsf{MUC}(H) \rvert,
        \lvert \mathsf{MC}(H) \cup W' \rvert \},
    \]
    where $\mathcal{X}$ is the family of all $\lvert S \rvert$-subsets $W'$ of
    $S \cup K$ such that $\{s, f(s)\} \not \subseteq W'$ for any $s \in S$.

	\item $W$ induces a maximum complete multipartite subgraph of $G$ if and
    only $W$ is a witness of
    \[ \max_{s_1,s_2 \in S} \{
        \lvert \{s_1, f(s_1), f(s_2)\} \rvert,
        \lvert \{f(s_1)\} \cup (S \setminus \{s_1\}) \rvert,
        \lvert \{s_1, f(s_1)\} \cup \mathsf{MI}(H) \rvert,
    \]
    \[
        \lvert S \cup \mathsf{MI}(H) \rvert,
        \lvert K \cup \mathsf{MJI}(H) \rvert
        \}.
    \]

  \item $W$ induces a maximum monopolar subgraph of $G$ if and only	$W$ is a
  witness of
  \[ \max_{k \in K} \{
    \lvert S \cup K \cup \mathsf{MS}(H) \rvert,
    \lvert S \cup \{k\} \cup \mathsf{MUC}(H) \rvert,
    \lvert S \cup \mathsf{MM}(H) \rvert
    \}.
  \]

  \item $W$ induces a maximum co-monopolar subgraph of $G$ if and only	$W$ is a
  witness of
  \[ \max_{s \in S} \{
    \lvert S \cup K \cup \mathsf{MS}(H) \rvert,
    \lvert K \cup \{s\} \cup \mathsf{MJI}(H) \rvert,
    \lvert K \cup \mathsf{McM}(H) \rvert
    \}.
  \]

    \item $W$ induces a maximum polar subgraph of $G$ if and only if $W$ is a
	    witness of \( \max \{ \lvert S \cup K \cup \mathsf{MP}(H) \rvert \}. \)

	\item $W$ induces a maximum unipolar subgraph of $G$ if and only if $W$ is a
	    witness of \( \max \{ \lvert S \cup K \cup \mathsf{MU}(H) \rvert \}. \)

	\item $W$ induces a maximum co-unipolar subgraph of $G$ if and only if $W$
	    is a witness of \( \max \{ \lvert S \cup K \cup \mathsf{McU}(H) \rvert
	    \}. \)
\end{enumerate}
\end{proposition}

\begin{proof}
This proof is similar in flavor to the proof of \Cref{prop:maxUnion}. Moreover,
we consider that property 7 is a good example of the general arguments used for
proving the entire statement.

\begin{enumerate}
  \item Let $W$ be a maximum clique of $G$. If $s \in W \cap S$, then $W \cap S
    = \{s\}$, $W \cap K \subseteq \{f(s)\}$, and $W \cap R = \varnothing$, so in
    this case $W = \{s, f(s)\}$. Otherwise, $W \cap S = \varnothing$, and since
    the union of any clique of $H$ with $K$ is a clique, we have that $W$ is the
    union of $K$ with a maximum clique of $H$.

  \item Let $W$ be a maximum stable set in $G$. If $f(s) \in W \cap K$ for some
    $s\in S$, then $W \cap K = \{f(s)\}$, $s\notin W \cap S$, and $W \cap R =
    \varnothing$, so in this case $W = \{f(s)\} \cup (S \setminus \{s\})$.
    Otherwise, $W \cap K = \varnothing$, and since the union of any independent
    set in $H$ with $S$ is an independent set, we have that $W$ is the union of
    $S$ with a maximum independent set of $H$.

  \item Let $W$ be a set inducing a maximum bipartite subgraph of $G$. If $W
    \cap R$ is a nonempty independent set, then $\lvert W \cap K \rvert \le 1$.
    Furthermore, since the union of an independent subset of $R$ with $S \cup
    \{k\}$ induces a bipartite graph for any $k\in K$, in this case we have that
    $W$ is the union of $S \cup \{k\}$ with a maximum independent set of $H$. If
    $W \cap R$ induces a nonempty bipartite graph, then $W \cap K = \varnothing$
    and $W$ clearly is the union of $S$ with a maximum subset of $R$ inducing a
    bipartite graph. Otherwise, $W \cap R = \varnothing$. Since $W$ induces a
    bipartite graph and $K$ is a clique, we have that $\lvert W \cap K \rvert
    \le 2$. Moreover, since the union of $S$ with any 2-subset of $K$ induces a
    bipartite graph, in this case $W$ is the union of $S$ with a 2-subset of
    $K$.

  \item Let $W$ be a set inducing a maximum co-bipartite subgraph of $G$. If $W
    \cap R$ is a nonempty clique, $\lvert W \cap S \rvert \le 1$. Furthermore,
    since the union of a clique in $H$ with $K \cup \{s\}$ induces co-bipartite
    graph for any $s\in S$, in this case we have that $W$ is the union of $K
    \cup \{s\}$ with a maximum clique of $H$. If $W \cap R$ induces a
    co-bipartite graph which is not a clique, then $W \cap S = \varnothing$ and
    $W$ clearly is the union of $K$ with the vertex set of a maximum
    co-bipartite subgraph of $H$. Else, $W \cap R = \varnothing$. Since $W$
    induces a co-bipartite graph and $S$ is an independent set, we have that
    $\lvert W \cap S\rvert \le 2$. If $W \cap S = \{s_1, s_2\}$, then $W \cap K
    \subseteq \{f(s_1), f(s_2)\}$and it easily follows that $W = \{s_1, s_2,
    f(s_1), f(s_2)\}$. Notice that $W \cap S \ne \varnothing$, because otherwise
    $W \subseteq K$, but $K \cup \{s\}$ induces a co-bipartite graph for any $s
    \in S$, contradicting the election of $W$. Thus, $W \cap S = \{s_1\}$. In
    this case $R = \varnothing$ or, for any $r \in R$, $W \cup \{r\}$ would be a
    subset of $V_G$ inducing a co-bipartite graph, which is impossible by the
    election of $W$. Since $K \cup \{s_1\}$ induces a co-bipartite graph and $R
    = \varnothing$, it follows that $W$ is the union of a maximum clique of $R$
    (which is the empty set) with $K \cup \{s_1\}$.

  \item For any subset $W'$ of $R$ inducing a graph with split partition $(A,
    B)$, the graph $G[S \cup K \cup W']$ has $(A \cup S, K \cup B)$ as a split
    partition. Thus, if $W$ is a set inducing a maximum split subgraph of $G$,
    $W \cap R$ is a maximum split subgraph of $H$, $W \setminus R = S \cup K$,
    and the result follows.

  \item Let $W$ be a set inducing a maximum cluster of $G$.  First, assume that
    $W \cap R = \varnothing$. Since $S \cup \{k\}$ induces a cluster of $G$ for
    any $k \in K$, we have that $\lvert W \rvert \ge \lvert S \rvert +1$, so
    $\{s, f(s)\} \subseteq W$ for some $s\in S$. Moreover, since clusters are
    $P_3$-free graphs, if $\{s_1, f(s_1)\} \subseteq W$, $W \cap K =
    \{f(s_1)\}$. Thus, in this case $W = S \cup \{k\}$ for some $k\in K$.
    Otherwise, if $W \cap R \ne \varnothing$, $W \cap R$ induces a cluster and
    $\{s, f(s)\} \not \subseteq W$ for every $s \in S$, so $\lvert W \setminus R
    \rvert \le \lvert S \rvert$. It follows that, if $W \cap R$ is a clique,
    then $W \setminus R$ is an $\lvert S \rvert$-subset of $K \cup S$ such that
    $\{s, f(s)\} \not\subseteq W\setminus R$ for any $s \in S$, and $W \cap R$
    is a maximum clique of $H$. Otherwise, if $W \cap R$ has at least two
    connected components, then $W \cap K = \varnothing$, $W \setminus R = S$,
    and $W \cap R$ induces a maximum cluster in $H$.

  \item Let $W$ be a set inducing a maximum complete multipartite subgraph of
    $G$. Notice that, for any subset $R'$ of $R$ inducing a complete
    multipartite graph, $G[K \cup R']$ is a complete multipartite graph. In
    consequence, if $W \cap S = \varnothing$, then $W$ is the union of $K$ with
    a maximum subset of $R$ inducing a complete multipartite graph. Also observe
    that, since complete multipartite graphs are $\overline{P_3}$-free graphs,
    either $W \cap S = \varnothing$ or $W \cap R$ is an independent set.

    If $\lvert W \cap K \rvert \ge 3$, then $W \cap S = \varnothing$, so we are
    done. Now, suppose that $W \cap K = \{f(s_1), f(s_2)\}$ for some $s_1, s_2
    \in S$. Observe that in this case $W \cap S$ must be contained in either
    $\{s_1\}$ or $\{s_2\}$. In addition, some of $W \cap S$ or $W \cap R$ must
    be an empty set. As in the former case, if $W \cap S = \varnothing$, $W$ is
    the union of $K$ with a maximum subset of $R$ inducing a complete
    multipartite graph. Otherwise, if $W \cap R = \varnothing$, thus $W = \{s_1,
    f(s_1), f(s_2)\}$ for some $s_1, s_2 \in S$.
    
    Now, suppose that $W \cap K = \{f(s_1)\}$ for some $s_1 \in S$. Notice that
    either $s_1 \notin W$ or $W \cap S \subseteq \{s_1\}$. Also, $W \cap S \ne
    \varnothing$, otherwise $K$ would be a subset of $W$, but $\lvert K \rvert
    \ge 2$ and we are assuming $\lvert W \cap K \rvert = 1$. Thus, if $W \cap S
    \subseteq \{s_1\}$, then $W \cap S = \{s_1\}$ and $W \cap R$ is a maximum
    independent subset of $R$. Else, if $W \cap S \not \subseteq \{s_1\}$, then
    $s_1\notin W$ and there is a vertex $s_2 \in W \cap (S \setminus \{s_1\})$.
    Hence, $W \cap R = \varnothing$ and $W \cap S = S \setminus \{s_1\}$.

    Finally, if $W \cap K = \varnothing$, then $W \cap S \neq \varnothing$, and
    $W$ is the union of $S$ with a maximum independent subset of $R$.

  \item Let $W$ be a set inducing a maximum monopolar subgraph of $G$, and let
    $W' = W \cap R$. If $W'$ induces a graph with split partition $(A,B)$, then
    $G[S \cup K \cup W']$ is a graph with monopolar partition $(A \cup S, B \cup
    K)$. Thus, if $W'$ induces a split graph, $W$ is the union of $S \cup K$
    with a maximum subset of $R$ inducing a split graph.

    Otherwise, if $W'$ induces a cluster which is not a split graph, then $W'$
    has a subset inducing a $2K_2$; from here, since $K_2 \oplus 2K_2$ is not a
    monopolar graph, we have that $\lvert W \cap K \rvert \le 1$, and it follows
    that $W = W' \cup S \cup \{k\}$ for some $k \in K$.

    Finally, if $W'$ induces a monopolar graph which is neither a cluster or a
    split graph, then any monopolar partition $(A, B)$ of $G[W']$ is such that
    $A \ne \varnothing$ and $B$ has at least one pair of nonadjacent vertices;
    it follows that $W \cap K = \varnothing$, so $W$ is the union of $S$ with a
    maximum monopolar subgraph of $H$.

  \item Let $W$ be a set inducing a maximum co-monopolar subgraph of $G$, and
    let $W' = W \cap R$. If $W'$ induces a graph with split partition $(A,B)$,
    then $G[S \cup K \cup W']$ is a graph with co-monopolar partition $(B \cup
    K, A \cup S)$. Thus, if $W'$ induces a split graph, $W$ is the union of $S
    \cup K$ with a maximum subset of $R$ inducing a split graph.

    Otherwise, if $W'$ induces a complete multipartite graph which is not a
    split graph, then $W'$ has a subset inducing a $C_4$; from here, since $2K_1
    + C_4$ is not a co-monopolar graph, we have that $\lvert W \cap S \rvert \le
    1$, and it follows that $W = W' \cup K \cup \{s\}$ for some $s \in S$.

    Finally, if $W'$ induces a co-monopolar graph which is neither a complete
    multipartite graph or a split graph, then any monopolar partition $(A, B)$
    of $\overline{G[W']}$ is such that $A \ne \varnothing$ and $B$ has at least
    one pair adjacent vertices; it follows that $W \cap S = \varnothing$, so $W$
    is the union of $K$ with a maximum co-monopolar subgraph of $H$.

  \item Let $W$ be a set inducing a maximum polar subgraph of $G$. Notice that
    the union of $S \cup K$ with any subset of $R$ inducing a graph with polar
    partition $(A, B)$, is a graph with polar partition $(A \cup S, B \cup K)$.
    Hence, $W$ is the union of $S \cup K$ with a maximum polar subgraph of $H$.

  \item For any subset $R'$ of $R$ inducing a graph with unipolar partition $(A,
    B)$, the graph $G[S \cup K \cup R']$ has unipolar partition $(A \cup K, B
    \cup S)$. Thus, if $W$ is a set inducing a maximum unipolar subgraph of $G$,
    $W = S\cup K \cup R'$, for some subset $R'$ of $R$ inducing a maximum
    unipolar graph.

  \item For any subset $R'$ of $R$ inducing a graph with co-unipolar partition
    $(A, B)$, the graph $G[S \cup K \cup R']$ has co-unipolar partition $(A \cup
    S, B \cup K)$. Thus, if $W$ is a set inducing a maximum co-unipolar subgraph
    of $G$, $W = S\cup K \cup R'$, for some subset $R'$ of $R$ inducing a
    maximum co-unipolar graph.
\end{enumerate}
\end{proof}

In the following propositions we strongly use the fact that a thin spider is the
complement of a thick spider and vice versa. Notice that by a simple
complementary argument, analogous results can be given for computing
$\mathsf{MI}(G_x)$, $\mathsf{McB}(G_x)$, $\mathsf{MJI}(G_x)$,
$\mathsf{McM}(G_x)$, and $\mathsf{McU}(G_x)$.

\begin{proposition}
\label{prop:sparseGxLinear}
Let $G$ be a $P_4$-sparse graph, and let $T$ be its ps-tree. For any node $x$ of
$T$ the following assertions hold true.
\begin{multicols}{2}
  \begin{enumerate}
    \item $\mathsf{MC}(G_x)$ can be found in linear time.
    \item $\mathsf{MB}(G_x)$ can be found in linear time.
    \item $\mathsf{MS}(G_x)$ can be found in linear time.
    \item $\mathsf{MUC}(G_x)$ can be found in linear time.
    \item $\mathsf{MM}(G_x)$ can be found in linear time.
    \item $\mathsf{MP}(G_x)$ can be found in linear time.
    \item $\mathsf{MU}(G_x)$ can be found in linear time.
  \end{enumerate}
\end{multicols}
\end{proposition}

\begin{proof}
  The proofs of all items are similar, and we suggest to read 4 as a reference
  of the general arguments used in the demonstration. We include the proof of
  all items for the sake of completeness.

  \begin{enumerate}
    \item The assertion is trivially satisfied if $x$ is a leaf of $T$. If $x$
        has type 0, we have by part 1 from \Cref{prop:maxUnion} that
        $\mathsf{MC}(G_x)$ is a set realizing $\max_i\{\mathsf{MC}(G_{c_ix})\}$.
        If $x$ has type 1, we have by part 1 from \Cref{prop:maxJoin} that
        $\mathsf{MC}(G_x) = \bigcup_i{\mathsf{MC}(G_{c_ix})}$. Finally, let us
        assume that $x$ has type 2, and let $(S, K, R)$ be the spider partition
        of $G_x$. If $G_x$ is a thin spider, we have from item 1 of
        \Cref{prop:maxThinSp} that $\mathsf{MC}(G_x)$ is a witness of \(
        \max_{s\in S} \{ \lvert \{s, f(s)\} \rvert, \lvert K \cup
        \mathsf{MC}(G[R]) \rvert \}, \) where $f(s)$ is the only neighbor of $s$
        in $K$ for each $s\in S$. Otherwise, if $G_x$ is a thick spider, we have
        from item 2 of \Cref{prop:maxThinSp} that $\mathsf{MC}(G_x)$ is a
        witness of \( \max_{s \in S} \{ \lvert \{s\} \cup (K \setminus \{f(s)\})
        \rvert, \lvert K \cup \mathsf{MC}(G[R]) \rvert \}, \) where,  for each
        $s\in S$, $f(s)$ is the only vertex in $K$ which is not a neighbor of
        $s$. The result follows since $G_x$ has $O(n)$ descendants.

    \item The assertion is trivially satisfied if $x$ is a leaf of $T$. If $x$
        has type 0, we have by part 3 from \Cref{prop:maxUnion} that
        $\mathsf{MB}(G_x) = \bigcup_{i} \mathsf{MB}(G_{c_ix})$. If $x$ has type
        1, we have by part 3 from \Cref{prop:maxJoin} that $\mathsf{MB}(G_x)$ is
        a set realizing $\max_{i,j} \{ \mathsf{MB}(G_{c_ix}),
        \mathsf{MI}(G_{c_ix}) \cup \mathsf{MI}(G_{c_jx}) \}$. Finally, let us
        assume that $x$ has type 2, and let $(S, K, R)$ be the spider partition
        of $G_x$. If $G_x$ is a thin spider, we have from item 3 of
        \Cref{prop:maxThinSp} that $\mathsf{MB}(G_x)$ is a witness of \(
        \max_{k_1,k_2 \in K} \{ \lvert S \cup \{k_1, k_2\} \rvert, \lvert
        \mathsf{MI}(H) \cup S \cup \{k_1\} \rvert, \lvert \mathsf{MB}(G[R]) \cup
        S \rvert \}\). Otherwise, if $G_x$ is a thick spider, we have from item
        4 of \Cref{prop:maxThinSp} that $\mathsf{MB}(G_x)$ is a witness of \(
        \max_{s_1,s_2\in S} \{ \lvert \{f(s_1), f(s_2), s_1, s_2\} \rvert,
        \lvert \mathsf{MI}(G[R]) \cup S \cup \{f(s_1)\} \rvert,$ $\lvert
        \mathsf{MB}(G[R]) \cup S \rvert \}, \) where $f$ is the bijection from
        $S$ to $K$ such that $N(s) = K\setminus\{f(s)\}$ for each $s\in S$. The
        result follows since $G_x$ has $O(n)$ descendants.

    \item The assertion is trivially satisfied if $x$ is a leaf of $T$.  If $x$
        has type 0, we have by part 5 from \Cref{prop:maxUnion} that
        $\mathsf{MS}(G_x)$ is a set realizing $\max_i \{MS(G_{c_ix}) \cup
        \bigcup_{j\ne i} MI(G_{c_jx})\}$. If $x$ has type 1, we have by part 5
        from \Cref{prop:maxJoin} that $\mathsf{MS}(G_x)$ is a set realizing
        $\max_i \{MS(G_{c_ix}) \cup \bigcup_{j\ne i} MC(G_{c_jx})\}$. If $x$ has
        type 2, we have from item 5 of \Cref{prop:maxThinSp} that
        $\mathsf{MS}(G_x)$ is the union of a maximum subset of $R$ inducing a
        split graph with $S\cup K$.  The result follows since $G_x$ has $O(n)$
        descendants.

    \item   The assertion is trivially satisfied if $x$ is a leaf of $T$. If $x$
        has type 0, we have by part 6 of \Cref{prop:maxUnion} that
        $\mathsf{MUC}(G_x)$ is a set realizing $\bigcup_i
        \mathsf{MUC}(G_{c_ix})$. If $x$ has type 1, we have by part 6 of
        \Cref{prop:maxJoin} that $\mathsf{MUC}(G_x)$ is a set realizing $\max_i
        \{\mathsf{MC}(G_x), \mathsf{MUC}(G_{c_ix})\}$. Finally, let us assume
        that $x$ has type 2, and let $(S,K,R)$ be the spider partition of $G_x$.
        If $G_x$ is a thin spider, we have from item 6 of \Cref{prop:maxThinSp}
        that $\mathsf{MUC}(G_x)$ is a witness of \[ \max_{\substack{k \in K \\ X
        \in \mathcal{X}}} \{ \lvert S \cup \{k\} \rvert, \lvert S \cup
        \mathsf{MUC}(G[R]) \rvert, \lvert \mathsf{MC}(G[R]) \cup X \rvert \}, \]
        where $\mathcal{X}$ is the family of all $\lvert S \rvert$-subsets $X$
        of $S \cup K$ such that $\{s, f(s)\} \not \subseteq X$ for any $s \in
        S$, being $f$ as usual. Notice that, when computing the maximum
        described above, we do not need to check each member of $\mathcal{X}$,
        because all of them have the same number of vertices. In addition, if
        such maximum is attained by $\lvert \mathsf{MC}(G[R]) \cup X \rvert$ for
        an $X \in \mathcal{X}$, then any $X \in \mathcal{X}$ can be used to
        construct a maximum cluster of $G_x$, particularly, $\lvert
        \mathsf{MC}(G[R]) \cup K \rvert$ is a maximum cluster of $G_x$.
        
        Finally, if $G_x$ is a thick spider, we have from item 7 of
        \Cref{prop:maxThinSp} that $\mathsf{MUC}(G_x)$ is a witness of
        \[
            \max_{s_1, s_2 \in S} \{ \lvert \{f(s_1), s_1, s_2\} \rvert,
            \lvert \{ s_1 \} \cup (K \setminus \{ f(s_1) \}) \rvert,
            \lvert \{ s_1, f(s_1) \} \cup \mathsf{MC}(G[R]) \rvert,
        \]
        \[
            \lvert K \cup \mathsf{MC}(G[R]) \rvert,
            \lvert S \cup \mathsf{MUC}(G[R]) \rvert \},
        \]
        where $f$ is the bijection from $S$ to $K$ such that $N(s) =
        K\setminus\{f(s)\}$ for each $s\in S$. The result follows since $G_x$
        has $O(n)$ descendants.

    \item   The assertion is trivially satisfied if $x$ is a leaf of $T$. If $x$
        has type 0, we have by part 8 of \Cref{prop:maxUnion} that
        $\mathsf{MM}(G_x)$ is a set realizing $\bigcup_i \mathsf{MM}(G_{c_ix})$.
        If $x$ has type 1, we have by part 8 from \Cref{prop:maxJoin} that
        $\mathsf{MM}(G_x)$ is a set realizing
        $\max_{i,j}\{\mathsf{MM}(G_{c_ix}), \mathsf{MS}(G_{c_ix}) \cup
        \bigcup_{j \ne i}\mathsf{MC}(G_{c_jx}),$ $\mathsf{MI}(G_{c_ix}) \cup
        \bigcup_{j \ne i}\mathsf{MUC}(G_{c_jx})\}$. Finally, let us assume that
        $x$ has type 2, and let $(S,K,R)$ be the spider partition of $G_x$. No
        matter if $G_x$ is a thin or a thick spider, we have from items 8 and 9
        of \Cref{prop:maxThinSp} that $\mathsf{MM}(G_x)$ is a witness of
        \[
            \max_{k \in K} \{ \lvert S \cup K \cup \mathsf{MS}(G[R]) \rvert,
            \lvert S \cup \{k\} \cup \mathsf{MUC}(G[R]) \rvert,
            \lvert S \cup \mathsf{MM}(G[R]) \rvert \}.
        \]
        The result follows since $G_x$ has $O(n)$ descendants.

    \item The assertion is trivially satisfied if $x$ is a leaf of $T$. If $x$
        has type 0, we have by part 10 from \Cref{prop:maxUnion} that
        $\mathsf{MP}(G_x)$ is a set realizing $\max_i \{\mathsf{MM}(G_x),
        \mathsf{MP}(G_{c_ix}) \cup \bigcup_{j\ne i}\mathsf{MUC}(G_{c_jx})\}$. If
        $x$ has type 1, we have by part 10 from \Cref{prop:maxJoin} that
        $\mathsf{MP}(G_x)$ is a set realizing $\max_i \{\mathsf{McM}(G_x),$
        $\mathsf{MP}(G_{c_ix}) \cup \bigcup_{j\ne i}\mathsf{MJI}(G_{c_jx})\}$.
        Finally, let us assume that $x$ has type 2, and let $(S,K,R)$ be the
        spider partition of $G_x$. No matter if $G_x$ is a thin or a thick
        spider, we have from item 10 of \Cref{prop:maxThinSp} that
        $\mathsf{MP}(G_x)$ is the union of $S \cup K$ with a maximum subset of
        $R$ inducing a polar graph.  The result follows since $G_x$ has $O(n)$
        descendants.

    \item The assertion is trivially satisfied if $x$ is a leaf of $T$. If $x$
        has type 0, we have by part 11 from \Cref{prop:maxUnion} that
        $\mathsf{MU}(G_x)$ is a set realizing $\max_i \{\mathsf{MU}(G_{c_ix})
        \cup \bigcup_{j\ne i}\mathsf{MUC}(G_{c_jx})\}$. If $x$ has type 1, we
        have by part 11 from \Cref{prop:maxJoin} that $\mathsf{MU}(G_x)$ is a
        set realizing $\max_i \{\mathsf{McB}(G_x), \mathsf{MU}(G_{c_ix}) \cup
        \bigcup_{j\ne i}\mathsf{MC}(G_{c_jx})\}$. Finally, let us assume that
        $x$ has type 2, and let $(S,K,R)$ be the spider partition of $G_x$. No
        matter if $G_x$ is a thin or a thick spider, we have from items 11 and
        12 of \Cref{prop:maxThinSp} that $\mathsf{MU}(G_x)$ is the union of $S
        \cup K$ with a maximum subset of $R$ inducing a unipolar graph.  The
        result follows since $G_x$ has $O(n)$ descendants.The result follows
        since $G_x$ has $O(n)$ descendants.
  \end{enumerate}
\end{proof}

We obtain the main result of this section as a direct consequence of the
proposition above.

\begin{theorem}
\label{thm:linearSparse}
    For any $P_4$-sparse graph $G$, maximum order subgraphs of $G$ with the
    properties of being monopolar, unipolar, or polar, can be found in linear
    time. In consequence, the problems of deciding whether a $P_4$-sparse graph
    is either a monopolar graph, a unipolar graph, or a polar graph are
    linear-time solvable.
\end{theorem}

\begin{proof}
	From \Cref{prop:sparseGxLinear}, $\mathsf{MM}(G_x), \mathsf{MU}(G_x)$ and
	$\mathsf{MP}(G_x)$ can be found in linear time for any node $x$ of the
	ps-tree associated to a $P_4$-sparse graph. Particularly, it can be done for
	the root of the ps-tree, so the result follows.
\end{proof}


\subsection{Largest polar subgraph in \texorpdfstring{$P_4$}{P4}-extendible
graphs}

Based on \Cref{thm:connChar}, it is possible to represent each $P_4$-extendible
graph $G$ by means of a labeled tree $T$ with root $G$, which can be constructed
in the following way. Let $H$ be a node of $T$. If $H$ is a trivial graph, it is
an unlabeled node of $T$ with no children. If $H$ is a disconnected graph, it is
labeled 0 and its children are its connected components. If $\overline{H}$ is
disconnected, then $H$ is labeled 1 and its children are the components of
$\overline{H}$. If $H$ is an extension graph, it is a node labeled 2 with as
many children as the order of $H$ which has additional information encoding the
graph induced by its children. Finally, if $H$ is an $X$-spider with nonempty
head whose spider partition is $(S, K, R)$, $H$ is a node labeled 3 and has
exactly two children: its left child, $H[S\cup K]$, and its right child, $H[R]$.
We will call the tree constructed in this way the \textit{parse tree} of $G$.
Hochst\"attler and Schindler \cite{hochstattler1995} showed that the problems of
recognizing $P_4$-extendible graphs and computing the parse tree of a
$P_4$-extendible graph can be solved in linear time.\footnote{Actually, the
parse tree defined in \cite{hochstattler1995} is slightly different than the one
we introduce, due to the fact that they assume by convention that the father of
a node labeled 2 is always a node labeled 3, that the root is always a node
labeled 1, and that nodes labeled 1 and 3 may have only one child. Nevertheless,
with some minor changes, the algorithm in \cite{hochstattler1995} can be adapted
to construct our version of the parse tree.}   From the results by
Hochst\"attler and Schindler in \cite{hochstattler1995}, it follows that the
parse tree of a $P_4$-extendible graph of order $n$ has $O(n)$ nodes. It implies
that it takes linear time to compute the lists of children for all nodes of the
parse tree. Since such lists can be considered additional information for each
node, this preserves the condition that the parse tree uses only linear space

It is really easy to provide characterizations for maximal substructures
associated to polarity in extension graphs but, for the sake of brevity, we only
notice that, since there is a finite number of extension graphs, it is possible
to compute the aforementioned maximal substructures in constant time.

We continue with propositions characterizing maximal substructures associated to
polarity in both $P$-spiders and $F$-spiders.  As the reader can notice, the
proofs are very similar in nature to those of \Cref{prop:maxThinSp}.  Notice
that $P_4$-spiders are special cases of thin (and thick) spiders, so a Lemma
analogous to \Cref{lem:Fspider,lem:ovPspider}, for $P_4$-spiders, is simply a
particular case of \Cref{prop:maxThinSp}.

\begin{lemma}
\label{lem:ovPspider}
Let $G = (S, K, R)$ be a $\overline P$-spider, where $S=\{a,a',d\}$, $K=\{b,c\}$
and $\{a,a',b\}$ induces $C_3$. Let $W$ be a subset of $V_G$, and let $H =
G[R]$. The following statements hold.
\begin{enumerate}

	\item $W$ is a maximum clique of $G$ if and only it is a witness of $\max \{
	 \lvert \{a, a', b\} \rvert, \lvert \mathsf{MC}(H) \cup K \rvert \}$.

	\item $W$ is a maximum independent set of $G$ if and only if it is a witness
	of
  \[
    \max \{ \lvert \{a, c\} \rvert,
    \lvert \{a, d\} \rvert,
    \lvert \{a', c\} \rvert,
    \lvert \{a', d\} \rvert,
    \lvert \{b, d\} \rvert,
    \lvert \mathsf{MI}(H) \cup \{a, d\} \rvert,
    \lvert \mathsf{MI}(H) \cup \{a', d\} \rvert \}.
  \]

  \item $W$ induces a maximum bipartite graph if and only if $W$ is a witness of
  \[
				\max \{ \lvert (S \cup K) \setminus \{a\} \rvert, \lvert (S \cup
				K) \setminus \{a' \} \rvert, \lvert \mathsf{MI}(H) \cup S \cup
                (K \setminus \{b\}) \rvert, \lvert \mathsf{MB}(H) \cup S \rvert
                \}.
	\]

  \item $W$ induces a maximum co-bipartite graph if and only if $W$ is a witness
  of \[
    \max \{ \lvert S \cup K \rvert,
    \lvert \mathsf{MC}(H) \cup S \cup (K \setminus\{d\}) \rvert,
    \lvert \mathsf{McB}(H) \cup \{c,b\} \rvert \}.
  \]

  \item $W$ induces a maximum split graph in $G$ if and only it is a witness of
  \begin{multline*}
    \max \{ \lvert \{a,a',b,d\} \rvert,
    \lvert \{a,a',b,c\} \rvert,
    \lvert \{a',b,c,d\} \rvert,
    \lvert \{a,b,c,d\} \rvert,
    \lvert \mathsf{MI}(H) \cup \{a,a',b,d\} \rvert, \\
    \lvert \mathsf{MS}(H) \cup \{a',b,c,d\} \rvert,
    \lvert \mathsf{MS}(H) \cup \{a,b,c,d\} \rvert \}.
  \end{multline*}
  
	\item $W$ induces a maximum cluster in $G$ if and only if it is a witness of
	\[
    \max \{ \lvert \{a,a',b,d\} \rvert,
    \lvert \{a,a',c,d\} \rvert,
    \lvert \mathsf{MC}(H) \cup \{a, a', c\} \rvert,
    \lvert \mathsf{MC}(H) \cup S \rvert,
    \lvert \mathsf{MUC}(H) \cup S \rvert \}.
  \]

	\item $W$ induces a maximum complete multipartite graph in $G$ if and only
	if it is a witness of
  \begin{multline*}
    \max \{ \lvert \{a,a',b\} \rvert,
    \lvert \{a,b,c\} \rvert,
    \lvert \{a',b,c\} \rvert,
    \lvert \{b,c,d\} \rvert,
    \lvert \mathsf{MI}(H) \cup \{a,b\} \rvert,
    \lvert \mathsf{MI}(H) \cup \{a',b\} \rvert, \\
    \lvert \mathsf{MI}(H) \cup \{c,d\} \rvert,
    \lvert \mathsf{MI}(H) \cup \{a,d\} \rvert,
    \lvert \mathsf{MI}(H) \cup \{a',d\} \rvert,
    \lvert \mathsf{MJI}(H) \cup K \rvert \}.
  \end{multline*}

	\item $W$ induces a maximum monopolar graph in $G$ if and only if it is a
  witness of
  \begin{multline*}
    \max \{ \lvert \mathsf{MC}(H) \cup S\cup K \rvert,
    \lvert \mathsf{MUC}(H) \cup \{a,a',c,d\} \rvert,
    \lvert \mathsf{MUC}(H) \cup \{a,a',b,d\} \rvert, \\
    \lvert \mathsf{MS}(H) \cup \{a,b,c,d\} \rvert,
    \lvert \mathsf{MS}(H) \cup \{a',b,c,d\} \rvert,
    \lvert \mathsf{MS}(H) \cup \{a,a',c,d\} \rvert,
    \lvert \mathsf{MM}(H) \cup S \rvert 
    \}.
  \end{multline*}

	\item $W$ induces a maximum co-monopolar graph in $G$ if and only if it is a
  witness of
  \[
    \max \{ \lvert \mathsf{MI}(H) \cup S\cup K \rvert,
    \lvert \mathsf{MS}(H) \cup \{a,b,c,d\} \rvert,
    \lvert \mathsf{MS}(H) \cup \{a',b,c,d\} \rvert,
  \]
  \[
    \lvert \mathsf{MJI}(H) \cup \{a,a',b,c\} \rvert,
    \lvert \mathsf{McM}(H) \cup K \rvert \}.
  \]

	\item $W$ induces a maximum polar graph in $G$ if and only if $W$ is the
	union of a maximum subset of $R$ inducing a polar graph with $S \cup K$.

	\item $W$ induces a maximum unipolar graph in $G$ if and only if $W$ is the
	union of a maximum subset of $R$ inducing a unipolar graph with $S \cup K$.

  \item $W$ induces a maximum co-unipolar graph in $G$ if and only if and only
  if $W$ is a witness of
    \[
      \max \{ \lvert \mathsf{MI}(H) \cup S \cup K \rvert, \lvert
      \mathsf{MB}(H) \cup S \cup \{b\} \rvert, \lvert \mathsf{McU}(H)
      \cup K \cup \{a, d\} \rvert, \lvert \mathsf{McU}(H) \cup K \cup
      \{a', d\} \rvert \}.
    \]

\end{enumerate}
\end{lemma}

\begin{proof}
Since the nature of the proofs is very similar, we consider that reading 8 is
enough to get an idea of the general arguments used throughout the entire proof.
Nonetheless, for the sake of completeness, we include the entire argument.

\begin{enumerate}
  \item Let $W$ be a maximum clique of $G$. If $W \cap R = \varnothing$, then $W
  = \{a, a', b\}$. Otherwise, if $W \cap R \ne \varnothing$, then $W \cap S =
  \varnothing$ and $W$ is the union of $K$ with a maximum clique in $H$.

  \item Let $W$ be a maximum independent set of $G$. If $W \cap R =
  \varnothing$, then $W$ is a maximum independent subset of $S \cup K$, i.e., $W
  \in \{\{a, c\}, \{a, d\}, \{a', c\}, \{a', d\}, \{b, d\}\}$. Otherwise, if $W
  \cap R \ne \varnothing$, then $W \cap K = \varnothing$, and $W$ is the union
  of a maximum independent set in $H$ with a maximum independent subset of $S$.

  \item Let $W$ be a set inducing a maximum bipartite subgraph of $G$. Notice
  that, since $\{a,a',b\}$ induces a triangle, $\lvert W \cap \{a,a',b\} \rvert
  \le 2$. It follows from the previous observation that, if $W \cap R =
  \varnothing$, $W$ is some of $(S \cup K) \setminus \{a\}$, $(S \cup K)
  \setminus \{a' \}$, or $(S \cup K) \setminus \{b\}$. Else, if $W \cap R$ is a
  nonempty independent set, $\lvert W \cap K \rvert \le 1$. Moreover, it is a
  simple observation that the union of any independent subset of $R$ with $S
  \cup (K \setminus \{b\})$ induces a bipartite graph, but the union of an
  independent subset of $R$ with any other 4-subset of $S \cup K$ does not
  induce a bipartite graph. Thus, when $W \cap R$ is a nonempty independent set,
  $W$ is the union of a maximum independent subset of $R$ with $S \cup (K
  \setminus \{b\})$. Finally, if $W \cap R$ induces a nonempty bipartite graph,
  then $W \cap K = \varnothing$ and we trivially have that $W \setminus R = S$.

  \item Let $W$ be a set inducing a maximum co-bipartite subgraph of $G$. Since
  $\overline{P}$ admits a partition in two cliques, if $W \cap R = \varnothing$,
  $W = S \cup K$. Else, if $W \cap R$ induces a nonempty clique, neither
  $\{a,d\}$ or $\{a',d\}$ is a subset of $W$. Moreover, the union of a clique
  contained in $R$ with $(S \cup K) \setminus \{d\}$ induces a co-bipartite
  graph, and the union of a nonempty subset of $R$ with any other 4-subset of $S
  \cup K$ does not induce a co-bipartite graph. Thus, if $W \cap R$ induces a
  nonempty clique, $W$ is the union of a maximum clique in $H$ with $(S \cup K)
  \setminus \{d\}$. Otherwise, $W \cap R$ is a co-bipartite graph which is not a
  clique, and then $W \cap S = \varnothing$, so clearly $W \setminus R = K$; the
  result follows.

  \item Let $W$ be a set inducing a split subgraph of $G$. If $W \cap R =
  \varnothing$, $W$ is a maximum subset of $S \cup K$ inducing a split graph, so
  $W$ is one of $\{a,a',b,d\}$, $\{a,a',b,c\}$, $\{a',b,c,d\}$, or
  $\{a,b,c,d\}$. Now, assume that $W \cap R \ne \varnothing$. Notice that in
  this case $\{a,a',c\} \not\subseteq W$, otherwise $\{a,a',c,r\}$ would induce
  $2K_2$ for any $r\in W \cap R$. Thus, if $W \cap R$ is an independent set, $W
  \setminus R$ is any of $\{a,b,c,d\}, \{a',b,c,d\}$, or $\{a,a',b,d\}$. Else,
  if $W \cap R$ induces a split graph which is not empty, $\{a,a' \}$ could not
  be a subset of $W$, because $\{a,a',r,r' \}$ would induce $2K_2$ for any
  adjacent vertices $r,r' \in W \cap R$. Thus, if $W \cap R$ is not an
  independent set, $W \setminus R$ must be one of $\{a,b,c,d\}$, or
  $\{a',b,c,d\}$, and the result follows.

  \item Let $W$ be a set inducing a maximum cluster of $G$. If $W \cap R =
  \varnothing$, $W$ is a maximum subset of $S \cup K$ inducing a cluster, i.e.,
  $W \in \{ \{a,a',b,d\}, \{a,a',c,d\} \}$. Now, assume that $W \cap R \ne
  \varnothing$. If $W \cap R$ is a clique, then $W$ cannot have simultaneously
  $c$ and $d$, or $b$ and any of $a$ or $a'$. Thus, in this case $W$ is the
  union of a maximum subset of $R$ inducing a clique with one of $\{a, a', c\}$
  or $\{a, a', d\}$. Otherwise, if $W \cap R$ induces a cluster which is not a
  complete graph, then $W \cap K = \varnothing$, and $W$ is the union of $S$
  with a maximum subset of $R$ inducing a cluster.

  \item Let $W$ be a set inducing a maximum complete multipartite subgraph of
  $G$. If $W \cap R = \varnothing$, $W$ is a maximum subset of $S \cup K$
  inducing a complete multipartite graph, i.e., $W$ is one of $\{a,a',b\},
  \{a,b,c\}, \{a',b,c\}$, or $\{b,c,d\}$. Now, assume that $W \cap R \ne
  \varnothing$. Notice that in this case, $W \cap S$ is completely adjacent to
  $W \cap K$. In addition, $W$ cannot have both, $a$ and $a'$. It follows that,
  if $W \cap R$ is an independent set, then $W \setminus R$ is one of $\{a,b\},
  \{a',b\}, \{c,d\}, \{a,d\}, \{a',d\}$, or $K$. Otherwise, if $W \cap R$
  induces a maximum complete multipartite graph of $R$ which is not empty, $W
  \cap S = \varnothing$ and $W \setminus R = K$, so the result follows.

  \item Let $W$ be a set inducing a maximum monopolar subgraph of $G$. If $W
  \cap R$ is a clique, then $\{a,a',c\}\cup (W \cap R)$ induces a cluster and,
  since $\{b,d\}$ is an independent set, we have that $W \setminus R = S \cup
  K$, so $W$ is the union of a maximum clique of $H$ with $S \cup K$.

  When $W \cap R$ induces a noncomplete graph which is simultaneously a split
  graph and a cluster, since $W \cap R$ is not a clique, $\{a,a',b,c\}$ could
  not be a subset of $W$ or, for any nonadjacent vertices $r,r' \in W \cap R$,
  $\{a,a',b,c,r,r' \}$ would induce $K_1 \oplus (K_2+P_3)$, which is not a
  monopolar graph. Moreover, some simple verifications show that $W \setminus R$
  is any of $\{a,b,c,d\}, \{a',b,c,d\}, \{a,a',c,d\}$, or $\{a,a',b,d\}$.

  Else, if $W \cap R$ induces a cluster which is not a split graph, then it has
  a subset $U$ inducing $2K_2$, so $\{b,c\} \not\subseteq W$, or $G[\{b,c\} \cup
  U] \cong K_2 \oplus 2K_2$, which is not a monopolar graph. From here, it is
  easy to verify that $W \setminus R$ is any of $\{a,a',b,d\}$, or
  $\{a,a',c,d\}$.

  Now, assume that $W \cap R$ induces a split graph which is not a cluster.
  Since $K_1 \oplus (K_2 +P_3)$ is not a monopolar graph and $W \cap R$ has a
  subset $W'$ inducing $P_3$, we have that $\{a,a',b\}$ is not a subset of $W$.
  From here, we can easily check that $W \setminus R$ is any of $\{a,b,c,d\}$,
  $\{a',b,c,d\}$, or $\{a,a',c,d\}$.

  Finally, suppose that $W \cap R$ induces a monopolar graph which is neither a
  cluster or a split graph. Suppose that there exists $k \in K \cap W$, and let
  $(A,B)$ be a monopolar partition of $G[W]$. If $k \in A$, then $W \cap R
  \subseteq B$, implying that $W \cap R$ induce a cluster, which is not the
  case. Then, it must be that $k \in B$, but then $W \cap R \cap B$ would be a
  clique and, since $(W \cap R )\setminus B \subseteq A$, we have that $W \cap
  R$ would induce a split graph, but we are assuming it does not. Therefore, $K
  \cap W = \varnothing$, and it follows that $W$ is the union of $S$ with a
  maximum subset of $R$ inducing a monopolar graph.

  \item Let $W$ be a set inducing a maximum co-monopolar subgraph of $G$. If $W
  \cap R$ is an independent set, then $(\{a,a',b\}, \{c,d\} \cup (W \cap R))$ is
  a co-monopolar partition of $G[(W \cap R) \cup S \cup K]$. Hence, if $W \cap
  R$ is an independent set, $W \setminus R = S \cup K$.

  Notice that, if $W \cap R$ is not an independent set, then $S \not\subseteq
  W$, otherwise $W$ would have a subset inducing $K_1 + 2K_2$, which is not a
  co-monopolar graph. In addition, if $W \cap R$ induces a graph with split
  partition $(A,B)$ and $W \setminus R$ is any of $\{a,b,c,d\}$, or
  $\{a',b,c,d\}$, then $W$ induces a graph with co-monopolar partition $(A \cup
  \{a,d\}, B \cup \{b,c\})$ or $(A \cup \{a',d\}, B \cup \{b,c\})$. Also, if $W
  \cap R$ induces a complete multipartite graph and $W \setminus R =
  \{a,a',b,c\}$, then $G[W]$ has the co-monopolar partition $(\{a,a' \}, (W \cap
  R) \cup \{b,c\})$.

  If $W \cap R$ induces a split graph which is not a complete multipartite
  graph, then $\{a,a' \} \not\subseteq W$ or, for any subset $\{r_1,r_2,r_3\}$
  of $W$ inducing $\overline{P_3}$, $\{a,a',r_1,r_2,r_3\}$ would induce $K_1 +
  2K_2$, which is not a co-monopolar graph. In addition, since $W \cap R$ is a
  split graph, $\{a,b,c,d\} \cup (W \cap R)$ and $\{a',b,c,d\} \cup (W \cap R)$
  induce split graphs, and hence co-monopolar graphs, so in this case $W$ is the
  union of a maximum subset of $R$ inducing a split graph with one of
  $\{a,b,c,d\}$ or $\{a',b,c,d\}$.

  Else, if $W \cap R$ induces a complete multipartite graph which is not a split
  graph, then $W$ has a subset $W'$ inducing $C_4$. Therefore, neither $\{a,d\}
  \subseteq W$ or $\{a',d\} \subseteq W$, otherwise $W$ would have a subset
  inducing $C_4+2K_1$, which is not a co-monopolar graph. Moreover, the union of
  any subset of $R$ inducing a complete multipartite graph with $\{a,a',b,c\}$
  induces a co-monopolar graph, so in this case $W$ is precisely the union of
  a maximal subset of $R$ inducing a complete multipartite graph with
  $\{a,a',b,c\}$.

  Finally, assume that $W \cap R$ induces a co-monopolar graph which is neither
  a split graph or a complete multipartite graph. Suppose for a contradiction
  that there exist a vertex $s \in S\cap W$, and let $(A,B)$ be a co-monopolar
  partition of $G[W]$. If $s\in A$, then $W \cap R \subseteq B$, which is
  impossible since $G[W \cap R]$ is not a complete multipartite graph. Then,
  $s\in B$, but in such a case $B\cap W \cap R$ is an independent set, and
  $(W \cap R) \setminus B \subseteq A$, implying that $W \cap R$ induces a split
  graph, contradicting our initial assumption. Hence $S \cap W = \varnothing$.
  Additionally, for any subset $W'$ of $R$ inducing a co-monopolar graph, $W'
  \cup K$ is also a co-monopolar graph, so in this case $W$ is the union of $K$
  with a maximum subset of $R$ inducing a co-monopolar graph.

  \item Let $W$ be a set inducing a maximum polar subgraph of $G$. If $(A, B)$
  is a polar partition of $G[W \cap R]$, then $(A\cup K, B \cup S)$ is a polar
  partition of $G[W]$.

  \item Let $W$ be a set inducing a maximum unipolar subgraph of $G$. If $(A,
  B)$ is a unipolar partition of $G[W \cap R]$, then $(A\cup K, B \cup S)$ is a
  polar partition of $G[W]$.

  \item Let $W$ be a set inducing a maximum co-unipolar subgraph of $G$. Notice
  that, for any independent subset $R'$ of $R$, $(\{a, d\} \cup R', \{a'\} \cup
  K)$ is a co-unipolar partition of $G[S \cup K \cup R']$. Therefore, if $W \cap
  R$ is an independent subset of $R$, we have that $W$ is the union of a maximum
  independent subset of $R$ with $S \cup K$.
  
  Observe that $K_2 + K_3$ is not a co-unipolar graph. Hence, if $W \cap R$ is
  not an independent set, either $\{a, a'\} \not\subseteq W$ or $c \notin W$. It
  easily follows from the previous observation that, if $W \cap R$ induces a
  nonempty graph with bipartition $(A, B)$, then $W$ is the union of a maximum
  subset of $R$ inducing a bipartite graph with some of $K \cup \{a,d\}$, $K
  \cup \{a',d\}$, or $S \cup \{b\}$.

  Now, assume that $W \cap R$ induces co-unipolar graph which is a nonempty
  bipartite graph. We claim that, in such case, $\{a, a'\} \not\subseteq W$, and
  we prove it by means of contradiction. Suppose that $a, a' \in W$, and let
  $(A, B)$ be a co-unipolar partition of $G[W]$. Since $G[W \cap R]$ is not an
  empty graph, $W \cap B \neq \varnothing$, and thus, either $a \in A$ and $a'
  \in B$, or vice versa. However, due to $B \cap \{a, a'\} \neq \varnothing$ we
  have that $W \cap R \cap B$ is an independent set, but then $W\cap R$ induces
  a bipartite graph, reaching a contradiction. From here, it is easy to conclude
  that, in this case, $W$ is the union of a maximum subset of $R$ inducing a
  co-unipolar graph with some of $K \cup \{a, d\}$ or $K \cup \{a', d\}$.
\end{enumerate}
\end{proof}

\begin{lemma}
\label{lem:Fspider}
Let $G = (S, K, R)$ be an $F$-spider, where $S=\{a,a',d\}$, $K=\{b,c\}$ and
$\{a,a',b\}$ induces $P_3$. Let $W$ be a subset of $V_G$, and let $H = G[R]$.
The following statements hold true.
\begin{enumerate}
\item $W$ is a maximum clique of $G$ if and only if $W$ is a witness of
\[
  \max \{ \lvert \{a,b\} \rvert,
  \lvert \{a',b\} \rvert,
  \lvert \{b,c\} \rvert,
  \lvert \{c,d\} \rvert,
  \lvert \mathsf{MC}(H) \cup K \rvert \}.
\]

\item $W$ is a maximum independent set of $G$ if and only if $W$ is a witness of
\[
  \max \{ \lvert \{a,a',c\} \rvert,
  \lvert \mathsf{MI}(H) \cup S \rvert \}.
\]

\item $W$ is a set inducing a maximum bipartite subgraph of $G$ if and only if
$W$ is a witness of
\[
  \max \{ \lvert S \cup K \rvert, \lvert \mathsf{MI}(H) \cup S
  \cup (K \setminus\{b\}) \rvert, \lvert \mathsf{MI}(H) \cup S
  \cup (K \setminus\{c\}) \rvert, \lvert \mathsf{MB}(H) \cup S
  \rvert \}.
\]

\item $W$ is a set inducing a maximum co-bipartite subgraph of $G$ if and only
if $W$ is a witness of
\begin{multline*}
  \max \{ \lvert (S \cup K) \setminus \{a\} \rvert,
  \lvert (S \cup K) \setminus \{a' \} \rvert,
  \lvert \mathsf{MC}(H) \cup K \cup \{a\} \rvert,
  \lvert \mathsf{MC}(H) \cup K \cup \{a' \} \rvert, \\
  \lvert \mathsf{MC}(H) \cup K \cup \{d\} \rvert,
  \lvert \mathsf{McB}(H) \cup K \rvert \}.
\end{multline*}

\item $W$ induces a maximum split graph in $G$ if and only if $W$ is the union
of a maximum subset of $R$ inducing a split graph with $S \cup K$.

\item $W$ induces a maximum cluster of $G$ if and only if $W$ is a witness of
\[
  \max \{ \lvert \{a,a',c,d\} \rvert,
  \lvert \mathsf{MC}(H) \cup \{a,a',c\} \rvert,
  \lvert \mathsf{MUC}(H) \cup S \rvert \}.
\]

\item $W$ induces a maximum complete multipartite graph in $G$ if and only if
$W$ is a witness of
\[
  \max \{ \lvert \{a,a',b,c\} \rvert,
  \lvert \mathsf{MI}(H) \cup S \rvert,
  \lvert \mathsf{MI}(H) \cup \{a,a',b\} \rvert,
  \lvert \mathsf{MIJ}(H) \cup K \rvert \}.
\]

\item $W$ induces a maximum monopolar graph in $G$ if and only if $W$ is a
witness of
\[
  \max \{ \lvert \mathsf{MS}(H) \cup S \cup K \rvert,
  \lvert \mathsf{MUC}(H) \cup \{a,a',b,d\} \rvert,
  \lvert \mathsf{MUC}(H) \cup \{a,a',c,d\} \rvert,
  \lvert \mathsf{MM}(H) \cup S \rvert \}.
\]

\item $W$ induces a maximum co-monopolar graph in $G$ if and only if $W$ is a
witness of
\[
  \max \{ \lvert \mathsf{MS}(H) \cup S\cup K \rvert,
    \lvert \mathsf{MJI}(H) \cup \{a,b,c\} \rvert,
    \lvert \mathsf{MJI}(H) \cup \{a',b,c\} \rvert,
  \]
  \[
    \lvert \mathsf{MJI}(H) \cup \{b,c,d\} \rvert,
    \lvert \mathsf{McM}(H) \cup K \rvert \}.
  \]

\item $W$ induces a maximum polar graph in $G$ if and only if $W$ is the union
of a maximum subset of $R$ inducing a polar graph with $S \cup K$.

\item $W$ induces a maximum unipolar graph in $G$ if and only if $W$ is the
union of a maximum subset of $R$ inducing a unipolar graph with $S \cup K$.

\item $W$ induces a maximum co-unipolar graph in $G$ if and only if $W$ is the
union of a maximum subset of $R$ inducing a co-unipolar graph with $S \cup K$.

\end{enumerate}
\end{lemma}

\begin{proof}
Again, due to the similarities in the proofs, we consider that reading 4 is
enough to have a general idea of the arguments used throughout the whole proof.

\begin{enumerate}
  \item Let $W$ be a maximum clique of $G$. If $R = \varnothing$, $W$ clearly is
    one of $\{a,b\}$, $\{a',b\}$, $\{b,c\}$, or $\{c,d\}$. Otherwise, if $R \neq
    \varnothing$, $R' \cup K$ is a clique, for any clique $R'$ contained in $R$,
    so in this case $W \cap R$ is a nonempty clique. It follows that $W \cap S =
    \varnothing$ and $W$ is the union of $K$ a maximum clique contained in $R$.

  \item Let $W$ be a maximum independent set of $G$. If $R = \varnothing$, $W$
    evidently is one of $\{a,a',c\}$ or $S$. Otherwise, if $R \neq \varnothing$,
    $R' \cup S$ is an independent set, for any independent subset $R'$ of $R$.
    Thus, if $R \neq \varnothing$, $W \cap R$ is a nonempty independent subset
    of $R$, so $W \cap K = \varnothing$. Hence, in this case $W$ is the union of
    $S$ with a maximum independent subset of $R$.

  \item Let $W$ be a set inducing a maximum bipartite subgraph of $G$. If $W
    \cap R = \varnothing$, then clearly $W = S \cup K$. Else, if $W \cap R$ is a
    nonempty independent set, then $\lvert W \cap K \rvert \le 1$. In addition,
    for any independent subset $R'$ of $R$, both $R' \cup S \cup \{b\}$ and $R'
    \cup S \cup \{c\}$ induce bipartite graphs, so in this case $W$ is the union
    of a maximum independent set of $R$ with either $S \cup \{b\}$ or $S \cup
    \{c\}$. Otherwise, $W \cap R$ induces a nonempty bipartite graph and $W \cap
    K = \varnothing$, where it easily follows that $W$ is the union of $S$ with
    a maximum bipartite subgraph of $H$.

  \item Let $W$ be a set inducing a maximum co-bipartite subgraph of $G$. It is
    an easy observation that the only subsets of $S \cup K$ inducing a maximum
    co-bipartite graph are $(S \cup K) \setminus \{a\}$ and $(S \cup K)
    \setminus \{ a' \}$; hence, if $W \cap R = \varnothing$, $W$ must be one of
    these sets. Notice that if $W \cap R \ne \varnothing$ then $\lvert W \cap S
    \rvert \le 1$. From here, it is easy to observe that if $W \cap R$ is a
    nonempty clique, then $W \setminus R$ is one of $K \cup \{a\}$, $K \cup \{a'
    \}$, or $K \cup \{d\}$, so in this case $W$ is the union of one of these
    sets with a maximum clique of $H$. Finally, if $W \cap R$ induces a
    co-bipartite graph which is not a clique, then $W \cap S = \varnothing$ and
    $W$ clearly is the union of $K$ with a maximum set inducing a co-bipartite
    subgraph of $H$.

  \item Let $W$ be a set inducing a maximum split subgraph of $G$. Just notice
    that, for any subset $R'$ of $R$ inducing a graph with split partition $(A,
    B)$, $(A \cup S, B \cup K)$ is a split partition of $G[S\cup K \cup R']$.

  \item Let $W$ be a set inducing a maximum cluster of $G$. If $R =
    \varnothing$, then $W = \{a,a',c,d\}$. Otherwise, the union of $S$ with any
    subset of $R$ inducing a cluster is also a cluster. Thus, we may assume that
    $ \lvert W \setminus R \rvert \ge 3$. Moreover, if $W \neq \{a,a',c,d\}$,
    then $W \cap R \neq \varnothing$ and then, none of $\{a,b\}$, $\{a',b\}$, or
    $\{c,d\}$, is a subset of $W$, or $W$ would have a subset inducing $P_3$.
    From here, it is a easy to conclude that, if $W \cap R$ is a clique, then $W
    \setminus R \in \{ S, \{a,a',c\} \}$, while, if $W \cap R$ induces a cluster
    which is not complete graph, then $W \setminus R = S$.

  \item Let $W$ be a set inducing a maximum complete multipartite subgraph of
    $G$. If $W \cap R = \varnothing$, then $W$ is a maximum subset of $S \cup K$
    inducing a complete multipartite graph, so $W = \{a,a',b,c\}$. Otherwise, $W
    \cap R \neq \varnothing$, and since $G[W]$ is $\overline{P_3}$-free, none of
    $\{c,a\}$, $\{c,a' \}$, or $\{b,d\}$, could be a subset of $W$. It follows
    that, in this case, $\lvert W \setminus R \rvert \le 3$. Notice that the
    union $S$ with any independent subset of $R$ is an independent set, so it
    induces a complete multipartite graph. Hence, if $W \cap R$ is an
    independent set, $\lvert W \setminus R \rvert = 3$ and a simple verification
    yields that $W \setminus R$ can be any of $S$ or $\{a,a',b\}$. Finally, if
    $W \cap R$ induces a complete multipartite graph which is not an empty
    graph, then $W \cap S = \varnothing$, and $W \setminus R = K$.

  \item Let $W$ be a set inducing a maximum monopolar subgraph of $G$. If $W
    \cap R$ induces a graph with split partition $(A,B)$, then $(A \cup S, B
    \cup K)$ is a split partition of $G[S \cup K \cup (W \cap R)]$. Else, if $W
    \cap R$ induces a cluster which is not a split graph, then $W \cap R$ has a
    subset inducing $2K_2$, so $K \not\subseteq W$, because $K_2 \oplus 2K_2$ is
    not a monopolar graph. In addition, it is easy to corroborate that for any
    subset $R'$ of $R$ inducing a cluster, $R \cup \{a,a',b,d\}$ and $R'
    \cup\{a,a',c,d\}$ induce monopolar graphs. Finally, assume that $W \cap R$
    induces a monopolar graph which is neither a split graph or a cluster.
    Suppose for a contradiction that there exists a vertex $k \in K \cap W$, and
    let $(A,B)$ be a monopolar partition of $G[W]$. If $k \in A$, then $W \cap R
    \subseteq B$, so $W \cap R$ induces a cluster, but we are assuming this is
    not the case. Thus, $k \in B$, but then, $B \cap W \cap R$ is a clique, and
    $(W \cap R) \setminus B \subseteq A$, so $W \cap R$ induces a split graph,
    which is impossible. Therefore, $K \cap W = \varnothing$. Moreover, if $R'$
    is a subset of $R$ inducing a graph with monopolar partition $(A,B)$, then
    $(A \cup S, B)$ is a monopolar partition of $G[R' \cup S]$, where the result
    follows.

  \item Let $W$ be a set inducing a maximum co-monopolar subgraph of $G$. If a
    subset $R'$ of $R$ induces a graph with split partition $(A, B)$, then $(B
    \cup K, A \cup S)$ is a co-monopolar partition of $G[R' \cup S \cup K]$.
    Thus, if $W \cap R$ induces a split graph, then $W \setminus R = S \cup K$.

    Now, if $W \cap R$ induces a complete multipartite graph which is not a
    split graph, there exists a subset $W'$ of $W \cap R$ inducing a 4-cycle.
    Hence, since $C_4+2K_1$ is not a co-monopolar graph, $\lvert W \cap S \rvert
    \le 1$. Moreover, for any subset $R'$ of $R$ inducing a complete
    multipartite graph and any $s\in S$, $(\{s\}, R' \cup K)$ is a co-monopolar
    partition of $G[R' \cup K \cup \{s\}]$. Thus, if $W \cap R$ induces a
    complete multipartite graph which is not a split graph, then $W \setminus R$
    is one of $\{a,b,c\}$, $\{a',b,c\}$, $\{b,c,d\}$.

    Finally, assume that $W \cap R$ induces a co-monopolar graph which is
    neither a complete multipartite graph or a split graph. Suppose for a
    contradiction that there exists a vertex $s \in S \cap W$, and let $(A, B)$
    be a co-monopolar partition of $G[W]$. If $s \in A$, then $W \cap R \cap A =
    \varnothing$, so $W \cap R$ must induce a complete multipartite graph, which
    is not the case. Thus, $s \in B$, so $B \cap W \cap R$ is an independent
    set, because complete multipartite graphs are $\overline{P_3}$-free graphs.
    But then, $W \cap R$ induces a split graph, which is impossible. Therefore
    $W \cap S = \varnothing$. In addition, if $R'$ is any subset of $R$ inducing
    a graph with co-monopolar partition $(A,B)$, then $(A\cup K, B)$ is a
    co-monopolar partition of $G[R' \cup K]$. Hence, if $W \cap R$ induces a
    co-monopolar graph which is neither a split graph or a complete multipartite
    graph, then $W \setminus R = K$.

  \item Let $W$ be a set inducing a maximum polar subgraph of $G$. The result
    follows since, for any subset $R'$ of $R$ inducing a graph with polar
    partition $(A, B)$, $(A \cup K, B \cup S)$ is a polar partition of $G[S\cup
    K \cup R']$.

  \item Let $W$ be a set inducing a maximum unipolar subgraph of $G$. It is
    enough to notice that, for any subset $R'$ of $R$ inducing a graph with
    unipolar partition $(A, B)$, $(A \cup K, B \cup S)$ is a unipolar partition
    of $G[S\cup K \cup R']$.

  \item Let $W$ be a set inducing a maximum co-unipolar subgraph of $G$. The
    result follows since, for any subset $R'$ of $R$ inducing a graph with
    co-unipolar partition $(A, B)$, we have that $(A \cup S, B \cup K)$ is a
    co-unipolar partition of $G[S \cup K \cup R']$.
  \end{enumerate}
\end{proof}

For the proof of the next proposition we strongly use, without explicit mention,
that the complements of $P$-spiders and the complements of $F$-spiders are,
respectively, $\overline{P}$-spiders and $\overline{F}$-spiders. Notice that by
a simple complementary argument, analogous results can be given for computing
$\mathsf{MI}(G_x)$, $\mathsf{McB}(G_x)$, $\mathsf{MJI}(G_x)$,
$\mathsf{McM}(G_x)$, and $\mathsf{McU}(G_x)$.

\begin{proposition}
\label{prop:extGxLinear}
    Let $G$ be a $P_4$-extendible graph, and let $T$ be its associated parse
    tree. For any node $x$ of $T$ the followings assertions are satisfied.
    \begin{multicols}{2}
    \begin{enumerate}
    \item $\mathsf{MC}(G_x)$ can be computed in linear time.
    \item $\mathsf{MB}(G_x)$ can be computed in linear time.
    \item $\mathsf{MS}(G_x)$ can be computed in linear time.
    \item $\mathsf{MUC}(G_x)$ can be computed in linear time.
    \item $\mathsf{MM}(G_x)$ can be computed in linear time.
    \item $\mathsf{MP}(G_x)$ can be computed in linear time.
    \item $\mathsf{MU}(G_x)$ can be computed in linear time.
    \end{enumerate}
    \end{multicols}
\end{proposition}

\begin{proof}
  The assertions trivially hold whenever $x$ is a leaf of $T$. Also, if $x$ is a
  node labeled 0 or 1, the proof follows exactly as in
  \Cref{prop:sparseGxLinear}. Thus, we will assume for the rest of the proof
  that $x$ has label either 2 or 3. Even in these cases the proof is similar in
  flavor to \Cref{prop:sparseGxLinear}, but we use
  \Cref{lem:ovPspider,lem:Fspider} besides \Cref{prop:maxThinSp}. Hence, we only
  write the proof for item 6.

  If $x$ is a node labeled 2, it is not hard to verify that $\mathsf{MP}(G_x) =
  G_x$. Otherwise, $x$ is a node labeled 3, so $G_x$ is an $X$-spider. By
  \Cref{prop:maxThinSp,lem:ovPspider,lem:Fspider}, if $G_x$ is a graph with
  $X$-spider partition $(S, K, R)$, then $\mathsf{MP}(G_x)$ is the union of
  $S\cup K$ with a maximum subset of $R$ inducing a polar graph. The result
  follows since $G_x$ has $O(n)$ descendants.
\end{proof}

The main results of this section are summarized in the next theorem, which is a
direct consequence of the proposition above.

\begin{theorem}
\label{thm:linearExtendible}
    For any $P_4$-extendible graph $G$, maximum order subgraphs of $G$ with the
    properties of being monopolar, unipolar, or polar, can be found in linear
    time. In consequence, the problems of deciding whether a $P_4$-extendible
    graph is either a monopolar graph, a unipolar graph, or a polar graph are
    linear-time solvable.
\end{theorem}

\begin{proof}
    From \Cref{prop:extGxLinear}, $\mathsf{MM}(G_x), \mathsf{MU}(G_x)$ and
    $\mathsf{MP}(G_x)$ can be found in linear time for any node $x$ of the parse
    tree associated to a $P_4$-extendible graph. Particularly, it can be done
    for the root of the parse tree, so the result follows.
\end{proof}

\section{Conclusions}
\label{sec:conc}

This work must be considered a sequel and a complement of
\cite{contrerasCogGen1}, where, among other things, some properties related to
polarity on $P_4$-sparse and $P_4$-extendible graphs were characterized by
finite families of forbidden induced subgraphs. Specifically, the families of
minimal $(s,1)$-polar obstructions for any nonnegative integer $s$, as well as
the families of minimal monopolar, unipolar, and polar obstructions, when
restricted to the mentioned graph classes, were exhibited in the aforementioned
paper. It is worth noticing that, from such characterizations, it directly
follows that there exist brute force algorithms of polynomial-time complexity
for deciding whether a $P_4$-sparse or a $P_4$-extendible graph is monopolar,
unipolar, or polar.

The results in this work are divided in two parts. First, we adapt the
techniques used in \cite{hellDAM261} to generalize the characterization of
cograph minimal $2$-polar obstructions given in that paper, by explicitly
exhibiting complete lists of minimal $2$-polar obstructions when restricted to
either $P_4$-sparse or $P_4$-extendible graphs. The following proposition
summarize our main result on this topic.

\begin{theorem}
    Let $\mathcal{G}$ be any subfamily of either $P_4$-sparse or
    $P_4$-extendible graphs which is both, hereditary and closed under
    complements. Let $\mathcal{F}$ be the family of graphs depicted in
    \Cref{fig:basic2obsmin}. A graph $G$ in $\mathcal{G}$ is a minimal $2$-polar
    obstruction if and only $G$ can be obtained from some graph in $\mathcal{G}
    \cap \mathcal{F}$ by a finite sequence of partial complementations.
\end{theorem}

\begin{figure}[ht!]
\centering
\begin{tikzpicture}

\begin{scope}[scale=0.8]

\begin{scope}[scale=0.73]
\foreach \i in {0,1}
	\foreach \j in {0,1,2}
		\node [vertex] (\i\j) at (\i*1.5,\j)[]{};
\node [vertex] (7) at (0.75,3)[]{};

\foreach \j in {0,1,2}
	\draw [edge]  (0\j) to (1\j);
\node [rectangle] (n) at (0.75,-0.7){$F_1$};
\end{scope}

\begin{scope}[xshift=3.5cm, scale=0.73]
\foreach \i in {0,1,2}
	\foreach \j in {2,3}
		\node [vertex] (\i\j) at (\i,\j-1.5)[]{};
\foreach \i in {0,1}
	\node [vertex] (\i) at (\i+0.5,2.5)[]{};

\foreach \i/\j in {0/1,03/13,13/23,02/12,12/22}
	\draw [edge]  (\i) to (\j);

\node [rectangle] (n) at (1,-0.7){$F_{6}$};
\end{scope}

\begin{scope}[xshift=7cm, yshift=0.53cm, scale=0.55]
\foreach \i in {0,2}
	\foreach \j in {0,1,2,4}
		\node [vertex] (\i\j) at (\i,3-\j)[]{};

\foreach \i/\j in {04/24,24/22,22/02,02/04,01/21}
	\draw [edge]  (\i) to (\j);

\node [rectangle] (n) at (1,-1.9){$F_{13}$};
\end{scope}

\begin{scope}[xshift=10.5cm, scale=0.73]
\foreach \i in {0,1}
	\foreach \j in {1,2,3}
		\node [vertex] (\i\j) at (\i*2,\j-1)[]{};
\foreach \i in {0,1,2}
	\node [vertex] (\i) at (\i,3)[]{};

\foreach \i in {01,02,03}
	\foreach \j in {11,12,13}
	\draw [edge]  (\i) to (\j);

\node [rectangle] (n) at (1,-0.7){$F_{21}$};
\end{scope}

\begin{scope}[xshift=0.5cm, yshift=-3cm, scale=0.85]
\foreach \i in {0,...,4}
	\node [vertex] (\i) at ($ (0,1) + ({90+(\i*72)}:0.9) $)[]{};
\foreach \i in {7,9,12}
	\node [vertex] (\i) at ($ (0,1) + ({90+(\i*18)}:1.9) $)[]{};

\foreach \i in {2,3}
	\foreach \j in {0,1,4}
		\draw [edge]  (\i) to (\j);

\foreach \i/\j in {2/7,2/9,7/9,3/12,2/3,4/0,0/1}
	\draw [edge]  (\i) to (\j);
\node [rectangle] (n) at (0,-1.5){$F_{26}$};
\end{scope}

\begin{scope}[xshift=4.25cm, yshift=-3.5cm, scale=0.85]
\foreach \i in {0,...,4}
	\node [vertex] (\i) at ({(\i-2)*0.75},2)[]{};
\node [vertex] (5) at (0,1)[]{};
\node [vertex] (6) at (0,0)[]{};
\node [vertex] (7) at (-1,0)[]{};

\foreach \i in {0,...,3}
	\draw let \n1={int(\i+1)} in [edge] (\i) to node [above] {} (\n1);

\draw [-, shorten <=1pt, shorten >=1pt, >=stealth, line width=.7pt, bend left=40]  (0) to (4);

\foreach \i in {0,...,4}
	\draw [edge]  (\i) to (5);

\draw [edge]  (6) to (5);

\node [rectangle] (n) at (0,-0.9){$F_{27}$};
\end{scope}

\begin{scope}[xshift=7.1cm, yshift=-3.5cm, scale=0.82, xscale=1.2]
\foreach \i in {0,1}
	\foreach \j in {0,...,2}
		\node [vertex] (\i\j) at (\i,\j)[]{};
\node [vertex] (03) at (0,2.6)[]{};
\node [vertex] (13) at (1,2.6)[]{};

\foreach \i/\j in {03/13,02/12,01/11,00/01,01/02,11/12}
	\draw [edge] (\i) to node [above] {} (\j);

\node [rectangle] (n) at (0.5,-0.9){$F_{32}$}; 
\end{scope}

\begin{scope}[xshift=11.25cm, yshift=-3.75cm, scale=0.8]
\foreach \i in {0,...,4}
	\node [vertex] (\i) at ($ (0,2) + ({90+(\i*72)}:1) $)[]{};
\foreach \i in {0,...,2}
	\node [vertex] (\i0) at (\i-1,0.3)[]{};

\foreach \i in {0,...,4}
	\draw let \n1={int(mod(\i+1,5))} in [edge] (\i) to node [above] {} (\n1);
\foreach \i/\j in {00/10,1/4}
	\draw [edge] (\i) to node [above] {} (\j);

\node [rectangle] (n) at (0,-0.65){$F_{37}$};
\end{scope}

\end{scope}

\end{tikzpicture}

\caption{Some minimal $2$-polar obstructions.}
\label{fig:basic2obsmin}
\end{figure}
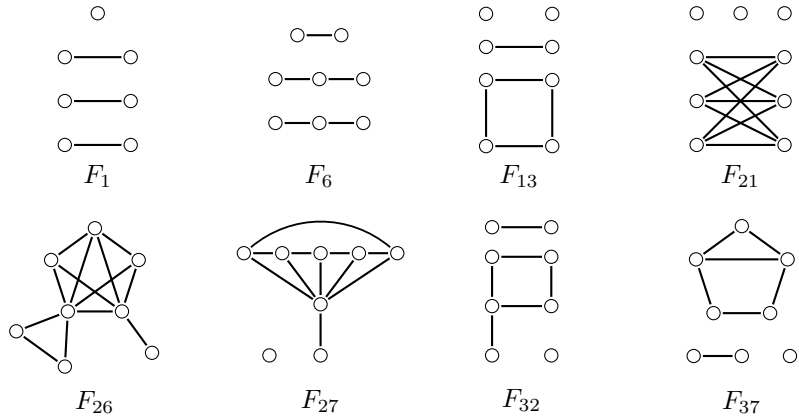

For the second part, based on unique tree representations for $P_4$-sparse and
$P_4$-reducible graphs, we present linear time algorithms for finding largest
subgraphs with properties related to polarity in any graph of such families (see
\Cref{thm:linearSparse,thm:linearExtendible}). These results generalize the one
given by Ekim, Mahadev and de Werra \cite{ekimDAM156} for finding the largest
polar subgraph in cographs based on their cotree.

Our algorithms can be easily adapted to give back yes-certificates, so we wonder
whether it can be adapted, preserving its time-complexity, to also return
no-certificates.

\begin{problem}
Can we adapt our algorithms to make them linear-time certifying algorithms?
\end{problem}

We also think it is possible to use an approach similar to the one used for
proving \Cref{thm:linearSparse,thm:linearExtendible}, to extend such results to
wider classes of graphs having a simple enough tree representation.
Specifically, we pose the next problem.

\begin{problem}
Can we give a linear time algorithm to find maximum monopolar, maximum unipolar,
and maximum polar subgraphs on $P_4$-tidy or extended $P_4$-laden graphs?
\end{problem}

In the context of matrix partitions, it was shown by Feder, Hell and Xie in
\cite{federENDM28} that, for any pair of fixed nonnegative integers, $s$ and
$k$, there is only a finite number of minimal $(s, k)$-polar obstructions, so
that theoretically there is a polynomial-time brute force algorithm to decide
whether a given graph is an $(s,k)$-polar graph. Moreover, Feder, Hell, Klein
and Motwani present in \cite{federSIAMJDM16} an explicit polynomial-time
algorithm for solving the problem of deciding whether an input graph admits a
fixed sparse-dense partition. Particularly, since both, complete $s$-partite
graphs and $k$-clusters can be recognized in quadratic time, we have that
$(s,k)$-polar graphs can be recognized in $O( \lvert V \rvert ^{4 + 2 \max
\{s,k\} })$-time. The aforementioned results make us wonder if it is possible to
improve the time complexity of such algorithms by restricting the input graph to
some of the graph classes with relatively few induced paths on four vertices.

\begin{problem}
Given arbitrary fixed nonnegative integers $s$ and $k$, can we a give
linear-time algorithm for finding a maximum order $(s,k)$-polar subgraph of a
cograph $G$?
\end{problem}

We also propose to solve the next natural problem which is closely related to
the previous question.

\begin{problem}
Give an efficient algorithm for computing the minimum value of $z = s + k$ such
that an input cograph $G$ is an $(s,k)$-polar graph.
\end{problem}

Finally, we think that an approach similar to the one used here can be helpful
to find the complete family of minimal $2$-polar obstructions for general
graphs, so we pose such problem as a future line of work.


\begin{thebibliography}{}

\bibitem{bondy2008}
	J.~A.~Bondy and U.S.R.~Murty,
	\textit{Graph Theory},
	Springer, Berlin, 2008.

\bibitem{bretscherSIAMJDM22}
    A.~Bretscher, D.~Corneil, M.~Habib and C.~Paul,
    A simple linear time LexBFS cograph recognition algorithms,
    \textit{SIAM Journal on Discrete Mathematics} 22(4) (2008) 1277--1296.

\bibitem{contrerasCogGen1}
	F.~E.~Contreras-Mendoza and C.~Hern\'andez-Cruz,
	Minimal obstructions for polarity, monopolarity, unipolarity and
    $(s,1)$-polarity in generalizations of cographs,
	\textit{arXiv e-prints}, 2022, arXiv: 2203.04953.

\bibitem{corneilDAM3}
	D.G.~Corneil, H.~Lerchs and L.~Stewart~Burlingham,
	Complement reducible graphs,
	\textit{Discrete Applied Mathematics} 3 (3) (1981) 163--174.

\bibitem{corneilSIAMJC14}
	D.G.~Corneil, Y.~Perl and L.K.~Stewart~Burlingham,
	A linear recognition algorithm for cographs,
	\textit{SIAM Journal on Computing} 14(4) (1985), 926--934.

\bibitem{ekimDAM156}
	T.Ekim, N.~V.~R.~Mahadev and D.~de~Werra,
	Polar cographs,
	\textit{Discrete Applied Mathematics} 156 (2008) 1652--1660.

\bibitem{federSIAMJDM16}
  T.~Feder, P.~Hell, S.~Klein and R.~Motwani,
  List partitions,
  \textit{SIAM Journal of Discrete Mathematics} 16(3) (2003), 449-478.

\bibitem{federENDM28}
	T.~Feder, P.~Hell and W.~Xie,
	Partitions with finitely many minimal obstructions,
	\textit{Electronic Notes in Discrete Mathematics} 28 (2007) 371--378.

\bibitem{foldesGTC1977}
	S.~Foldes and P.~L.~Hammer,
	Split graphs,
	\textit{in: Proc. 8th Southeastern Conf. on Combinatorics, Graph Theory, and
	Computing}, 1977, 311--315.

\bibitem{hannebauerTH'10}
	C.~Hannebauer,
	\textit{Matrix colorings of $P_4$-sparse graphs},
	Master's thesis (2010), FernUniversit\"at in Hagen.

\bibitem{hellDAM261}
	P.~Hell, C.~Hern\'andez-Cruz and C.~Linhares-Sales,
	Minimal obstructions to $2$-polar cographs,
	\textit{Discrete Applied Mathematics} 261 (2019) 219--228.

\bibitem{hochstattler1995}
	W.~Hochst\"attler and H.~Schindler,
	\textit{Recognizing $P_4$-extendible graphs in linear time},
	Technical Report No. 95.188, Universit\"at zu K\"oln (1995).

\bibitem{jamisonDAM34}
	B.~Jamison and S.~Olariu,
	On a unique tree representation for $P_4$-extendible graphs,
	\textit{Discrete Applied Mathematics} 34 (1991) 151--164.

\bibitem{jamisonDAM35}
	B.~Jamison and S.~Olariu,
	A tree representation for $P_4$-sparse graphs,
	\textit{Discrete Applied Mathematics} 35 (1992) 115--129.

\bibitem{jamisonSIAMJC21}
  B.~Jamison and S.~Olariu,
  Recognizing $P_4$-sparse graphs in linear time,
  \textit{SIAM Journal on Computing} 21 (2) (1992) 381--406.

\end{thebibliography}
\end{document}